\documentclass[a4paper,11pt]{amsart}
\usepackage{amsmath, amssymb, amsthm,enumerate}
\usepackage{verbatim}
\usepackage{hyperref}
\usepackage{xypic}

\usepackage{color}
\definecolor{black}{RGB}{0,0,0}
\definecolor{a}{RGB}{0,150,0}
\definecolor{b}{RGB}{150,0,0}
\definecolor{c}{RGB}{0,0,150}
\definecolor{d}{RGB}{180,80,0}

\newcounter{theorem}
\newtheorem{thm}[theorem]{Theorem}
\newtheorem{lemma}[theorem]{Lemma}
\newtheorem{prop}[theorem]{Proposition}
\newtheorem{cor}[theorem]{Corollary}

\newtheorem{conjecture}[theorem]{Conjecture}
\newtheorem{claim}[theorem]{Claim}
\newtheorem{property}[theorem]{Property}

\theoremstyle{remark}
\newtheorem*{remark*}{Remark}
\newtheorem{remark}[theorem]{Remark}

\theoremstyle{definition}
\newtheorem{defn}[theorem]{Definition}
\newtheorem{notation}[theorem]{Notation}

\numberwithin{equation}{section}
\numberwithin{theorem}{section}

\renewcommand{\setminus}{\backslash}
\renewcommand{\emptyset}{\varnothing}

\newcommand{\id}{\mathrm{id}}
\newcommand{\Q}{\mathcal Q}
\newcommand{\leb}{\mathrm{Leb}}

\title[Quasidiagonality of nuclear $\mathrm{C}^{*}$-algebras]{Quasidiagonality of \\ nuclear C$^{*}$-algebras}

\begin{document}

\author[A.\ Tikuisis]{Aaron Tikuisis}
\address{\hskip-\parindent Aaron Tikuisis, Institute of Mathematics, School of Natural and Computing Sciences, University of Aberdeen, AB24 3UE, Scotland.}
\email{a.tikuisis@abdn.ac.uk}
\author[S.\ White]{Stuart White}
\address{\hskip-\parindent Stuart White, School of Mathematics and Statistics, University of Glasgow, Glasgow, G12 8QW, Scotland and Mathematisches Institut der WWU M\"unster, Einsteinstra\ss{}e 62, 48149 M\"unster, Deutschland.}
\email{stuart.white@glasgow.ac.uk}
\author[W.\ Winter]{Wilhelm Winter}
\address{\hskip-\parindent
Wilhelm Winter, Mathematisches Institut der WWU M\"unster, Einsteinstra\ss{}e 62, 48149 M\"unster, Deutschland.}
\email{wwinter@uni-muenster.de}
\thanks{Research partially supported by EPSRC (EP/N002377), NSERC (PDF, held by AT), by an Alexander von Humboldt foundation fellowship (held by SW) and by the DFG (SFB 878).}

\date{\today}

\maketitle

\begin{abstract}
We prove that faithful traces on separable and nuclear $\mathrm{C}^*$-algebras in the UCT class are quasidiagonal.  This has a number of consequences.  Firstly, by results of many hands, the classification of unital, separable, simple and nuclear $\mathrm{C}^*$-algebras of finite nuclear dimension which satisfy the UCT is now complete. Secondly, our result links the finite to the general version of the Toms--Winter conjecture in the expected way and hence clarifies the relation between decomposition rank and nuclear dimension. Finally, we confirm the Rosenberg conjecture: discrete, amenable groups have quasidiagonal $\mathrm{C}^*$-algebras. 
\end{abstract}

\renewcommand*{\thetheorem}{\Alph{theorem}}
\section*{Introduction}

\noindent Quasidiagonality was first introduced by Halmos for sets of operators on Hilbert space; see \cite[Section 4]{Hal:BAMS}. An abstract $\mathrm{C}^{*}$-algebra is called quasidiagonal if it has a faithful representation that is quasidiagonal, i.e., for which there is an approximately central net of finite rank projections converging  strongly to the unit (see \cite{Thayer,V:IEOT}, or \cite{Bro:survey}, or \cite[Chapter 7]{BrOz}). In \cite{Voi:Duke}, Voiculescu characterised quasidiagonality in terms of the existence of almost isometric and almost multiplicative completely positive contractive (c.p.c.) maps into finite dimensional $\mathrm{C}^{*}$-algebras. It was observed in  \cite[2.4]{V:IEOT} that unital quasidiagonal $\mathrm{C}^*$-algebras always have traces (see also \cite[Proposition 7.1.6]{BrOz}). Traces which can be witnessed by quasidiagonal approximations are called quasidiagonal (see Definition~\ref{def:QDTraces} below); these were introduced and systematically investigated in \cite{B:MAMS}.

In Hadwin's paper \cite{Hadwin:JOT} quasidiagonality was linked to nuclearity of $\mathrm{C}^{*}$-algebras. The latter can be expressed in terms of c.p.c.\ approximations through finite dimensional $\mathrm{C}^{*}$-algebras; it is in many respects analogous to amenability for discrete groups. This connection was further exploited in \cite{BlaKir:MathAnn}; it also features prominently in Elliott's programme to classify simple nuclear $\mathrm{C}^{*}$-algebras by $K$-theoretic data. 

In \cite{RS:DMJ}, Rosenberg and Schochet established what they called the universal coefficient theorem (UCT), relating Kasparov's bivariant $KK$-theory to homomorphisms between $K$-groups. They showed that the UCT holds for all separable nuclear $\mathrm{C}^*$-algebras which are $KK$-equivalent to abelian ones.  It remains an important open problem dating back to \cite{RS:DMJ} whether this class in fact contains all separable nuclear $\mathrm{C}^{*}$-algebras. At first sight the UCT does not seem related to quasidiagonality, but one link was noted by L.~Brown in his work \cite{B:OAGR} on universal coefficient theorems for $\mathrm{Ext}$; see \cite[2.3]{V:IEOT}.  Our main result provides a new connection.

\begin{thm}
\label{thm:MainThm}
Let $A$ be a separable, nuclear $\mathrm C^*$-algebra which satisfies the UCT.
Then every faithful trace on $A$ is quasidiagonal.
\end{thm}

\bigskip

Before outlining the proof at the end of this introduction, let us describe some consequences.  Blackadar and Kirchberg in \cite{BlaKir:MathAnn} analysed the relation between quasidiagonality and other approximation properties, introducing the notions of generalised inductive limits and of MF and NF algebras. They asked whether every stably finite nuclear $\mathrm{C}^{*}$-algebra is quasidiagonal (now known as the Blackadar--Kirchberg problem). As the existence of a faithful quasidiagonal trace entails quasidiagonality, Theorem \ref{thm:MainThm} can be used to show that many nuclear $\mathrm{C}^{*}$-algebras are quasidiagonal. A variety of natural constructions of stably finite unital nuclear $\mathrm{C}^*$-algebras, such as the reduced group $\mathrm{C}^*$-algebra of a discrete group, come readily equipped with a faithful trace, but in general one needs Haagerup's spectacular theorem (\cite{Haa:quasitraces}) to obtain a trace on stably finite unital nuclear $\mathrm{C}^*$-algebra; when the algebra is simple every trace must be faithful. With this, we obtain:
\begin{cor}
\label{cor:QD}
Every separable, nuclear $\mathrm{C}^{*}$-algebra in the UCT class with a faithful trace is quasidiagonal. In particular, the Blackadar--Kirchberg problem has an affirmative answer for simple $\mathrm{C}^*$-algebras satisfying the UCT.
\end{cor}

\bigskip

In \cite{Hadwin:JOT}, Hadwin asked which locally compact groups have (strongly) quasidiagonal $\mathrm{C}^{*}$-algebras. In the appendix to \cite{Hadwin:JOT}, Rosenberg showed that any countable discrete group $G$ with quasidiagonal left regular representation is ame\-nable. As pointed out in \cite[Corollary~7.1.17]{BrOz}, the argument does not really depend on the specific representation, and $G$ is amenable provided $\mathrm{C}^{*}_{\mathrm{r}}(G)$ is quasidiagonal. The converse to this statement -- is $\mathrm{C}^{*}_{\mathrm{r}}(G)$ quasidiagonal if $G$ is amenable -- has received substantial attention over the decades and is now commonly referred to as the Rosenberg conjecture; cf.\ \cite[Section 3]{V:IEOT} and \cite[Section~7.1]{BrOz}. Recently, Carri{\'o}n, Dadarlat, and Eckhardt gave a positive answer for groups which are locally embeddable into finite groups (LEF groups) in \cite{CDE:JFA}. Even more recently, Ozawa, R\o{}rdam, and Sato handled the case of elementary amenable groups and extensions of LEF groups by elementary amenable groups in \cite{ORS:GAFA}. In this paper we settle the matter, as a consequence of  Theorem~\ref{thm:MainThm}:

\begin{cor}
\label{cor:rosenberg}
Let $G$ be a discrete, amenable group.  Then $\mathrm{C}^*_{\mathrm{r}}(G)$ is quasidiagonal.
\end{cor}

Using results from classification, we actually get the stronger statement that $\mathrm{C}^{*}_{\mathrm{r}}(G)$ is AF-embeddable if $G$ is countable, discrete, and amenable; see Corollary~\ref{AF-embeddablegps}  (see also \cite[Section 4]{V:IEOT} and \cite[Section 8]{BrOz} for a discussion of AF-embeddability and relations to quasidiagonality).

Note that the approach of \cite{ORS:GAFA} uses the full strength of Elliott's programme in the simple and monotracial case (which is remarkable, since group $\mathrm{C}^{*}$-algebras themselves are not expected to be directly accessible to classification). Our argument also has a classification component (albeit of a more basic sort), namely through the technology of stable uniqueness theorems which lie at the heart of numerous classification results such as \cite{EG:Ann}, \cite{Lin:Duke} and \cite{GLN:arXiv}. Early on, stable uniqueness theorems were derived for commutative domains in \cite{EG:Ann,EGLP:DMJ} with a definitive theorem in this direction obtained in \cite{D:KThy}.  For more general domains, they were further developed by Lin in \cite{Lin:JOT} and \cite{Lin:apprUCT} and by Dadarlat and Eilers in \cite{DE:PLMS}, of which more is said later on.

\bigskip

The first prominent entry of quasidiagonality into Elliott's classification programme was in the purely infinite case, in the proof of Kirchberg's $\mathcal{O}_{2}$-embeddability theorem \cite{Kir:ICM}, via Voiculescu's remarkable result on homotopy invariance of quasidiagonality \cite{Voi:Duke}. 

The relevance of quasidiagonality in the stably finite case goes back to Popa's article \cite{Popa:PJM}, which uses local quantisation to establish an internal approximation property for simple quasidiagonal $\mathrm{C}^{*}$-algebras with sufficiently many projections. This idea triggered Lin's definition of tracially approximately finite dimensional (TAF)	 algebras, which then made it possible to come up with the first stably finite analogue of Kirchberg--Phillips classification; see \cite{Lin:Duke}. 

The potential relevance of quasidiagonality of traces to Elliott's programme was highlighted well ahead of its time by N.~Brown in \cite[Section 6.1]{B:MAMS}. This foresight has come to fruition very recently through the entry of quasidiagonal traces in the noncommutative covering dimension calculations of \cite{BBSTWW:arXiv}.  The main result of \cite{EGLN:arXiv} (which in turn relies on \cite{EN:arXiv,R:Adv,Win:AJM} to reduce to the classification theorem of \cite{GLN:arXiv}, a result which builds on a large body of work over decades) classifies separable, unital, simple, nuclear $\mathrm{C}^*$-algebras with finite nuclear dimension (the noncommutative notion of covering dimension from \cite{WinZac:dimnuc}) in the UCT class under the additional assumption that all traces are quasidiagonal.  Theorem \ref{thm:MainThm} removes this assumption, and so completes the classification of these algebras.  Since finite nuclear dimension is a ``finitely coloured'' or higher dimensional version of being approximately finite dimensional (AF), the following corollary can be viewed as the $\mathrm{C}^*$-analogue of the classification of hyperfinite von Neumann factors.

\begin{cor}
\label{cor:Classification}
Let $A$ and $B$ be separable,  unital, simple and infinite dimensional $\mathrm{C}^*$-algebras with finite nuclear dimension which satisfy the UCT.  Then $A$ is isomorphic to $B$ if and only if $A$ and $B$ have isomorphic Elliott invariants. 
\end{cor}

We wish to point out that our contribution is just the final piece in a longtime endeavour. Also, it only affects the stably finite case of Elliott's programme -- the purely infinite case was dealt with by Kirchberg and Phillips in the 90's. We deliberately state Corollary~\ref{cor:Classification} without making this distinction to emphasise that the two strands now run completely parallel. 

In the situation of the corollary, one can extract much more refined information. For example, the $\mathrm C^*$-algebras are inductive limits of nice building block algebras with low topological dimension (subhomogeneous algebras in the stably finite case and Cuntz algebras over circles in the purely infinite situation; see \cite{R:Book}). Moreover, they have nuclear dimension 1 and decomposition rank $\infty$ (see \cite{KirWin:dr}) when they are purely infinite (under the hypotheses of the corollary, this is equivalent to the absence of traces); in the stably finite case (i.e., in the presence of traces) their decomposition rank is at most 2 (in many cases, and probably always, it is at most $1$). The nuclear dimension hypothesis in Corollary~\ref{cor:Classification} is known to hold for $\mathcal{Z}$-stable $\mathrm{C}^{*}$-algebras whose trace spaces are empty or have compact extreme boundaries; see \cite{BBSTWW:arXiv}. (Here, $\mathcal{Z}$-stability is tensorial absorption of the Jiang--Su algebra $\mathcal{Z}$ -- see \cite{JS:AJM} -- which is a $\mathrm{C}^{*}$-algebraic analogue of a von Neumann algebra being McDuff.) In the case of at most one trace we obtain a particularly clean statement, which again encapsulates both the finite and the infinite case (and which unlike Corollary~\ref{cor:Classification} does not rely on the as yet unpublished results of \cite{BBSTWW:arXiv}, \cite{EGLN:arXiv} and \cite{GLN:arXiv}):

\begin{cor}
\label{monotracial-classification}
Separable, simple, unital, nuclear  $\mathrm{C}^*$-algebras in the UCT class with at most one trace are classified up to $\mathcal{Z}$-stability by their ordered $K$-theory.  
\end{cor}

By classification up to $\mathcal{Z}$-stability, we mean that any isomorphism between the Elliott invariants (which reduce to ordered $K$-theory in the monotracial case and to just $K$-theory in the purely infinite situation) of $A \otimes \mathcal{Z}$ and $B \otimes \mathcal{Z}$ lifts to an isomorphism between $A \otimes \mathcal{Z}$ and $B \otimes \mathcal{Z}$. There are many important classes of C*-algebras that are automatically $\mathcal{Z}$-stable, and there are highly useful criteria to check this property. Even without $\mathcal{Z}$-stability we obtain a nontrivial statement, since one can simply tensor by $\mathcal{Z}$ and determine the isomorphism class of the resulting $\mathrm C^*$-algebra in terms of the invariant.

\bigskip

We will discuss these and other consequences of Theorem \ref{thm:MainThm} -- including its  relation to the Toms--Winter conjecture -- in greater detail in Section \ref{Structure}.
\bigskip

The rest of this introduction is devoted to an outline of the proof of Theorem \ref{thm:MainThm}, so fix a separable, unital and nuclear $\mathrm{C}^*$-algebra $A$ in the UCT class, and a faithful trace $\tau_A$ on $A$. We seek an embedding $\Psi$ of $A$ into the ultrapower, $\Q_\omega$, of the universal UHF algebra inducing $\tau_A$ (i.e.\ $\tau_{A} = \tau_{\Q_\omega}\circ\Psi$). 

Our starting point is the order zero tracial quasidiagonality result from \cite[Section 3]{SWW:Invent}, a consequence of the uniqueness of the injective II$_1$ factor. From this one obtains a kind of ``$2$-coloured quasidiagonality of $\tau_A$'', which we express in terms of homomorphisms from cones: there exist two $^*$-homomorphisms\footnote{The notation $\grave{\Phi}$ and $\acute{\Phi}$ is designed to indicate the orientation of the cones $C_0([0,1))$ and $C_0((0,1])$ respectively through the appearance of canonical positive contractions generating these cones.}   
 $\grave\Phi:C_0([0,1),A)\to\Q_\omega$ and $\acute\Phi:C_0((0,1],A)\to\Q_\omega$, which both induce the trace $x\mapsto\int_0^1 \tau_A(x(t))dt$ and whose scalar parts $C_0([0,1),\mathbb C1_A)$ and $C_0((0,1],\mathbb C1_A)$ are the restrictions of a common unital $^*$-homomor\-phism $\Theta:C([0,1])\to\Q_\omega$. The details of how to obtain these cone $^*$-homomorphisms from \cite{SWW:Invent} are the subject of Section \ref{Sect:LebCones}.
 
To construct the desired $^*$-homomorphism $\Psi$ from $\grave{\Phi}$ and $\acute{\Phi}$ we use a ``stable uniqueness across the interval argument'' (see \cite{EN:arXiv} for an argument of a similar flavour). Imagine that we can find some (presumably large) $n\in\mathbb N$ such that for any open 
interval $J\subseteq (0,1)$, 
\begin{align}
&\grave\Phi|_{C_0(J,A)}^{\oplus n}\text{ and }\grave\Phi|_{C_0(J,A)}^{\oplus (n-1)}\oplus \acute\Phi|_{C_0(J,A)}^{\oplus 1}\text{ are unitarily equivalent; and }\nonumber\\
&\acute\Phi|_{C_0(J,A)}^{\oplus n}\text{ and }\acute\Phi|_{C_0(J,A)}^{\oplus (n-1)}\oplus \grave\Phi|_{C_0(J,A)}^{\oplus 1}\text{ are unitarily equivalent,}\label{intro.ue}
\end{align}
as maps into $\Q_\omega\otimes M_n$.  Write $N:=2n$ and divide $[0,1]$ into $N$ overlapping subintervals $I_1,\dots,I_N$ as painted in Figure \ref{Fig.Intro}.\footnote{In the actual proof, we define $I_1$ and $I_N$ to be open, and introduce additional intervals $I_0,I_{N+1}$ and corresponding maps.}

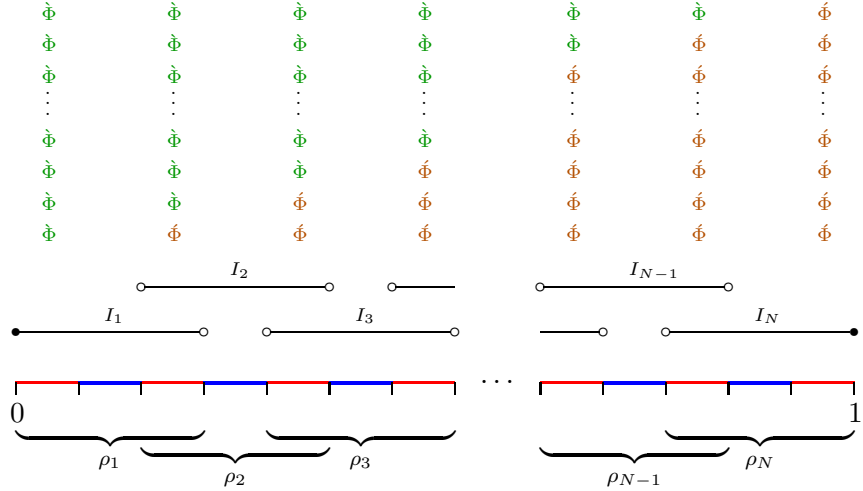
\begin{figure}[h!]
\setlength{\unitlength}{0.15mm}

\begin{picture}(855,410)
\color{red} \linethickness{.25mm}
\put(60,70){\line(1,0){55}}  % interval
\put(170,70){\line(1,0){55}}  % interval
\put(280,70){\line(1,0){55}}  % interval
\put(390,70){\line(1,0){55}}  % interval
\put(520,70){\line(1,0){55}}  % interval
\put(630,70){\line(1,0){55}}  % interval
\put(740,70){\line(1,0){55}}  % interval
\color{blue} \linethickness{.4mm}
\put(115,70){\line(1,0){55}}  % interval
\put(225,70){\line(1,0){55}}  % interval
\put(335,70){\line(1,0){55}}  % interval
\put(575,70){\line(1,0){55}}  % interval
\put(685,70){\line(1,0){55}}  % interval
\color{black} \linethickness{.15mm}
\put(60,70){\line(0,-1){10}}  % tick
\put(115,70){\line(0,-1){10}}  % tick
\put(170,70){\line(0,-1){10}}  % tick
\put(225,70){\line(0,-1){10}}  % tick
\put(280,70){\line(0,-1){10}}  % tick
\put(335,70){\line(0,-1){10}}  % tick
\put(390,70){\line(0,-1){10}}  % tick
\put(445,70){\line(0,-1){10}}  % tick
\put(520,70){\line(0,-1){10}}  % tick
\put(575,70){\line(0,-1){10}}  % tick
\put(630,70){\line(0,-1){10}}  % tick
\put(685,70){\line(0,-1){10}}  % tick
\put(740,70){\line(0,-1){10}}  % tick
\put(795,70){\line(0,-1){10}}  % tick
 
%\put(475,70){$\dots$}  % ellipsis
\put(467,70){$\dots$}  % ellipsis
\put(55,35){$0$}  % x-label
\put(790,35){$1$}  % x-label
 
\put(60,115){\circle*{7}}  % I_1
\put(63,115){\line(1,0){158}}  % I_1
\put(225,115){\circle{7}}  % I_1
\put(137.5,125){\tiny{$I_1$}}  % I_1
\put(170,155){\circle{7}}  % I_2
\put(173,155){\line(1,0){158}}  % I_2
\put(335,155){\circle{7}}  % I_2
\put(247.5,165){\tiny{$I_2$}}  % I_2
\put(280,115){\circle{7}}  % I_3
\put(283,115){\line(1,0){158}}  % I_3
\put(445,115){\circle{7}}  % I_3
\put(357.5,125){\tiny{$I_3$}}  % I_3
\put(390,155){\circle{7}}  % I_4
\put(393,155){\line(1,0){52}}  % I_4
\put(520,115){\line(1,0){52}}  % I_{N-2}
\put(575,115){\circle{7}}  % I_{N-2}
\put(520,155){\circle{7}}  % I_{N-1}
\put(523,155){\line(1,0){158}}  % I_{N-1}
\put(685,155){\circle{7}}  % I_{N-1}
\put(597.5,165){\tiny{$I_{N-1}$}}  % I_{N-1}
\put(630,115){\circle{7}}  % I_N
\put(633,115){\line(1,0){158}}  % I_N
\put(795,115){\circle*{7}}  % I_N
\put(707.5,125){\tiny{$I_N$}}  % I_N
\color{a}
\put(82.5,195){\tiny{$\grave\Phi$}}  % 1st stack
\put(82.5,223){\tiny{$\grave\Phi$}}  % 1st stack
\put(82.5,251){\tiny{$\grave\Phi$}}  % 1st stack
\put(82.5,279){\tiny{$\grave\Phi$}}  % 1st stack
\put(82.5,335){\tiny{$\grave\Phi$}}  % 1st stack
\put(82.5,363){\tiny{$\grave\Phi$}}  % 1st stack
\put(82.5,391){\tiny{$\grave\Phi$}}  % 1st stack
\color{d}
\put(192.5,195){\tiny{$\acute\Phi$}}  % 2nd stack
\color{a}
\put(192.5,223){\tiny{$\grave\Phi$}}  % 2nd stack
\put(192.5,251){\tiny{$\grave\Phi$}}  % 2nd stack
\put(192.5,279){\tiny{$\grave\Phi$}}  % 2nd stack
\put(192.5,335){\tiny{$\grave\Phi$}}  % 2nd stack
\put(192.5,363){\tiny{$\grave\Phi$}}  % 2nd stack
\put(192.5,391){\tiny{$\grave\Phi$}}  % 2nd stack
\color{d}
\put(302.5,195){\tiny{$\acute\Phi$}}  % 3rd stack
\put(302.5,223){\tiny{$\acute\Phi$}}  % 3rd stack
\color{a}
\put(302.5,251){\tiny{$\grave\Phi$}}  % 3rd stack
\put(302.5,279){\tiny{$\grave\Phi$}}  % 3rd stack
\put(302.5,335){\tiny{$\grave\Phi$}}  % 3rd stack
\put(302.5,363){\tiny{$\grave\Phi$}}  % 3rd stack
\put(302.5,391){\tiny{$\grave\Phi$}}  % 3rd stack
\color{d}
\put(412.5,195){\tiny{$\acute\Phi$}}  % 4th stack
\put(412.5,223){\tiny{$\acute\Phi$}}  % 4th stack
\put(412.5,251){\tiny{$\acute\Phi$}}  % 4th stack
\color{a}
\put(412.5,279){\tiny{$\grave\Phi$}}  % 4th stack
\put(412.5,335){\tiny{$\grave\Phi$}}  % 4th stack
\put(412.5,363){\tiny{$\grave\Phi$}}  % 4th stack
\put(412.5,391){\tiny{$\grave\Phi$}}  % 4th stack
\color{d}
\put(542.5,195){\tiny{$\acute\Phi$}}  % (N-1)-th stack
\put(542.5,223){\tiny{$\acute\Phi$}}  % (N-1)-th stack
\put(542.5,251){\tiny{$\acute\Phi$}}  % (N-1)-th stack
\put(542.5,279){\tiny{$\acute\Phi$}}  % (N-1)-th stack
\put(542.5,335){\tiny{$\acute\Phi$}}  % (N-1)-th stack
\color{a}
\put(542.5,363){\tiny{$\grave\Phi$}}  % (N-1)-th stack
\put(542.5,391){\tiny{$\grave\Phi$}}  % (N-1)-th stack
\color{d}
\put(652.5,195){\tiny{$\acute\Phi$}}  % N-th stack
\put(652.5,223){\tiny{$\acute\Phi$}}  % N-th stack
\put(652.5,251){\tiny{$\acute\Phi$}}  % N-th stack
\put(652.5,279){\tiny{$\acute\Phi$}}  % N-th stack
\put(652.5,335){\tiny{$\acute\Phi$}}  % N-th stack
\put(652.5,363){\tiny{$\acute\Phi$}}  % N-th stack
\color{a}
\put(652.5,391){\tiny{$\grave\Phi$}}  % N-th stack
\color{d}
\put(762.5,195){\tiny{$\acute\Phi$}}  % (N+1)-th stack
\put(762.5,223){\tiny{$\acute\Phi$}}  % (N+1)-th stack
\put(762.5,251){\tiny{$\acute\Phi$}}  % (N+1)-th stack
\put(762.5,279){\tiny{$\acute\Phi$}}  % (N+1)-th stack
\put(762.5,335){\tiny{$\acute\Phi$}}  % (N+1)-th stack
\put(762.5,363){\tiny{$\acute\Phi$}}  % (N+1)-th stack
\put(762.5,391){\tiny{$\acute\Phi$}}  % (N+1)-th stack
\color{black}
\put(85.5,307){\tiny{$\vdots$}}  % 1st stack
\put(195.5,307){\tiny{$\vdots$}}  % 2nd stack
\put(305.5,307){\tiny{$\vdots$}}  % 3rd stack
\put(415.5,307){\tiny{$\vdots$}}  % 4th stack
\put(545.5,307){\tiny{$\vdots$}}  % (N-1)-th stack
\put(655.5,307){\tiny{$\vdots$}}  % N-th stack
\put(765.5,307){\tiny{$\vdots$}}  % (N+1)-th stack
 
\put(60,30){$\underbrace{\hspace{24.75mm}}_{\rho_1}$}  % \rho_1
\put(170,15){$\underbrace{\hspace{24.75mm}}_{\rho_2}$}  % \rho_2
\put(280,30){$\underbrace{\hspace{24.75mm}}_{\rho_3}$}  % \rho_3
\put(520,15){$\underbrace{\hspace{24.75mm}}_{\rho_{N-1}}$}  % \rho_{N-1}
\put(630,30){$\underbrace{\hspace{24.75mm}}_{\rho_N}$}  % \rho_N

\end{picture}
\caption{Stable uniqueness across the interval}\label{Fig.Intro}
\end{figure}

In this way, on each closed blue middle third of $I_i$
\begin{equation}\label{e0.2}
\grave\Phi|^{\oplus (N-i+1)}_{\ldots}\oplus \acute\Phi|^{\oplus (i-1)}_{\ldots}\text{ and }\grave\Phi|^{\oplus (N-i)}_{\ldots}\oplus \acute\Phi|^{\oplus i}_{\ldots}
\end{equation}
are unitarily equivalent as maps into $\Q_\omega\otimes M_N$ (using either the first or second line of \eqref{intro.ue}, depending on whether $i \leq n$ or not). From this unitary equivalence, we can ``patch'' across each of the intervals $I_i$, to produce $^*$-homo\-mor\-phisms $\rho_i$ defined on $C_0(I_i,A)$ which are compatible with $\Theta$ and agree with the two maps in (\ref{e0.2})  on the left- and right-hand red thirds of $I_i$.  This is achieved using the additional space given by a $2$-fold matrix amplification, so that each $\rho_i$ maps $C_0(I_i,A)$ into $\Q_\omega\otimes M_N\otimes M_2$ and restricts to 
\begin{equation}
\begin{pmatrix}\grave\Phi|_{\ldots}^{\oplus (N-i+1)}\oplus \acute\Phi|_{\ldots}^{\oplus (i-1)}&0\\0&0\end{pmatrix}\text{ and }\begin{pmatrix}\grave\Phi|_{\ldots}^{\oplus (N-i)}\oplus \acute\Phi|_{\ldots}^{\oplus i}&0\\0&0\end{pmatrix}
\end{equation}
on the left and right hand thirds of $I_i$ respectively. Additionally, the construction is such that $\rho_i$ induces the tracial functional $x\mapsto \frac{1}{2}\int_{I_i} \tau_A(x(t))\,dt$; the factor of $\frac{1}{2}$ arising from the two-fold amplification. In particular $\rho_i$ and $\rho_{i+1}$ agree on $I_i\cap I_{i+1}$.  Then, taking a partition of unity $(f_i)_i$ for $C([0,1])$ subordinate to $(I_i)_i$, $\Psi(a):=\sum_{i=1}^N\rho_i(f_i\otimes a)$ defines a $^*$-homomorphism into $\Q_\omega\otimes M_N\otimes M_2$ inducing $\frac{1}{2}\tau_A$.  This would suffice to establish quasidiagonality of $\tau_A$.

Finding an $n$ satisfying \eqref{intro.ue} on the nose seems challenging. However, working throughout with  the greater flexibility of approximate unitary equivalence on finite sets up to specified tolerances, a small modification of a ``stable uniqueness theorem'' of Dadarlat and Eilers from \cite{DE:PLMS} -- which was motivated by, and both reproves and generalises that of Lin \cite{Lin:JOT} (see also \cite{Lin:apprUCT}) -- gives the required $n$ for an approximate version of \eqref{intro.ue}.

Roughly speaking, stable uniqueness theorems ensure that nuclear $^*$-homomorphisms $\phi,\psi:C\rightarrow D$ which agree at the level of $KK$ become approximately unitary equivalent after taking a direct sum with an amplification of a suitable map $\iota:C\rightarrow D$. Somewhat more precisely, given a finite subset $\mathcal G$ of $C$ and tolerance $\delta>0$, there exists some $n\in\mathbb N$ such that $\phi\oplus\iota^{\oplus n}$ is approximately unitarily equivalent to $\psi\oplus\iota^{\oplus n}$ on $\mathcal G$ up to $\delta$ (see \cite[Theorem 4.5]{DE:PLMS} and Theorem \ref{DE:SU} below for the precise statement).  As explained in \cite{DE:PLMS}, for many applications, including ours, it is of particular importance to know when $n$ can be chosen to depend only on $\mathcal G$ and $\delta$ and not on the particular maps $\phi,\psi$ and $\iota$. Using the UCT, Dadarlat and Eilers achieved such a result, which has subsequently seen wide application, by means of a sequence of counterexamples argument for simple and unital $A$ (see \cite[Theorem 4.12]{DE:PLMS}). The UCT enters crucially to control the $KK$-classes of the sequence of counterexamples.

In our case the domain $C$ is nonsimple, so we give a slight generalisation of \cite[Theorem 4.12]{DE:PLMS} to allow for this, using ``controlled fullness'' in place of simplicity. In nice cases (including the situation of Theorem \ref{thm:MainThm}), this leads to a stable uniqueness theorem where $n$ depends on $\mathcal G$ and $\delta$ and the tracial data of $\iota$, but not on further properties of $\iota$, nor on $\phi$ and $\psi$.  We set out this generalisation in Section \ref{Sect:SU}.  A quite similar result can be found in \cite{Lin:apprUCT}.  

Thus, we fix a canonical domain algebra $C\cong C_{0}((\frac13,\frac23),A)^\sim$ which, after an appropriate rescaling, will stand in for $C_0(J,A)^\sim$ ($J$ being a small interval), and obtain an approximate version of (\ref{intro.ue}) for $\mathcal G$ and $\delta$ (which in turn come from an approximate version of the patching lemma). However, as $|J|$ gets smaller, the estimates on the controlled fullness of the resulting maps 
\begin{equation}\label{intro.scale}
\xymatrix{C\ar[rr]^{\cong\quad\quad}&&C_0(J,A)^\sim\ar[rrr]^{\quad\quad\acute\Phi|_{C_0(J,A)}^\sim,\ \grave\Phi|_{C_0(J,A)}^\sim}&&&{\Q_\omega}}
\end{equation}
will get worse. This is rectified using projections $(q_J)_J$ in $\Q_\omega$ (constructed in Lemma \ref{lem:CharFunction}) such that the maps in \eqref{intro.scale} take values in the corner $q_J\Q_\omega q_J$. With respect to the unique normalised trace on this corner, the tracial behaviour of these maps becomes independent of $J$ so that our controlled stable uniqueness theorem applies (agreement of the maps in (\ref{intro.scale}) on total $K$-theory is essentially automatic, as they are constructed from contractible maps). Finally, by exploiting the inherent orthogonality of $\rho_i$ and $\rho_j$ whenever $I_i\cap I_j=\emptyset$, we ensure that the final approximate multiplicativity estimates for $\Psi$ do not depend on $n$.

Full details of the stable uniqueness across the interval argument are given in Section \ref{Section:Proof} with the approximate patching result in Section \ref{Sect.Patch}.

\subsection*{Acknowledgements} AT and SW thank WW, the staff and members of the Mathematics Institute at WWU M\"unster for their help and hospitality during the extended visits during which this research was undertaken. We thank Bruce Blackadar, Nate Brown, Jos{\'e} Carri{\'o}n, Marius Dadarlat, Caleb Eckhardt, George Elliott, Ilijas Farah, Guihua Gong, David Kerr, Eberhard Kirchberg, Huaxin Lin, Hiroki Matui, Zhuang Niu, Narutaka Ozawa, Mikael R{\o}rdam and Yasuhiko Sato for inspiring conversations. We also thank Dominic Enders, Jamie Gabe and Ilan Hirshberg for helpful comments on an earlier draft of this paper. Finally, we thank the referee for their helpful comments on the first submitted version.

\renewcommand*{\thetheorem}{\roman{theorem}}
\numberwithin{theorem}{section}

\section{Quasidiagonality, $\Q_\omega$, and the UCT}

\noindent In this section we record a number of preliminary facts regarding quasidiagonality, quasidiagonal traces, the ultrapower of the universal UHF algebra and the universal coefficient theorem.

Let $A$ be a $\mathrm C^*$-algebra. Denote the positive cone of $A$ by $A_+$, and the set of positive contractions by $A_+^1$.
For $a,b \in A$ and $\varepsilon > 0$, we write $a \approx_\varepsilon b$ to mean $\|a-b\|\leq \varepsilon$ (and we use the respective notation for real or complex numbers).
For a unitary $u \in A$, we let $\mathrm{Ad}(u)$ denote the automorphism of $A$ defined by $\mathrm{Ad}(u)(a):=uau^*$, $a \in A$.

A \emph{tracial functional} is a positive functional $\tau:A \to \mathbb C$ such that $\tau(ab)=\tau(ba)$ for every $a,b\in A$ and for us a \emph{trace} is a tracial functional that is normalised, i.e., a state. We write $T(A)$ for the collection of all traces on $A$. The $\mathrm C^*$-algebra of $k \times k$ complex matrices is denoted $M_k$; we write $\tau_{M_k}$ for its unique trace, and $\mathrm{Tr}_k$ for the canonical tracial functional on $M_k$ satisfying $\mathrm{Tr}_k(1_k) = k$.

We now recall the central notions of this paper: quasidiagonality and quasidiagonal traces.
(Rather than the original definition of quasidiagonality in terms of a quasidiagonal faithful representation, we use as our definition a formulation shown to be equivalent by Voiculescu in \cite[Theorem 1]{Voi:Duke}.)
We make these definitions for unital $\mathrm{C}^*$-algebras, the primary context of the paper; see Remark \ref{QDNonUnital} for the non-unital case.

\begin{defn}\label{DefQD}
A unital $\mathrm{C}^*$-algebra $A$ is \emph{quasidiagonal} if, for every finite subset $\mathcal F_A$ of $A$ and $\varepsilon>0$, there exist a matrix algebra $M_k$ and a unital completely positive (u.c.p.{}) map $\psi:A\to M_k$ such that
\begin{align}
\psi(ab) &\approx_\varepsilon \psi(a)\psi(b),\quad a,b \in \mathcal F_A, \quad \text{and} \\
\|\psi(a)\|&\approx_\varepsilon\|a\|,\quad a\in\mathcal F_A.
\end{align}
\end{defn}

\begin{defn}[{\cite[Definition 3.3.1]{B:MAMS}}]\label{def:QDTraces}
Let $A$ be a unital $\mathrm C^*$-algebra. A trace $\tau_A\in T(A)$ is \emph{quasidiagonal} if for every finite set $\mathcal F_A$ of $A$ and $\varepsilon>0$, there exist a matrix algebra $M_k$ and a u.c.p.\ map $\psi:A \to M_k$ such that
\begin{align}
\label{eq:QDTraces0}
\psi(ab) & \approx_\varepsilon \psi(a)\psi(b),\quad a,b \in \mathcal F_A, \quad \text{and} \\
\label{eq:QDTraces1}
 \tau_{M_k} \circ \psi (a) & \approx_\varepsilon \tau_A(a),\quad a\in\mathcal F_A.
\end{align}
\end{defn}
Note that quasidiagonality of a trace $\tau_A$ does not mean the same thing as quasidiagonality (as a set of operators, in Halmos' sense) of $\pi_{\tau_A}(A)$, where $\pi_{\tau_A}$ is the GNS representation corresponding to $\tau_A$.

We now turn to a standard characterisation of quasidiagonality in terms of embeddings. Write $\Q$ for the universal UHF algebra. Fix, throughout the paper, a free ultrafilter $\omega$ on $\mathbb N$, and let $\Q_\omega$ be the ultrapower of $\Q$, defined by
\begin{equation}
\Q_\omega := \ell^\infty(\Q) / \{(x_n)\in\ell^\infty(\Q) \mid \lim_{n\to\omega} \|x_n\| = 0\}.
\end{equation}
When there is no prospect of confusion, we often work with a representative of an element $x$ in $\Q_\omega$, namely a lift of $x$ to some element in $\ell^\infty(\Q)$ as in \eqref{DeftauQw} and Lemma \ref{lem:CharFunction} below. In particular, recall that all projections can be represented by sequences of projections, and unitaries by sequences of unitaries. 

The unique trace $\tau_\Q$ on $\Q$ induces a trace $\tau_{\Q_\omega}$ on $\Q_\omega$ by
\begin{equation}\label{DeftauQw} \tau_{\Q_\omega}(x) := \lim_\omega \tau_\Q(x_n), \quad x \in \Q_\omega \text{ is represented by }(x_n) \in \ell^{\infty}(\mathcal{Q}).
\end{equation}
By \cite[Lemma 4.7]{MS:AJM} (see also \cite[Theorem 8]{O:JMSUT}), $\tau_{\Q_\omega}$ is the unique trace on $\Q_\omega$. 
The following two observations are consequences of the standard fact that for any nonzero projection $q\in\Q$, $q\Q q\cong\Q$.
\begin{prop}
\label{prop:Qfacts}
\begin{enumerate}[(i)]
\item For any nonzero projection $q \in \Q_\omega$, $q\Q_\omega q \cong \Q_\omega$.\label{prop:Qfacts.1}
\item For any projection $q\in\Q_\omega$ with $\tau_{\Q_\omega}(q)>0$, there exist $k\in\mathbb N$ and a (not necessarily unital) embedding $\theta:\Q_\omega\to q\Q_\omega q\otimes M_k$ such that $\theta(qxq)=qxq\otimes e_{11}$ for $x\in\Q_\omega$.\label{prop:Qfacts.2}
\end{enumerate}
\end{prop}

\begin{proof}
(\ref{prop:Qfacts.1}) Given a nonzero projection $q\in\Q_\omega$, we can find a representative $(q_n)$ of $q$ such that each $q_n$ is a nonzero projection in $\Q$. For each $n$, fix an isomorphism $\theta_n:\Q\rightarrow q_n\Q q_n\subset\Q$.  Then the sequence $(\theta_n)$ induces an injective $^*$-homomorphism $\theta:\Q_\omega\rightarrow \Q_\omega$ with image $q\Q_\omega q$.

(\ref{prop:Qfacts.2}) Given such a $q$, fix $k\in\mathbb N$ with $\frac{1}{k}<\tau_{\Q_\omega}(q)$, and fix a representative sequence $(q_n)$ of $q$ by projections in $\Q$ with $\tau_\Q(q_n)>\frac{1}{k}$ for all $n$. Since $k\tau_{\Q}(q_n)>1$, elementary properties of $\Q$ provide embeddings $\theta_n:\Q\rightarrow q_n\Q q_n\otimes M_k$ with $\theta_n(q_nxq_n)=q_nxq_n\otimes e_{11}$ for $x\in \Q$.  The embedding $\theta:\Q_\omega\to q\Q_\omega q\otimes M_k$ induced by the sequence $(\theta_n)$ then has the specified property.
\end{proof}

We now record characterisations of quasidiagonality and quasidiagonal traces in terms of maps into $\Q_\omega$. The condition in (\ref{EasyQDT.3}) below is exactly the form we will use to prove Theorem \ref{thm:MainThm}.
\begin{prop}\label{Prop.EasyQDT}
Let $A$ be a separable, unital and nuclear $\mathrm{C}^*$-algebra.
\begin{enumerate}[(i)]
\item\label{EasyQD} Then $A$ is quasidiagonal if and only if there exists a unital embedding $A\hookrightarrow \Q_\omega$.
\item\label{EasyQDT} For a trace $\tau_{A}\in T(A)$, the following statements are equivalent:
\begin{enumerate}[(a)]
\item $\tau_A$ is quasidiagonal;\label{EasyQDT.1}
\item\label{EasyQDT.2} there exists a unital $^*$-homomorphism $\theta:A\to\Q_\omega$ such that $\tau_{\Q_\omega}\circ\theta=\tau_A$;
\item\label{EasyQDT.3} there exists $\gamma \in (0,1]$ such that for every finite subset $\mathcal F_A\subset A$ and $\varepsilon>0$ there is a completely positive map $\phi:A\rightarrow\Q_\omega$ such that
\begin{align}
\label{eq:EasyQD1}
\phi(ab)&\approx_\varepsilon \phi(a)\phi(b),\quad a,b\in \mathcal F_A, \text{ and}\\
\label{eq:EasyQD2}
\tau_{\Q_\omega}\circ\phi(a) &= \gamma\tau_A(a),\quad a\in \mathcal F_A.
\end{align}
\end{enumerate}
\end{enumerate}
In particular if $A$ has a faithful quasidiagonal trace, then $A$ is quasidiagonal.\footnote{Using a variant of this proposition which involves maps into $\Q_\omega$ with c.p.\ liftings to $\ell^\infty(\Q)$, this last statement holds for not-necessarily nuclear $\mathrm{C}^*$-algebras. }
\end{prop}
\begin{proof}
(\ref{EasyQD}) is well-known; we recall the argument as it will be reused in (\ref{EasyQDT}). By the Choi--Effros lifting theorem \cite[Theorem 3.10]{CE:Ann}, a unital embedding of $A$ into $\Q_\omega$ can be lifted to a sequence of c.p.c.\ maps into $\Q$, which must be approximately multiplicative and approximately isometric (along $\omega$). Following these maps by expectations onto sufficiently large unital matrix subalgebras of $\Q$ gives the maps required by Definition \ref{DefQD}.  Conversely, by separability of $A$, Definition \ref{DefQD} provides a sequence $\theta_n:A\rightarrow M_{k_n}\subset\Q$ of approximately multiplicative and approximately isometric u.c.p.\ maps.
The induced map is a unital embedding $A\hookrightarrow\Q_\omega$.

The equivalence of (\ref{EasyQDT.1}) and (\ref{EasyQDT.2}) proceeds in exactly the same fashion as the proof of (\ref{EasyQD}).  It is clear that (\ref{EasyQDT.2}) implies (\ref{EasyQDT.3}) (with $\gamma=1$).
Conversely, if condition (\ref{EasyQDT.3}) holds, then a standard reindexing argument (see for example \cite[Section 1.3]{BBSTWW:arXiv}) gives a $^*$-homomorphism $\theta:A\rightarrow\Q_\omega$ such that $\tau_{\Q_\omega}(\theta(a))=\gamma \tau_A(a)$ for all $a\in A$. Then $q:=\theta(1_A)$ is a projection in $\Q_\omega$ which is nonzero as $\tau_{\Q_\omega}(q)=\gamma>0$.  We can view $\theta$ as a unital $^*$-homomorphism taking values in $q\Q_\omega q$, which is isomorphic to $\Q_\omega$ by Proposition \ref{prop:Qfacts} (\ref{prop:Qfacts.1}). The uniqueness of the trace on $\Q_\omega$ ensures that $\tau_{q\Q_\omega q}(\theta(a))=\frac{1}{\tau_{\Q_\omega}(q)}\tau_{\Q_\omega}(\theta(a))=\tau_A(a)$ for $a\in A$, so (\ref{EasyQDT.2}) holds.

Finally, note that if $\tau_A$ is faithful, then a unital $^*$-homomorphism $\theta$ as in (\ref{EasyQDT.2}) is injective, so $A$ is quasidiagonal.
\end{proof}

\begin{remark}\label{QDNonUnital}
\begin{enumerate}[(i)]
\item \label{QDNonUnital1}
When $A$ is a non-unital $\mathrm{C}^*$-algebra one can define quasidiagonality in two natural equivalent ways: as in Definition \ref{DefQD} using c.p.c.\ maps in place of u.c.p.\ maps or by asking for the minimal unitisation $A^\sim$ to be quasidiagonal (see  \cite[Sections 2.2 and 7.1]{BrOz}). 
\item\label{QDNonUnital2} Likewise, given a trace $\tau_A$ on a non-unital $\mathrm{C}^*$-algebra $A$, we obtain a trace $\tau_{A^\sim}$ on $A^\sim$ by the formula
\begin{equation}
\tau_{A^{\sim}}(\lambda 1 + a) := \lambda + \tau_A(a), \quad \lambda \in \mathbb C,\ a \in A.
\end{equation}
We define $\tau_A$ to be quasidiagonal if the trace $\tau_{A^\sim}$ is, or equivalently, if Definition \ref{def:QDTraces} holds using c.p.c.\ maps in place of u.c.p.\ maps.
\item\label{QDNonUnital3} It is well-known that if the trace $\tau_A$ is faithful on the non-unital $\mathrm{C}^*$-algebra $A$, then so too is $\tau_{A^\sim}$.\footnote{Every positive element of $A^\sim\setminus A$ is, up to scalar multiplication, of the form $1-a$ where $a \in A$ is a self-adjoint whose positive part $a_+$ is a contraction. Since $a_+$ isn't the unit, $x(1-a_+)x^* \neq 0$ for some $x \in A$, so that $\tau_A(x(1-a_+)x^*)>0$. Hence $\tau_A(a_+) < 1$ (see \cite[Proposition 2.11]{T:MA}, for example) so that $\tau_{A^{\sim}}(1-a) \geq 1-\tau_A(a_+) > 0$.}
In particular, a faithful trace $\tau_A$ on $A$ is quasidiagonal provided the faithful trace $\tau_{A^\sim}$ on $A^\sim$ is. In this case $A^\sim$, and hence also $A$, is quasidiagonal.
\end{enumerate}
\end{remark}

The next lemma provides the projections in $\Q_\omega$ used to rescale the trace in the ``stable uniqueness across the interval'' procedure of Section~\ref{Section:Proof}.

\begin{lemma}
\label{lem:CharFunction}
Let $k \in \mathcal Q_\omega$ be a positive contraction such that $\tau_{\Q_\omega}(h(k)) = \int_0^1 h(t)\,dt$ for all $h \in C([0,1])$. For each relatively open interval $I\subseteq [0,1]$, there exists a projection $q_I\in\Q_\omega$ such that the family $(q_I)_I$ satisfies:
\begin{enumerate}[(i)]
\item $q_I$ commutes with $k$, for all $I$;\label{CharFunction1}
\item $q_I$ acts as a unit on $h(k)$, for all $h \in C_0(I)$ and all $I$;\label{CharFunction2}
\item if $h\in C([0,1])$ satisfies $h|_I \equiv 1$, then $h(k)q_I=q_I$;\label{CharFunctionNew}
\item if $I,J$ are disjoint, then $q_Iq_J=0$;\label{CharFunction3}
\item for all $I$, $\tau_{\mathcal Q_\omega}(q_I)$ is equal to the length, $|I|$, of $I$.\label{CharFunction4}
\end{enumerate}
\end{lemma}

\begin{proof}
Represent $k$ by a sequence $(k_n) \in \ell^\infty(\Q)$ of positive contractions with discrete spectrum.
Fix a relatively open interval $I$ in $[0,1]$. For each $n$, let $q_{I,n}\in \mathcal Q$ denote the spectral projection of $k_n$ corresponding to the interval $I$, and then set $q_I$ to be the element of $\Q_\omega$ represented by $(q_{I,n})$. 

As each $q_{I,n}$ commutes with $k_n$, it follows that $q_I$ 
commutes with $k$, hence (\ref{CharFunction1}) holds.  For $h\in C_0(I)$, $h(k)$ is represented by $(h(k_n))$.
Since each $q_{I,n}$ acts as a unit on $h(k_n)$, it follows that $q_I$ acts as a unit on $h(k)$, proving (\ref{CharFunction2}).
Likewise, in (\ref{CharFunctionNew}), if $h\in C([0,1])$ satisfies $h|_I\equiv 1$, then $h(k_n)q_{I,n}=q_{I,n}$ for each $n$ so that $h(k)q_I=q_I$.
For (\ref{CharFunction3}), note that if $I\cap J=\emptyset$, then $q_{I,n}q_{J,n}=0$ for all $n$, whence $q_Iq_J=0$.  

For (\ref{CharFunction4}), note that (\ref{CharFunction2}) gives
\begin{equation} \tau_{\Q_\omega}(q_I) \geq \sup_{h \in C_0(I)_+^1}\int_0^1h(t)dt = |I|. \end{equation}
Likewise, by (\ref{CharFunctionNew}), 
\begin{equation} \tau_{\Q_\omega}(q_I) \leq \inf_{h \in C([0,1])_+^1,\ h|_I \equiv 1} \int_0^1h(t)dt  = |I|. \end{equation}
These two inequalities yield (\ref{CharFunction4}).
\end{proof}

\begin{remark*} The family of projections $(q_I)$ in Lemma \ref{lem:CharFunction} is not canonical as they depend on the choice of lift for $k$ in $\ell^\infty(\Q)$.  
\end{remark*}

We now turn to the universal coefficient theorem (UCT) which is a key component of the stable uniqueness theorem in Section \ref{Sect:SU}. The definition, given below for completeness, is in terms of a natural sequence relating Kasparov's bivariant $KK$-groups to $K$-theory.

\begin{defn}[{cf.\ \cite[Theorem 1.17]{RS:DMJ}}]
\label{UCT:def}
A separable $\mathrm{C}^*$-algebra $A$ is said to satisfy the \emph{universal coefficient theorem} (UCT) if
\begin{align}
\notag
0 &\to \mathrm{Ext}(K_*(A),K_{*+1}(B)) \to \\
&\qquad\qquad KK(A,B) \to \mathrm{Hom}(K_*(A),K_*(B)) \to 0
\end{align}
is an exact sequence, for every $\sigma$-unital $\mathrm C^*$-algebra $B$.
\end{defn}

In \cite{RS:DMJ}, Rosenberg and Schochet showed that a large class of separable $\mathrm{C}^*$-algebras satisfy the UCT (the collection of $\mathrm{C}^*$-algebras satisfying the UCT is now known as the UCT class), and established closure properties; their work shows that the UCT class consists of precisely those separable $\mathrm{C}^*$-algebras which are $KK$-equivalent to abelian $\mathrm{C}^*$-algebras (this precise statement can be found as \cite[Theorem 23.10.5]{B:KThy}). In particular, note that if $A$ satisfies the UCT, then so too does the unitisation $C_0((0,1),A)^\sim$ of the suspension of $A$.\footnote{If $A$ is $KK$-equivalent to the abelian $\mathrm{C}^*$-algebra $B$, then $C_0((0,1),A)^\sim$ is $KK$-equivalent to the abelian $\mathrm{C}^*$-algebra $C_0((0,1),B)^\sim$ (for example, by \cite[19.1.2 (c) and (d)]{B:KThy}).}

It was shown by Tu in \cite[Lemma 3.5 and Proposition 10.7]{Tu:KT} that the $\mathrm C^*$-algebra of each amenable \'etale groupoid satisfies the UCT.  In particular, this includes the group $\mathrm C^*$-algebras of countable discrete amenable groups.

We end this section by setting out our conventions regarding matrix amplifications which will appear frequently throughout the paper.
\begin{notation}\label{notation}
 For a $^*$-algebra $A$ and $n\in\mathbb N$, we freely identify $A\otimes M_n$ and $M_n(A)$.   Given $n$ elements, $a_1,\dots,a_n\in A$, we write $a_1\oplus a_2\oplus\cdots\oplus a_n$ for the diagonal matrix in $M_n(A)$ with entries $a_1,a_2,\dots,a_n$. Similar conventions apply to maps: given $\theta_1,\dots,\theta_n:A\to B$, write $\theta_1\oplus\cdots\oplus\theta_n:A\to M_n(B)$ for the map given by $a\mapsto \theta_1(a)\oplus\cdots\oplus \theta_n(a)$.  When the elements involved are constant, we sometimes use $a^{\oplus n}$ to denote the diagonal element of $M_n(A)$ with constant entry $a$, and likewise, given $\theta:A \to B$, we write $\theta^{\oplus n}$ for the map $a\mapsto \theta(a)^{\oplus n}$.
\end{notation}
\section{Two Lebesgue trace cones}\label{Sect:LebCones}

The central purpose of this section is to produce the two Lebesgue trace cones over $A$ in $\Q_\omega$ (obtained in Lemma \ref{MapsExist} below) with unital sum, which are the inputs into the ``stable uniqueness across the interval'' procedure.  We start by using Cuntz subequivalence, originating in \cite{C:MA} and further developed in \cite{R:JFA2,CEI:Crelle} (we refer to \cite{APT:Contemp} for a full account of the Cuntz semigroup), and strict comparison of positive elements (which has its origins in \cite{B:LMSnotes}) to record two technical observations.

Recall that for $a,b\in A_+$, $a$ is said to be \emph{Cuntz below} $b$ if there exists a sequence $(x_n)$ of elements of $A$ such that $a=\lim_n x_nbx_n^*$; $a$ is called \emph{Cuntz equivalent} to $b$ if $a$ is Cuntz below $b$ and $b$ is Cuntz below $a$.  For $\varepsilon>0$, $(a-\varepsilon)_+$ denotes the functional calculus output given by applying the function $h(t)=\max\{t-\varepsilon,0\}$ to $a$. With this notation, $a$ is Cuntz below $b$ in $A$ if and only if for all $\varepsilon>0$, there exists $v\in A$ such that $(a-\varepsilon)_+=vbv^*$ (see \cite[Proposition 2.4]{R:JFA2}, which shows that $a$ is Cuntz below $b$ if and only if for all $\varepsilon>0$, there exist $\delta>0$ and $w\in A$ such that $(a-\varepsilon)_+=w(b-\delta)_+w^*$).

A $\mathrm{C}^*$-algebra $A$ is said to have \emph{strict comparison of positive elements with respect to bounded traces}, if whenever $k \in \mathbb N$ and $a,b \in (A \otimes M_k)_+$ satisfy
\begin{align}
d_{\tau \otimes \mathrm{Tr}_k}(a) < d_{\tau \otimes \mathrm{Tr}_k}(b)
\end{align}
for all $\tau\in T(A)$ (where $d_\tau(a) := \lim_{n\to\infty} \tau(a^{1/n})$), it follows that $a$ is Cuntz below $b$ in $A\otimes M_k$.\footnote{Strict comparison is usually defined using densely defined lower semicontinuous $2$-quasitraces, but we do not need this version of the definition.} The functions $d_{\tau\otimes\mathrm{Tr}_k}$ above provide functionals on the Cuntz semigroup (\cite[Proposition 4.2]{ERS:AJM}); in particular\begin{align}\label{dtaustate}
d_{\tau\otimes\mathrm{Tr}_k}(a)+d_{\tau\otimes\mathrm{Tr}_k}(b)&=d_{\tau\otimes\mathrm{Tr}_k}(a+b),\quad a,b\in (A\otimes M_k)_+,\ ab=0;\\
d_{\tau\otimes\mathrm{Tr}_k}(a)&\leq d_{\tau\otimes\mathrm{Tr}_k}(b),\quad a\text{ is Cuntz below }b\text{ in }A\otimes M_k.\label{dtaustate.2}
\end{align}
The key example for us is that $\Q_\omega$ has strict comparison with respect to its unique bounded trace. This follows from the corresponding fact for $\Q$ (by \cite[Theorem 5.2 (a)]{R:JFA2})\footnote{Up to scalar multiplication the only lower semicontinuous dimension function on $\Q$ arises from the unique trace.} and \cite[Lemma 1.23]{BBSTWW:arXiv}, for example, which shows that this property passes to ultraproducts.

Our first use of Cuntz comparison is the following, now-standard consequence of Ciuperca and Elliott's classification of $^*$-homomorphisms from $C_0((0,1])$ to stable rank one $\mathrm C^*$-algebras by their Cuntz semigroup data, from \cite[Theorem 4]{CE:IMRN} (and extended in \cite{RS:JFA} and \cite{JST:arXiv}).
We shall use this result in Lemma \ref{MapsExist} to obtain our second Lebesgue trace cone as a unitary conjugate of the first cone.
The argument below is very similar to the proof of \cite[Lemma 5.1]{SWW:Invent}.
\begin{lemma}\label{Lem:CE}\begin{enumerate}[(i)]
\item\label{Lem:CE.1} Let $A$ be a unital $\mathrm{C}^*$-algebra of stable rank one and strict comparison of positive elements with respect to bounded traces. Suppose that $a_1,a_2\in A$ are positive contractions such that
\begin{equation} \tau(h(a_1)) = \tau(h(a_2)) > 0,\quad \tau\in T(A),\ h\in C_0((0,1])_+\setminus\{0\}. \label{e2.4}\end{equation}
Then $a_1$ and $a_2$ are approximately unitarily equivalent.
\item Suppose that $a_1,a_2\in \Q_\omega$ are positive contractions such that
\begin{equation} \tau_{\Q_\omega}(h(a_1)) = \tau_{\Q_\omega}(h(a_2)) > 0,\quad h\in C_0((0,1])_+\setminus\{0\}. \end{equation}
Then $a_1$ and $a_2$ are unitarily equivalent.\label{Qfacts.4}
\end{enumerate}
\end{lemma}
\begin{proof}
(\ref{Lem:CE.1}) By \cite[Theorem 4]{CE:IMRN}, it suffices to show that $h(a_1)$ is Cuntz equivalent to $h(a_2)$ for all nonzero $h\in C_0((0,1])_+$.
For such $h$, let $\varepsilon > 0$, and define
\begin{equation} U:= \{t \in (0,1] \mid 0 < h(t) < \varepsilon\}. \end{equation}
This open set is nonempty and therefore there exists a nonzero function $g \in C_0(U)_+$ of norm one. Then for any trace $\tau\in T(A)$, we have $d_\tau(g(a_1)) \geq \tau(g(a_1)) > 0$ by (\ref{e2.4}). Since $(h(a_1)-\varepsilon)_+$ is orthogonal to $g(a_1)$, 
\begin{align}
\notag
d_\tau((h(a_1)-\varepsilon)_+) &\stackrel{\phantom{\eqref{dtaustate}}}{<} d_\tau((h(a_1)-\varepsilon)_+) + d_\tau(g(a_1)) \\
&\stackrel{\eqref{dtaustate}}{=}d_\tau((h(a_1)-\varepsilon)_++g(a_1))\notag\\
&\stackrel{\eqref{dtaustate.2}}{\leq} d_\tau(h(a_1))\notag\\
&\stackrel{\eqref{e2.4}}{=} d_\tau(h(a_2)), \quad \tau \in T(A)
\end{align}
as $(h(a_1)-\varepsilon)_++g(a_1)$ is Cuntz below $h(a_1)$.
Since $A$ has strict comparison of positive elements by bounded traces, it follows that $(h(a_1)-\varepsilon)_+$ is Cuntz below $h(a_2)$; as $\varepsilon$ is arbitrary, $h(a_1)$ is Cuntz below $h(a_2)$.
Symmetrically, we also obtain that $h(a_2)$ is Cuntz below $h(a_1)$, as required.

(\ref{Qfacts.4}) Since stable rank one passes to ultrapowers (see \cite[Lemma 2.4]{SWW:Invent}, a simple modification of \cite[Lemma 19.2.2 (1)]{L:Book} to ultrapowers), $\Q_\omega$ has stable rank one and strict comparison with respect to its unique trace $\tau_{\Q_\omega}$. The result now follows from (\ref{Lem:CE.1}), noting that a standard reindexing argument shows that two positive elements in an ultrapower are unitarily equivalent if and only if they are approximately unitarily equivalent (see \cite[Lemma 1.17 (i)]{BBSTWW:arXiv}, for example).
\end{proof}

Our second application of strict comparison is a standard computation familiar to experts; it will be used in the proof of Theorem \ref{thm:MainThm} in order to verify $\Delta$-fullness as in Definition \ref{def:Control}.

\begin{lemma}
\label{lem:TrFull}
Let $A$ be a unital $\mathrm C^*$-algebra with strict comparison of positive elements by bounded traces, such as $\Q_\omega$.
Let $a \in A_+^1$. If $m \in \mathbb N$ satisfies $\tau(a) > \tfrac2{m}$ for all $\tau\in T(A)$ then there exist $m^2$ contractions $v_1,\dots,v_{m^2}\in A$ such that
\begin{equation} 1_A= \sum_{i=1}^{m^2} v_iav_i^*. \end{equation}
\end{lemma}

\begin{proof}
Set $\delta:=\frac12 \min_{\tau\in T(A)}\tau(a)$ so that $m\delta>1$.\footnote{This minimum exists as $T(A)$ is compact.} For $\tau\in T(A)$, we have
\begin{equation} d_\tau((a-\delta)_+^2)\geq \tau((a-\delta)_+)\geq\tau(a)-\delta\geq \delta, \end{equation}
so that \begin{equation}
d_{\tau\otimes \mathrm{Tr}_m}(((a-\delta)^2)^{\oplus m})>d_{\tau\otimes \mathrm{Tr}_m}(1_A\oplus 0_A^{\oplus(m-1)}).
\end{equation}
Strict comparison of positive elements shows that $1_A\oplus 0_A^{\oplus (m-1)}$ is Cuntz below $((a-\delta)^2)^{\oplus m}$ in $A\otimes M_m$ so that (as $(1_A-\varepsilon)_+=(1-\varepsilon)1_A$), there exist $b_1,\dots,b_m \in A$ such that
\begin{equation} 1_A = \sum_{j=1}^m b_j(a-\delta)_+^2b_j^*, \end{equation}
by \cite[Proposition 2.4]{R:JFA2}, as discussed after the definition of Cuntz comparison above. In particular, $\|b_j(a-\delta)_+\| \leq 1$ for each $j$.

Define $h \in C([0,1])$ by
\begin{equation} h(t) := \begin{cases} \frac1{\sqrt{t}}, \quad &t \geq \delta; \\ \frac t{\delta^{3/2}},\quad &t \in [0,\delta]. \end{cases} \end{equation}
Note that $h(t)^2t=1$ for $t \geq \delta$, so $(a-\delta)_+h(a)^2a=(a-\delta)_+$.  Writing $c_j:=m^{-1/2}b_j(a-\delta)_+h(a)$, which is a contraction as $\|c_j\|\leq m^{-1/2}\|h(a)\|\leq(m\delta)^{-1/2}<1$, we have
\begin{align}
\notag
1_A &= \sum_{j=1}^m b_j(a-\delta)_+^2b_j^* \\
&= m\sum_{j=1}^mc_jac_j^*,
\end{align}
as required.
\end{proof}

Our two Lebesgue trace cones are obtained from the tracially large order zero maps constructed in \cite[Proposition 3.2]{SWW:Invent} using Connes' uniqueness of the injective II$_1$ factor (\cite{C:Ann}). The form we need is given in Proposition \ref{SWW-Trace} below.

\begin{defn}[{\cite[Definition 1.3]{WZ:MJM}}]
\label{def:ord0}
Let $A,B$ be $\mathrm C^*$-algebras. A completely positive map $\phi:A \to B$ is said to be \emph{order zero} if for every $a,b \in A_+$ with $ab=0$, one has $\phi(a)\phi(b)=0$.
\end{defn}

Since traces on a unital $\mathrm{C}^*$-algebra $A$ form a weak$^*$-compact convex set, $T(A)$ is the weak$^*$-closed convex hull of extremal traces on $A$ by the Krein--Milman theorem. Recall that $\tau\in T(A)$ is extremal if and only if the associated GNS representation $\pi_\tau(A)''$ is a factor (\cite[Theorem 6.7.3]{Dixmier:C*Book}).

\begin{prop}\label{SWW-Trace}
Let $A$ be a separable, unital and nuclear $\mathrm{C}^*$-algebra and let $\tau_A$ be a trace on $A$.
Then there exists a c.p.c.\ order zero map $\Omega:A\rightarrow\Q_\omega$ such that
\begin{equation}\label{SWW-Trace.1}
\tau_{\Q_\omega}(\Omega(a)\Omega(1_A)^{n-1}) =\tau_A(a),\quad a\in A,\ n\in\mathbb N.
\end{equation}
\end{prop}

\begin{proof}
The proof of \cite[Proposition 3.2]{SWW:Invent}, establishes the proposition in the case that $\tau_A$ is an extremal trace on $A$  for which the GNS representation $\pi_{\tau_A}(A)''$ is type II$_1$ (and hence a copy of the hyperfinite II$_1$ factor) except that \eqref{SWW-Trace.1} is stated by the conditions that $\tau_{\Q_\omega}\circ\Omega=\tau_A$ and $1_{\Q_\omega}-\Omega(1_A)$ is in the kernel associated to $\tau_{\Q_\omega}$.\footnote{The key point is that the argument in \cite[Proposition 3.2]{SWW:Invent} is readily seen to be valid when $\pi_{\tau_A}(A)''$ is a copy of the hyperfinite II$_1$ factor; this is the only use of the hypothesis in the statement of \cite[Proposition 3.2]{SWW:Invent} that $A$ has no finite dimensional quotients.}
To obtain \eqref{SWW-Trace.1} for $n>1$ from these conditions, by the Cauchy--Schwarz inequality we have
\begin{align}
\notag
|\tau_{\Q_\omega}(y(1_{\Q_\omega}-\Omega(1_A)))| &\leq \tau_{\Q_\omega}(y^*y)^{1/2}\tau_{\Q_\omega}((1_{\Q_\omega}-\Omega(1_A))^2)^{1/2} \\
&=0,\quad y\in\Q_\omega,
\end{align}
as $1_{\Q_\omega}-\Omega(1_A)$ lies in the kernel associated $\tau_{\Q_\omega}$. Setting $y:=\Omega(a)\Omega(1_A)^{n-2}$ allows us to prove \eqref{SWW-Trace.1} inductively.

In contrast, if $\tau_A$ is an extremal trace on $A$ such that $\pi_{\tau_A}(A)''$ is a type I factor, we may define $\Omega$ to be the composition of the unital embedding $\pi_{\tau_A}:A\rightarrow \pi_{\tau_A}(A)''$ with some unital (and necessarily trace preserving) embedding $\pi_{\tau_A}(A)''\hookrightarrow \Q_\omega$. This is a $^*$-homomorphism with $\tau_{\Q_\omega}\circ\Omega=\tau_A$. Thus the proposition holds when $\tau_A$ is an extremal trace on $A$.  

The case of a general trace $\tau_A$ on $A$ follows by approximating by convex combinations and a standard reindexing argument, as we now explain.
The reindexing argument we use is Kirchberg's $\varepsilon$-test from \cite[Lemma A.1]{K:Abel} or \cite[Lemma 3.1]{KR:Crelle}.  For $i\in \mathbb N$, write $X_i$ for the collection of $^*$-linear maps $A\rightarrow \Q$. Then, as set out in \cite[Lemma 1.12]{BBSTWW:arXiv}, there exists a countable collection of functions $f^{(k)}_i:X_i\rightarrow [0,\infty]$ indexed by $k,i\in\mathbb N$, such that a sequence $(\phi_i)_{i=1}^\infty$ in $\prod_{i=1}^\infty X_i$ (the set product) induces a c.p.c.\ order zero map $A\rightarrow\Q_\omega$ if and only if $\lim_{i\rightarrow\omega}f^{(k)}_i(\phi_i)=0$ for all $k\in\mathbb N$.  Fixing a countable dense subset $(a_k)_{k=1}^\infty$ in $A$, define functions $g_i^{(k,n)}:X_i\rightarrow[0,\infty]$ for $i,k,n\in\mathbb N$ by
\begin{equation}
g_i^{(k,n)}(\phi_i):=|\tau_\Q(\phi_i(a_k)\phi_i(1_A)^{(n-1)})-\tau_A(a_k)|.
\end{equation}
In this way, a sequence $(\phi_i)_{i=1}^\infty$ in $\prod_{i=1}^\infty X_i$ induces a c.p.c.\ order zero map $A\rightarrow\Q_\omega$ satisfying \eqref{SWW-Trace.1} if and only if 
\begin{equation}\label{SWW-Trace.3}
\lim_{i\rightarrow\omega}f^{(k)}_i(\phi_i)=\lim_{i\rightarrow\omega}g^{(k,n)}_i(\phi_i)=0,
\end{equation}
for all $k,n\in\mathbb N$.

Now fix $k\in\mathbb N$ and $\varepsilon>0$.  Find $m\in\mathbb N$, positive numbers $\lambda_1,\dots,\lambda_m$ with $\sum_{j=1}^m\lambda_j=1$, and extremal traces $\tau_1,\dots,\tau_m$ on $A$ such that 
\begin{equation}\label{SWW-Trace.2}
\Big|\sum_{j=1}^m\lambda_j\tau_j(a_r)-\tau_A(a_r)\Big|<\varepsilon,\quad r=1,\dots,k.
\end{equation}

Choose pairwise orthogonal projections $p_1,\dots,p_m\in\Q_\omega$ which sum to $1_{\Q_\omega}$ and satisfy $\tau_{\Q_\omega}(p_j)=\lambda_j$ for each $j$. Then $p_j\Q_\omega p_j\cong \Q_\omega$ by Proposition \ref{prop:Qfacts} (\ref{prop:Qfacts.1}).  By the first three paragraphs of this proof, for each $j$ we can find a c.p.c.\ order zero map $\Omega_j:A\rightarrow p_j\Q_\omega p_j$ such that
\begin{equation}
\tau_{p_j\Q_\omega p_j}(\Omega_j(a)\Omega_j(1_A)^{n-1})=\tau_j(a),\quad a\in A,\ n\in\mathbb N.
\end{equation}
Define $\Omega:A\rightarrow\Q_\omega$ by $\Omega(a)=\sum_{j=1}^m\Omega_j(a)$ for $a\in A$. Since the $p_j$ are pairwise orthogonal, $\Omega$ is a c.p.c.\ order zero map, and it satisfies
\begin{align}
\notag
\tau_{\Q_\omega}(\Omega(a)\Omega(1_A)^{n-1})&= \sum_{j=1}^m \lambda_j \tau_{p_j\Q_\omega p_j}(\Omega_j(a)\Omega_j(1_A)^{n-1}) \\
&= \sum_{j=1}^m\lambda_j\tau_j(a),\quad a\in A,\ n\in\mathbb N.
\end{align}
Combining this with (\ref{SWW-Trace.2}), it follows that any lift of $\Omega$ to a sequence $(\phi_i)_{i=1}^\infty$ in $\prod_{i=1}^\infty X_i$ has $\lim_{i\rightarrow\omega}g^{(r,n)}_i(\phi_i)\leq\varepsilon$ for all $n\in\mathbb N$ and $r=1,\dots,k$ and $\lim_{i\rightarrow\omega}f^{(l)}_i(\phi_i)=0$ for all $l$.  Thus by Kirchberg's $\varepsilon$-test, there exists a sequence $(\phi_i)_{i=1}^\infty$ in $\prod_i X_i$ satisfying (\ref{SWW-Trace.3}), and hence providing the required c.p.c.\ order zero map $\Omega:A\rightarrow\Q_\omega$ satisfying (\ref{SWW-Trace.1}).
\end{proof} 

Our two Lebesgue trace cones over $A$ will be constructed so that their scalar parts agree with a $^*$-homomorphism on $C([0,1])$.  We encapsulate this property in the following definition.
\begin{defn}
\label{def:Compatible}
Let $A$ and $E$ be unital $\mathrm{C}^*$-algebras and let $I\subset\mathbb [0,1]$ be an interval. Let $\theta:C([0,1])\rightarrow E$ be a unital $^*$-homomorphism.  We say that a $^*$-homomorphism $\nu:C_{0}(I,A) \rightarrow E$ is \emph{compatible} with $\theta$ if \begin{equation}
\nu(hx)=\theta(h)\nu(x)=\nu(x)\theta(h),\quad h\in C([0,1]),\ x \in C_{0}(I,A).
\end{equation}
\end{defn}

The essential feature of order zero maps we need in the next proposition is their correspondence with $^*$-homomorphisms from cones as in \cite[Corollary 3.1]{WZ:MJM}: a c.p.c.\ map $\Omega:A\rightarrow B$ is order zero if and only if there exists a ${}^*$-homomorphism $\pi_\Omega:C_0((0,1]) \otimes A \to B$ such that, 
\begin{equation} \label{ord0.2}
\Omega(a)=\pi_\Omega(\id_{[0,1]} \otimes a), \quad a \in A.\footnote{We use $\id_{[0,1]}$ for the identity function on $[0,1]$ satisfying $\id_{[0,1]}(t)=t$. This should not be confused with $1_{C([0,1])}$, the unit of $C([0,1])$.}
\end{equation}

Given a bounded interval $I\subset\mathbb R$, we write $\tau_{\leb}$ for the tracial functional on $C_0(I)$ induced by Lebesgue measure on $I$, i.e.
\begin{equation}
\label{eq:LebTraceDef1}
\tau_{\leb}(h) := \int_I h(t)\,dt,\quad h\in C_0(I). \end{equation}
This is a trace when $I$ has length $1$ (in particular for $I=[0,1]$, $(0,1]$ or $[0,1)$).  A positive contraction $a$ in a unital $\mathrm{C}^*$-algebra $A$ is said to have Lebesgue spectral measure with respect to $\tau_A\in T(A)$ if
\begin{equation}\label{e1.38}
\tau_A(h(a))=\tau_{\leb}(h), \quad h \in C([0,1]).
\end{equation}
Note that $a$ has Lebesgue spectral measure with respect to $\tau_A$ if and only if $1_A-a$ does.

\begin{lemma}\label{MapsExist}
Let $A$ be a separable, unital and nuclear $\mathrm{C}^*$-algebra and $\tau_A\in T(A)$. Then there are $^*$-homomorphisms 
\begin{equation}
\Theta:C([0,1])\rightarrow\Q_\omega,\ \grave{\Phi}:C_0([0,1),A)\rightarrow \Q_\omega,\ \acute{\Phi}:C_0((0,1],A)\rightarrow\Q_\omega
\end{equation}
such that:
\begin{enumerate}[(i)]
\item $\Theta$ is unital and $\grave{\Phi},\acute{\Phi}$ are compatible with $\Theta$;\label{MapsExist.1}
\item $\tau_{\Q_\omega}\circ\grave\Phi = \tau_{\leb} \otimes \tau_A$;  \label{MapsExist.2}
\item $\tau_{\Q_\omega}\circ\acute\Phi = \tau_{\leb} \otimes \tau_A$.\label{MapsExist.3}
\end{enumerate}
\end{lemma}

\begin{proof}
By Proposition \ref{SWW-Trace}, there exists a c.p.c.\ order zero map $\Omega:A\rightarrow\Q_\omega$ satisfying \eqref{SWW-Trace.1}.
Let $\pi_\Omega:C_0((0,1],A) \to \Q_\omega$ be the $^*$-homomorphism associated to $\Omega$ satisfying (\ref{ord0.2}). Let $k\in \Q_+$ be a positive contraction with spectrum $[0,1]$ such that $k$ has Lebesgue spectral measure with respect to $\tau_{\Q}$. Let $\pi_k:C_0((0,1]) \to \Q$ denote the $^*$-homomorphism defined by $\pi_k(h) := h(k)$.

Define $\alpha:(0,1] \times (0,1] \to (0,1]$ by $\alpha(s,t) := st$ and use this to define $\acute\Phi:C_0((0,1],A) \rightarrow \Q \otimes \Q_\omega$ by
\begin{equation}
\acute{\Phi}(x) := (\pi_k \otimes \pi_\Omega)(x \circ \alpha),\quad x\in C_0((0,1],A).\footnote{Alternatively $\acute{\Phi}$ can be defined as the $^*$-homomorphism $C_0((0,1],A) \to \Q\otimes\Q_\omega$ associated by (\ref{ord0.2}) to the c.p.c.\ order zero map $A\rightarrow \Q\otimes\Q_\omega$ given by $a\mapsto k\otimes\Omega(a)$.}
\end{equation}
An isomorphism $\Q\otimes\Q\cong\Q$ can be used to induce a unital embedding $\Q\otimes\Q_\omega\to\Q_\omega$. Hence we can regard $\acute{\Phi}$ as taking values in $\Q_\omega$. Define $\Theta:C([0,1])\rightarrow \Q_\omega$ to be the unitisation of $\acute{\Phi}|_{C_0((0,1],\mathbb C1_A)}$.  In this way $\acute{\Phi}$ is certainly compatible with $\Theta$.

For $a\in A$ and $n\in\mathbb N$, the definition of $\acute\Phi$ ensures that
\begin{align}
\notag
\acute{\Phi}(\id_{[0,1]}^n\otimes a)&\stackrel{\phantom{\eqref{ord0.2}}}{=}(\pi_k \otimes \pi_\Omega)(\id_{[0,1]}^n \otimes \id_{[0,1]}^n \otimes a) \\
&\stackrel{\eqref{ord0.2}}{=}k^n \otimes \left(\Omega(a)\Omega(1_A)^{n-1}\right).\label{e1.40}
\end{align}
Thus, 
\begin{eqnarray}
\tau_{\Q_\omega}(\acute{\Phi}(\id_{[0,1]}^n\otimes a))&\stackrel{\eqref{e1.40}}{=}&\tau_{\Q}(k^n)\tau_{\Q_\omega}(\Omega(a)\Omega(1_A)^{n-1})\nonumber\\
&\stackrel{\eqref{SWW-Trace.1},\eqref{e1.38}}{=}
& \tau_{\leb}(\id_{[0,1]}^n)\tau_A(a).
\end{eqnarray}
By linearity and density,
\begin{equation}\label{e1.44}
\tau_{\Q_\omega}\circ \acute{\Phi}=\tau_{\leb} \otimes \tau_A
\end{equation}
proving (\ref{MapsExist.3}).

Since the positive contraction $\Theta(\id_{[0,1]})=\acute{\Phi}(\id_{[0,1]}\otimes 1_A)$ has Lebesgue spectral measure with respect to $\tau_{\Q_\omega}$, so too does $1_{\Q_\omega}-\Theta(\id_{[0,1]})$. Therefore, by Lemma \ref{Lem:CE} (\ref{Qfacts.4}), there is a unitary $u\in \Q_\omega$ such that  
\begin{equation}\label{e1.45}
\Theta(1_{C([0,1])}-\id_{[0,1]})=1_{\Q_\omega}-\acute{\Phi}(\id_{[0,1]})=u\acute{\Phi}(\id_{[0,1]})u^*.
\end{equation}
Define $\grave{\Phi}$ to be the composition
\begin{equation}\label{e1.46}
\grave{\Phi}:C_0([0,1),A)\stackrel{\sigma}{\longrightarrow}C_0((0,1],A)\stackrel{\acute{\Phi}}{\longrightarrow}\Q_\omega\stackrel{\mathrm{Ad}(u)}{\longrightarrow}\Q_\omega,
\end{equation}
where $\sigma$ is the flip map given by $\sigma(x)(t):=x(1-t)$ for $x\in C_0([0,1),A)$, $t\in (0,1]$.
By construction, $\sigma(1_{C([0,1])}-\id_{[0,1]})=\id_{[0,1]}$, and hence
\begin{eqnarray}
\notag
\grave{\Phi}(1_{C([0,1])}-\id_{[0,1]})&=&u\acute{\Phi}(\id_{[0,1]})u^* \\
&\stackrel{\eqref{e1.45}}{=}&\Theta(1_{C([0,1])}-\id_{[0,1]}).
\end{eqnarray}
Since $\Theta$ is unital and the cone $C_0([0,1))$ is generated by $1_{C([0,1])}-\id_{[0,1]}$, it follows that
\begin{equation}
\grave{\Phi}|_{C_0([0,1),\mathbb C1_A)}=\Theta|_{C_0([0,1))},
\end{equation}
and so $\grave{\Phi}$ is compatible with $\Theta$, establishing (\ref{MapsExist.1}).

For (\ref{MapsExist.2}), we compute 
\begin{eqnarray}
\tau_{\Q_\omega}\circ\grave{\Phi}&\stackrel{\eqref{e1.46}}{=}&\tau_{\Q_\omega}\circ \mathrm{Ad}(u) \circ \acute{\Phi} \circ \sigma \nonumber\\
&=&\tau_{\Q_\omega} \circ \acute{\Phi} \circ \sigma\nonumber\\
&\stackrel{\eqref{e1.44}}{=}&(\tau_{\leb} \otimes \tau_A) \circ \sigma \nonumber\\
&=&\tau_{\leb}\otimes \tau_A,
\end{eqnarray}
as integration with respect to Lebesgue measure on $[0,1]$ is certainly invariant under flipping the interval.
\end{proof}

\section{A controlled stable uniqueness theorem}\label{Sect:SU}

\noindent
This section contains a small modification of a stable uniqueness theorem from \cite{DE:PLMS} (which was in turn inspired by Lin's paper \cite{Lin:JOT}).
A similar modification was given by Lin in \cite[Theorem 5.9]{Lin:apprUCT}; we have chosen to give an argument (based on \cite{DE:PLMS}) to make it transparent how the UCT hypothesis is used in Theorem \ref{thm:MainThm}. Recall that a $^*$-homomorphism $\iota:A\rightarrow B$ is said to be \emph{totally full}\footnote{Unital totally full $^*$-homomorphisms are called unital full embeddings in \cite{DE:PLMS}.} if for every nonzero $a\in A$, $\iota(a)$ is full in $B$, i.e. $\overline{\text{span}\, B\iota(a)B}=B$.  With this, our starting point is the stable uniqueness theorem below, from \cite{DE:PLMS}. For a separable $\mathrm C^*$-algebra $A$ and a $\sigma$-unital $\mathrm C^*$-algebra $B$, the group $KK_{\mathrm{nuc}}(A,B)$ is Skandalis' modification of $KK(A,B)$ obtained from the Cuntz picture by working only with strictly nuclear maps and homotopies (\cite[Section 2]{S:KThy}) and, as such, when $A$ or $B$ is nuclear, the canonical homomorphism $KK_{\mathrm{nuc}}(A,B) \to K(A,B)$ is an isomorphism.  This also holds when $A$ is $KK$-equivalent to a nuclear $\mathrm{C}^*$-algebra (\cite[Propositions 3.2 and 3.3]{S:KThy}), and hence whenever $A$  satisfies the UCT.

\begin{thm}[{Dadarlat--Eilers; cf.\ \cite[Theorem 4.5]{DE:PLMS}}]\label{DE:SU}
Let $A,B$ be unital $\mathrm{C}^*$-algebras with $A$ separable. Let $\iota:A\rightarrow B$ be a totally full unital $^*$-homo\-morphism, and suppose that $\phi,\psi:A\rightarrow B$ are nuclear $^*$-homo\-morphisms inducing the same class in $KK_{\mathrm{nuc}}(A,B)$ and such that $\phi(1_A)$ is unitarily equivalent to $\psi(1_A)$. Then, for any finite subset $\mathcal G\subset A$ and $\delta>0$, there exist $n\in\mathbb N$ and a unitary $u\in M_{n+1}(B)$ such that
\begin{equation}
\|u(\phi(a)\oplus\iota^{\oplus n}(a))u^*-(\psi(a)\oplus\iota^{\oplus n}(a))\|<\delta,\quad a\in \mathcal G.
\end{equation}
\end{thm}
\begin{proof}
This is a special case of \cite[Theorem 4.5]{DE:PLMS}.  The totally full map $\iota$ induces a representation $\gamma:A\rightarrow M(\mathcal K\otimes B)$ by $\gamma(a)=\iota^{\oplus \infty}(a)$, for $a\in A$, which is nuclearly absorbing in the sense of \cite[Definition 2.4]{DE:PLMS} by \cite[Theorem 2.22]{DE:PLMS}. The representation $\gamma$ is quasidiagonal\footnote{Here the term quasidiagonal representation is used in the setting of representations on Hilbert modules; it differs from the use of the term quasidiagonality elsewhere in this paper. The quasidiagonality of $\gamma$ is witnessed by the projections $e_n\otimes 1_B \in\mathcal K\otimes B$ which commute exactly with $\gamma(A)$, where $(e_n)$ is an increasing sequence of projections with $e_n$ of rank $n$, converging strictly to the identity operator. Thus, identifying $e_n(M(\mathcal K\otimes B))e_n$ with $M_n(B)$, one has $\gamma_n(a)=e_n(\gamma(a))e_n=\iota^{\oplus n}(a)$ for $a\in A$.} in the sense of \cite[Definition 2.5]{DE:PLMS} with a quasidiagonalisation of the form $\gamma_n=\iota^{\oplus n}:A\rightarrow M_n(B)$, and so we are in a situation covered by \cite[Theorem 4.5]{DE:PLMS}, yielding the result.
\end{proof}

In the above stable uniqueness theorem, the $n$ depends not only on $\mathcal G$ and $\delta$, but also on the maps $\iota,\phi$ and $\psi$. From this, Dadarlat and Eilers obtain a stable uniqueness theorem for simple domains $A$ (\cite[Theorem 4.12]{DE:PLMS}), in which $n$ depends only on $\mathcal G$ and $\delta$ and not on $\iota,\phi$ and $\psi$ by using a sequence of counterexamples and Theorem~\ref{DE:SU}. This produces two sequences of $^*$-homomorphisms which agree in $\prod_n KK_{\mathrm{nuc}}(A,B_n)$. %(in our application they are even zero homotopic and hence yield the zero element in $\prod_n KK_{\mathrm{nuc}}(A,B_n)$). 
However, a sequence of homotopies witnessing this agreement does not necessarily give rise to a single continuous homotopy of sequences: the parameter speeds might increase too quickly.  The UCT resolves this difficulty as it ensures that the map 
\begin{equation}\label{InjectiveEquation}
KK_{\mathrm{nuc}}\Big(A,\prod_n B_n\Big)\rightarrow  \prod_nKK_{\mathrm{nuc}}(A,B_n)
\end{equation}
is injective,\footnote{One must also have some mild control over the structure of the target algebras.} and the sequences have the same class in $KK_{\mathrm{nuc}}(A,\prod_n B_n)$. Now Dadarlat and Eilers apply their \cite[Theorem 4.5]{DE:PLMS} to reach a contradiction. We note in passing that in our application of stable uniqueness to prove Theorem \ref{thm:MainThm}, the two sequences in $\prod_n KK_{\mathrm{nuc}}(A,B_n)$ are not only trivial in $KK$, but come from genuinely zero homotopic maps; a priori this is still not enough to show that they agree in $KK_{\mathrm{nuc}}(A,\prod_n B_n)$ without the UCT.

%In the above stable uniqueness theorem, the $n$ depends not only on $\mathcal G$ and $\delta$, but also on the maps $\iota,\phi$ and $\psi$. From this, Dadarlat and Eilers obtain a stable uniqueness theorem for simple domains $A$ (\cite[Theorem 4.12]{DE:PLMS}), in which $n$ depends only on $\mathcal G$ and $\delta$ and not on $\iota,\phi$ and $\psi$ by using a sequence of counterexamples and Theorem~\ref{DE:SU}. This argument essentially relies on injectivity of maps of the form
%\begin{equation}\label{InjectiveEquation}
%KK_{\mathrm{nuc}}\Big(A,\prod_n B_n\Big)\rightarrow  \prod_nKK_{\mathrm{nuc}}(A,B_n),
%\end{equation}
%which is not automatic.  The problem is that a sequence of homotopies witnessing the agreement of two sequences in $\prod_n KK_{\mathrm{nuc}}(A,B_n)$, does not necessarily give rise to a single continuous homotopy of sequences: the parameter speeds might increase quickly.  It is for this purpose (and this purpose only) that the UCT enters the argument.\footnote{One must also have some mild control over the structure of the target algebras.} Note that in our application of stable uniqueness to prove Theorem \ref{thm:MainThm}, the two sequences in $\prod_n KK_{\mathrm{nuc}}(A,B_n)$ are not only trivial in $KK$, but come from genuinely zero homotopic maps, but a-priori this is still not enough to show that they agree in $KK_{\mathrm{nuc}}(A,\prod_n B_n)$ without the UCT.

For completeness, we give a full account of Dadarlat and Eilers' argument below, which is the basis for our generalisation.   We first state the required structural conditions on the target algebras:

\begin{defn}[{\cite[Definition 4.9]{DE:PLMS}}]
A $\mathrm C^*$-algebra $B$ is an \emph{admissible target algebra of finite type}\footnote{\cite{DE:PLMS} also considers admissible target algebras of infinite type; we do not need these here.}  if $B$ is unital, $B$ has real rank zero, and the following conditions hold:
\begin{enumerate}[(i)]
\item for $k \in \mathbb N$ and projections $p,q \in B \otimes M_k$, if $[p]=[q]$ in $K_0(B)$ then $p\oplus 1_B$ is Murray--von Neumann equivalent to $q \oplus 1_B$; \label{ata.1}
\item the canonical map from unitaries in $B$ to $K_1(B)$ is surjective;\label{ata.2}
\item for any $r \in K_0(B)$, if $nr\geq 0$ for some $n\in\mathbb N$, then $[1_B]+r\geq 0$;\label{ata.3}
\item for any $r \in K_0(B)$ and any $n\in\mathbb N$, there is $s\in K_0(B)$ such that $-[1_B]\leq s \leq [1_B]$ and $r-s \in nK_0(B)$. \label{ata.4}
\end{enumerate}
\end{defn}
In the proof of Theorem \ref{thm:MainThm} we only require that $\Q_\omega$ is an admissible target algebra of finite type. To see this, first note that basic properties of $\Q$ show that $\Q$ is an admissible target algebra of finite type.\footnote{Indeed in (\ref{ata.1}) $[p]=[q]$ in $K_0(\Q)$ implies $p$ is Murray--von Neumann equivalent to $q$; in (\ref{ata.2}) $K_1(\Q)=0$, so surjectivity is automatic; in (\ref{ata.3}) $K_0(\Q)$ is unperforated --- if $nr\geq 0$ in $K_0(\Q)$ for some $n\in\mathbb N$, then $r\geq 0$; finally in (\ref{ata.4}), for each $n$, we have $K_0(\Q)=nK_0(\Q)$, so one can always take $s=0$.} Given a sequence $(B_n)$ of admissible target algebras of finite type, the product $\prod_n B_n$ and sequence algebra $\prod_n B_n/\sum_n B_n$ are both admissible target algebras of finite type by \cite[Theorem 4.10\ (iv)]{DE:PLMS}. This is readily adapted to ultraproducts in place of $\prod_n B_n/\sum_n B_n$, so that  $\Q_\omega$ is an admissible target algebra of finite type.

The second ingredient in \cite[Theorem 4.12]{DE:PLMS} is the \emph{total $K$-theory} of a $\mathrm{C}^*$-algebra $A$, $\underline{K}(A)$, defined as the direct sum of the $K$-theory group, $K_*(A)$, and all $K$-theory groups with coefficients, $K_*(A;\mathbb Z/n\mathbb Z)$, $n \in \mathbb N$, from \cite{Sch:PJM} (see Theorem 6.4 there in particular); in the literature, $\underline{K}(A)$ is sometimes just called the $K$-theory of $A$ with coefficients. As noted in \cite[Section 4]{DP:JFA}, there exists a $\mathrm{C}^*$-algebra $C$ such that $\underline{K}(A)\cong K_0(A\otimes C)$ for every $\mathrm{C}^*$-algebra $A$. In this way $\underline{K}(\cdot)$ is a functor from $\mathrm{C}^*$-algebras to abelian groups invariant under homotopy and stable isomorphism such that, for $^*$-homomorphisms $\phi,\psi:A\rightarrow B$, one has (identifying $\underline{K}(M_2(B))$ with $\underline{K}(B)$)
\begin{equation}\label{prop:totalKfacts;extract} (\phi\oplus \psi)_*=\phi_* + \psi_*:\underline{K}(A) \to \underline{K}(B). \end{equation}
When considering maps between total $K$-theory groups, one considers only those morphisms respecting certain additional structure (see \cite[Section 3.1]{DE:PLMS} or \cite{DL:DMJ}); the set of these morphisms is denoted by $\mathrm{Hom}_{\Lambda}(\underline{K}(A),\underline{K}(A))$. 

The next proposition is the only place where the UCT is used in \cite[Theorem 4.12]{DE:PLMS}, and hence the only use of the UCT in our Theorem \ref{thm:MainThm}. Related to the injectivity of (\ref{InjectiveEquation}), its proof relies on injectivity, in this case of the natural map 
\begin{equation}\label{InjectiveEquation2}
KK(A,B_\infty)\rightarrow\mathrm{Hom}_\Lambda(\underline{K}(A),\underline{K}(B_{\infty})).
\end{equation}

\begin{prop}[{Dadarlat--Eilers \cite{DE:PLMS}}]\label{prop:totalKfacts}
Let $A$ be a separable $\mathrm C^*$-algebra that satisfies the UCT and, for each $n\in\mathbb N$, let $B_n$ be an admissible target algebra of finite type.  Set
\begin{equation} B_\infty := \prod_{n=1}^\infty B_n / \{(b_n)_{n=1}^\infty \mid \lim_{n \to \infty} \|b_n\| = 0\}. 
\end{equation}
If $\phi_n,\psi_n:A \to B_n$ are $^*$-homomorphisms satisfying $(\phi_n)_*=(\psi_n)_*:\underline{K}(A) \to \underline{K}(B_n)$ for each $n$, then the $^*$-homomorphisms $A \to B_\infty$ induced by the sequences $(\phi_n)_n$ and $(\psi_n)_n$ have the same class in $KK_{\mathrm{nuc}}(A,B_\infty)$.
\end{prop}

\begin{proof}
Denote the two $^*$-homomorphisms in question by $\Phi$ and $\Psi$ respectively.
By \cite[Theorem 4.10 (ii)]{DE:PLMS}, both $(\phi_n)_n$ and $(\psi_n)_n$ induce the same map $\underline{K}(A) \to \underline{K}(\prod_n B_n)$, whence $\Phi_*=\Psi_*$ in $\mathrm{Hom}_{\Lambda}(\underline{K}(A),\underline{K}(B_\infty))$. As $A$ satisfies the UCT and each $B_n$ is an admissible target algebra of finite type, \cite[Theorem 4.10 (iii)]{DE:PLMS} shows that the natural map (\ref{InjectiveEquation2}) is an isomorphism, so that $[\Phi]=[\Psi]$ in $KK(A, B_\infty)$. Since $A$ is separable and satisfies the UCT, $KK(A, B_\infty)=KK_{\mathrm{nuc}}(A,B_\infty)$ by \cite{S:KThy} (see the discussion at the beginning of the section), so $[\Phi]=[\Psi]$ in $KK_{\mathrm{nuc}}(A, B_\infty)$.\footnote{Note that, when $A$ is nuclear, as it is in our main theorem, $KK(A,B_\infty)$ and $KK_{\mathrm{nuc}}(A,B_\infty)$ automatically agree, and so here the UCT is only used to allow for more general $A$.}
\end{proof}

With the above ingredients in place, the other detail needed to prove \cite[Theorem 4.12]{DE:PLMS} is a device for ensuring that given a sequence of totally full maps $\iota_n:A\rightarrow B_n$ into admissible target algebras of finite type, the induced map $I:A\rightarrow B_\infty$ is totally full.  In general this is false, so in \cite[Theorem 4.12]{DE:PLMS} Dadarlat and Eilers consider simple unital domain algebras $A$, as in this case, every unital $^*$-homomorphism is totally full.
When we apply the stable uniqueness theorem in the proof of Theorem \ref{thm:MainThm}, the domain will be the nonsimple $\mathrm{C}^*$-algebra $C_0((0,1),A)^\sim$ for some separable unital and nuclear $\mathrm C^*$-algebra $A$ satisfying the UCT.  To ensure that the resulting $I$ is totally full we use the following notion of controlled fullness.

\begin{defn}\label{def:Control}
Let $B,C$ be unital $\mathrm{C}^*$-algebras.  A \emph{control function} on $C$ is a function $\Delta:C_{+}^1\setminus\{0\}\rightarrow\mathbb N$.  Say that a unital $^*$-homomorphism $\iota:C\rightarrow B$ is \emph{$\Delta$-full} if for every nonzero $x\in C^1_+$ there exist contractions $b_1,\dots,b_{\Delta(x)}\in B$ such that 
\begin{equation}
1_B=\sum_{i=1}^{\Delta(x)}b_i^*\iota(x)b_i.
\end{equation}
\end{defn}
\noindent Note that a $^*$-homomorphism $\iota:C\rightarrow B$ is totally full if and only if there exists a control function $\Delta:C_+^1\setminus\{0\}\rightarrow\mathbb N$ such that $\iota$ is $\Delta$-full. 

Using these control functions we obtain the small generalisation\footnote{If $C$ is simple and unital, then there exists a control function $\Delta_C$ such that every unital $^*$-homomorphism $C\to B$ is $\Delta_C$-full; in this way Theorem \ref{thm:StableUniqueness} does generalise \cite[Theorem 4.12]{DE:PLMS}.} of \cite[Theorem 4.12]{DE:PLMS}, with constants independent of the exact form of the maps $\iota,\phi$ and $\psi$, and depending only on how full the map $\iota$ is in terms of these control functions (and on the finite set and tolerance).  The proof is essentially the same as that in \cite{DE:PLMS}.
A similar adaptation of the stable uniqueness argument to nonsimple (but nuclear) domains appears in \cite[Lemma 5.9]{Lin:apprUCT}.

\begin{thm}
\label{thm:StableUniqueness}
Let $C$ be a separable, unital, exact $\mathrm{C}^{*}$-algebra that satisfies the UCT.
Let $\Delta:C^1_+ \setminus \{0\} \to \mathbb N$ be a control function, let $\mathcal G \subset C$ be a finite subset, and let $\delta > 0$.
Then there exists $n \in \mathbb N$ such that, for any admissible target algebra $D$ of finite type, any unital $\Delta$-full $^{*}$-homomorphism $\iota:C \to D$, and any nuclear $^{*}$-homomorphisms $\phi,\psi:C \to D$, if
\begin{enumerate}[(i)]
\item $\phi_*=\psi_*:\underline{K}(C) \to \underline{K}(D)$, and
\item $\phi(1_C)$ is unitarily equivalent to $\psi(1_C)$,
\end{enumerate}
then there exists a unitary $u \in M_{n+1}(D)$ such that
\begin{equation}
\|u(\phi(c) \oplus \iota^{\oplus n}(c))u^* -(\psi(c) \oplus \iota^{\oplus n}(c))\|<\delta, \quad c \in \mathcal G. \end{equation}
\end{thm}

\begin{proof}
Suppose that the statement is false, so it fails for some control function $\Delta$, finite set $\mathcal G$ and $\delta>0$.
Then for every $n \in \mathbb N$, there exists an admissible target algebra of finite type $D_n$ and $^*$-homomorphisms $\iota_n,\phi_n,\psi_n$ such that $\iota_n$ is unital and $\Delta$-full, $\phi_n,\psi_n$ are nuclear, $\phi_n(1_C)$ is unitarily equivalent to $\psi_n(1_C)$, $(\phi_n)_* = (\psi_n)_*$, yet
\begin{equation}
\label{eq:StableUniquenesAssume}
\max_{c \in \mathcal G} \|v(\phi_n(c) \oplus \iota_n^{\oplus n}(c))v^* - (\psi_n(c) \oplus \iota_n^{\oplus n}(c))\| \geq \delta \end{equation}
for every unitary $v \in M_{n+1}(D_n)$.

Define
\begin{equation} D_\infty := \prod_{n=1}^\infty D_n/\{(d_n)_{n=1}^\infty \mid \lim_{n\to\infty} \|d_n\| = 0\}, \end{equation}
and define $\Phi,\Psi,I:C\rightarrow D_\infty$ to be the $^*$-homomorphisms induced by the sequences $(\phi_n),(\psi_n),(\iota_n)$ respectively. Since each $\phi_n$ and $\psi_n$ is nuclear, and $C$ is exact, $\Phi,\Psi$ are nuclear by \cite[Proposition 3.3]{D:JFA97}. As $C$ satisfies the UCT, $[\Phi]=[\Psi]$ in $KK(C,D_\infty)$ by Proposition \ref{prop:totalKfacts}. Each $\phi_n(1_C)$ is unitarily equivalent to $\psi_n(1_C)$, so that $\Phi(1_C)$ is unitarily equivalent to $\Psi(1_C)$.

Since each $\iota_n$ is $\Delta$-full, it follows that $I$ is also $\Delta$-full.  Indeed, given a nonzero positive contraction $c\in C$, there are contractions $d_{1,n},\dots,d_{\Delta(c),n}\in D_n$ such that $1_{D_n}=\sum_{i=1}^{\Delta(c)} d_{i,n}^*\iota_n(c)d_{i,n}$.  Thus letting $d_i$ be the contraction in $D_\infty$ represented by $(d_{i,n})_n$, we have $1_{D_\infty}=\sum_{i=1}^{\Delta(c)}d_i^* I(c)d_i$.  Consequently $I(c)$ is full in $D_\infty$ for every nonzero positive contraction $c\in C$. The same holds for all nonzero $c\in C$, as $I(|c|)$ is in the ideal generated by $I(c)$. That is, $I$ is totally full.

Hence by Theorem \ref{DE:SU}, there exist $m \in \mathbb N$ and a unitary $u \in M_{m+1}(D_\infty)$ such that
\begin{equation} \|u(\Phi(c) \oplus I^{\oplus m}(c))u^* -( \Psi(c) \oplus I^{\oplus m}(c))\| < \delta, \quad c \in \mathcal G. \end{equation}
The unitary $u$ lifts to a sequence of unitaries $u_n \in M_{m+1}(D_n)$, and for $n$ sufficiently large, it follows that
\begin{equation}\label{e:3.10} \|u_n(\phi_n(c) \oplus \iota_n^{\oplus m}(c))u_n^* - (\psi_n(c) \oplus \iota_n^{\oplus m}(c))\| < \delta, \quad c \in \mathcal G. \end{equation}
Taking $n\geq m$ large enough so that \eqref{e:3.10} holds, the unitary $v:=u_n \oplus 1_{D_n}^{\oplus (n-m)}$ in $M_{n+1}(D_n)$ gives a contradiction to \eqref{eq:StableUniquenesAssume}, proving the theorem.
\end{proof}

\section{A patching lemma}\label{Sect.Patch}

\allowdisplaybreaks[2]

\noindent
In this section we provide the patching lemma, which connects two $^*$-homo\-mor\-phisms $\nu_0$ and $\nu_1$ defined on the suspension $C_0((0,1),A)$ of $A$ into a single map $\rho$, in the sense that $\rho$ recovers the two original maps on the left and right thirds of the interval respectively. The map $\rho$ will be c.p.c.\ and approximately multiplicative, with the estimates depending on how close the original maps are to being approximately unitarily equivalent over the middle third of the interval.

\begin{lemma}
\label{lem:Patching}
There exists a partition of unity $f_0,f_{\frac12},f_1$ for $C([0,1])$ with $f_{\frac12} \in C_0((\tfrac13,\tfrac23))$, such that the following holds.

Let $A$ and $E$ be unital $\mathrm{C}^*$-algebras and $\theta:C([0,1])\rightarrow E$ a unital $^*$-homomorphism. 
Given $^*$-homomorphisms $\nu_0,\nu_{1}:C_0((0,1),A)\rightarrow E$ compatible with $\theta$ and a unitary $v\in E$, there exists a completely positive map $\rho:C_0((0,1),A)\rightarrow M_2(E)=E\otimes M_2$ such that:
\begin{enumerate}[(i)]
\item $\rho|_{C_0((0,\frac13),A)} = 
\left(\begin{array}{cc} \nu_0|_{C_0((0,\frac13),A)} & 0 \\ 0& 0 \end{array}\right)$.\label{Patch.C1}
\item $\rho|_{C_0((\frac23,1),A)} = 
\left(\begin{array}{cc} \nu_1|_{C_0((\frac23,1),A)} & 0 \\ 0& 0 \end{array}\right)$.\label{Patch.C2}
\item If $\tau \in T(E)$ satisfies $\tau\circ\nu_0=\tau\circ\nu_1$, then $(\tau \otimes \mathrm{Tr}_2) \circ \rho = \tau \circ \nu_0$, where we recall that $\mathrm{Tr}_2$ is the canonical non-normalised tracial functional on $M_2$.\label{Patch.C3}
\item For $x\in C_0((0,1),A)$ and $h\in C([0,1])$, \label{Patch.C5}
\begin{equation}
\|\rho(hx)-\theta^{\oplus 2}(h)\rho(x)\| \leq \Big(\|[\theta(f_{\frac12}),v]\| \|h\|+\|[\theta(hf_{\frac12}),v]\|\Big)\|x\|.
\end{equation}
\item\label{Patch.C4} For $x,y\in C_0((0,1),A)$, 
\begin{align}
\notag
\lefteqn{\|\rho(xy)-\rho(x)\rho(y)\| }  \\
& \leq\Big(7\|[v,\theta(f_{\frac12})]\|+2\|[v,\theta(f_0f_{\frac12})]\|+5\|[v,\theta(f_{\frac12}f_1)]\|\Big)\|x\|\|y\| \nonumber\\
\notag
& \phantom{\leq} \quad +\Big(\|\nu_1(f_{\frac12}f_1x)v-v\nu_0(f_{\frac12}f_1x)\|\nonumber\\
&  \phantom{\leq} \qquad +\|\nu_1(f_{\frac12}x)v-v\nu_0(f_{\frac12}x)\|\Big)\|y\|.
\end{align}
\end{enumerate}

\end{lemma}

\begin{remark}
\label{rmk:Patching}
Note that in (\ref{Patch.C5}) and (\ref{Patch.C4}), the estimates on the right-hand sides are bounded in terms of how well $\mathrm{Ad}(v) \circ \nu_0|_{C_0((\frac13,\frac23),A)}$ approximates $\nu_1|_{C_0((\frac13,\frac23),A)}$.
Thus, for example, if
\begin{equation}
\mathrm{Ad}(v) \circ \nu_0|_{C_0((\frac13,\frac23),A)} = \nu_1|_{C_0((\frac13,\frac23),A)}
\end{equation}
then (\ref{Patch.C5}) and (\ref{Patch.C4}) tell us that $\rho$ is a $^*$-homomorphism that is compatible with $\theta$ giving an exact patching result.

Likewise, given finite sets $\mathcal F \subset C_0((0,1),A)$, $\mathcal F' \subset C([0,1])$, and a tolerance $\eta>0$, there are a finite set $\mathcal G\subset C_0((\tfrac13, \tfrac23),A)$ and $\delta>0$ such that, if 
\begin{equation}
\mathrm{Ad}(v) \circ \nu_0(x) \approx_\delta \nu_1(x), \quad x \in \mathcal G,
\end{equation}
then
\begin{align}
\notag
\rho(xy) &\approx_\eta \rho(x)\rho(y), \quad x,y\in \mathcal F, \text{ and} \\
\rho(hx) &\approx_\eta \theta^{\oplus 2}(h)\rho(x) \quad x\in\mathcal F,\ h\in\mathcal F'.
\end{align}
Namely, set
\begin{align}
\notag
\mathcal G&:= \{f_{\frac12},f_0f_{\frac12},f_1f_{\frac12}\} \cup \{hf_{\frac12} \mid h \in \mathcal F'\} \\
& \phantom{:=}  \quad \cup \{f_{\frac12}x \mid x \in \mathcal F\} \cup \{f_{\frac12}f_1x\mid x \in \mathcal F\}
\end{align}
and take $\delta$ such that
\begin{equation}
(M'+1)M\delta\leq \eta\text{ and }14M^2\delta+2M\delta \leq\eta,
\end{equation}
where $M:=\max\{\|x\| \mid x \in \mathcal F\}$ and $M':=\max\{\|h\| \mid h \in \mathcal F'\}$.
\end{remark}

\begin{proof}[Proof of Lemma \ref{lem:Patching}]
Define functions $f_0,f_1\in C([0,1])$ so that
\begin{align}
f_0|_{[0,\frac13]} &\equiv 1, & f_0|_{[\frac25,1]} \equiv 0, \label{Patch.E1}\\
f_1|_{[0,\frac35]} &\equiv 0, & f_1|_{[\frac23,1]} \equiv 1,\label{Patch.E2}
\end{align}
and $f_0,f_1$ are linear on the complements of the intervals in (\ref{Patch.E1}) and (\ref{Patch.E2}) respectively. Write 
\begin{equation}\label{Patch.E3}
f_{\frac12}:= 1_{C([0,1])}-f_0-f_1
\end{equation}so that $f_0,f_{\frac12},f_1$ comprise a partition of unity of $C([0,1])$ and $f_{\frac12} \in C_0((\frac13,\frac23))$. Let $U \in C([0,1],M_2) \cong C([0,1]) \otimes M_{2}$ be a unitary which satisfies
\begin{equation}\label{Patch.E5} U|_{[0,\frac25]} \equiv \left(\begin{array}{cc} 1&0\\0&1 \end{array}\right), \quad
U|_{[\frac35,1]} \equiv \left(\begin{array}{cc} 0&1\\1&0 \end{array}\right). \end{equation}
 These objects are summarised in Figure \ref{Fig1}.

\begin{figure}[h!]
\begin{picture}(280,130)
\setlength{\unitlength}{0.15mm}
\linethickness{0.15mm}
\put(0,255){\tiny{$1$}}
\put(10,110){\tiny{$0$}}
\put(0,130){\tiny{$0$}}
\put(590,110){\tiny{$1$}}
\put(20,140){\line(1,0){580}}
\put(20,140){\line(0,-1){5}}
\put(600,140){\line(0,-1){5}}
\put(213,140){\line(0,-1){5}}
\put(253,140){\line(0,-1){5}}
\put(408,140){\line(0,-1){5}}
\put(368,140){\line(0,-1){5}}
\put(205,105){\tiny{$\tfrac13$}}
\put(249,105){\tiny{$\tfrac25$}}
\put(360,105){\tiny{$\tfrac35$}}
\put(399,105){\tiny{$\tfrac23$}}
\put(20,260){\line(1,0){193}}
\put(213,260){\line(1,-3){40}}
\put(100,290){\tiny{$f_0$}}
\put(300,290){\tiny{$f_{\frac12}$}}
\put(213,140){\line(1,3){40}}
\put(368,140){\line(1,3){40}}
\put(253,260){\line(1,0){115}}
\put(368,260){\line(1,-3){40}}
\put(408,260){\line(1,0){192}}
\put(500,290){\tiny{$f_1$}}
\put(20,60){\vector(1,0){233}}
\put(80,25){\tiny{$U\equiv\begin{pmatrix}1&0\\0&1\end{pmatrix}$}}
\put(430,25){\tiny{$U\equiv\begin{pmatrix}0&1\\1&0\end{pmatrix}$}}
\put(253,60){\vector(-1,0){233}}
\put(368,60){\vector(1,0){233}}
\put(600,60){\vector(-1,0){233}}
\end{picture}
\caption{$f_0$, $f_{\frac12}$, $f_1$ and the unitary $U$}\label{Fig1}
\end{figure}
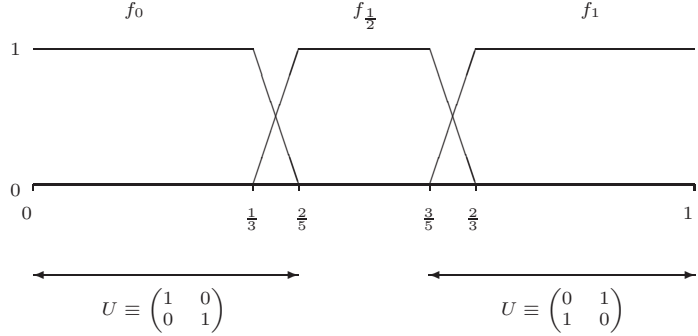

Given $E$, a unital $^*$-homomorphism $\theta:C([0,1])\rightarrow E$, compatible $^*$-homomorphisms $\nu_i:C_0((0,1),A)\rightarrow E$ and a unitary $v\in E$, define c.p.c.\ maps $\rho_0,\rho_{\frac12},\rho_1:C_0((0,1),A) \to M_2(E)$ by
\begin{align}
\rho_i(x) &:= \begin{pmatrix} \nu_i(f_ix) & 0 \\ 0&0 \end{pmatrix},\quad x \in C_0((0,1),A),\ i=0,1 \text{ and} \label{Patch.DefRho1}\\
\rho_{\frac12}(x) &:= V\begin{pmatrix} \nu_0(f_{\frac12}x) & 0 \\ 0&0 \end{pmatrix}V^*,\quad  x\in C_0((0,1),A),\label{Patch.DefRho2}
\end{align}
where 
\begin{equation}\label{Patch.DefV}
V:=\Big((\theta\otimes \id_{M_2})(U^*)\begin{pmatrix} 1_E & 0 \\ 0 & v \end{pmatrix}(\theta\otimes \id_{M_2})(U)\Big),
\end{equation}
a unitary in $M_2(E)$.  Then define 
\begin{equation}\label{Patch.E20}
\rho:=\rho_0+\rho_{\frac12}+\rho_1.
\end{equation} This is certainly a completely positive map, and the location of the supports of $f_0,f_{\frac12},f_1$ ensure that (\ref{Patch.C1}) and (\ref{Patch.C2}) hold. For (\ref{Patch.C3}), suppose that $\tau\in T(E)$ has $\tau\circ\nu_0=\tau\circ\nu_1$.  Then
\begin{align}
(\tau \otimes \mathrm{Tr}_{2})(\rho(x))&\stackrel{\hphantom{(\ref{Patch.DefRho1}),(\ref{Patch.DefRho2})}}{=}(\tau \otimes \mathrm{Tr}_{2})(\rho_0(x)+\rho_{\frac12}(x)+\rho_1(x))\nonumber
\\
&\stackrel{(\ref{Patch.DefRho1}),(\ref{Patch.DefRho2})}{=}\tau(\nu_0(f_0x)+\nu_0(f_{\frac12}x)+\nu_1(f_1x))\nonumber\\
&\stackrel{\hphantom{(\ref{Patch.DefRho1}),(\ref{Patch.DefRho2})}}{\stackrel{\tau\circ\nu_0=\tau\circ\nu_1}{=}}\tau\circ\nu_0(f_0x+f_{\frac12}x+f_1x)\nonumber\\
&\stackrel{\hphantom{(\ref{Patch.DefRho1}),(\ref{Patch.DefRho2})}}{\stackrel{(\ref{Patch.E3})}{=}}\tau\circ\nu_0(x),\quad x\in C_0((0,1),A),
\end{align}
establishing (\ref{Patch.C3}).  

For (\ref{Patch.C5}), fix $h\in C([0,1])$ and $x\in C_0((0,1),A)$; by rescaling, we may assume these to be contractions. Then compatibility gives
\begin{align}
\rho_{\frac12}(hx)&=V\begin{pmatrix}\theta(f_{\frac12}h)&0\\0&\theta(f_{\frac12}h)\end{pmatrix}\begin{pmatrix}\nu_0(x)&0\\0&0\end{pmatrix}V^*,
\end{align}
whence
\begin{align}
\notag
\lefteqn{\|\rho_{\frac12}(hx)-\theta^{\oplus 2}(h)\rho_{\frac12}(x)\|} \\
\notag &\leq  \|[V,\theta^{\oplus 2}(f_{\frac12}h)]\|+\|[V,\theta^{\oplus 2}(f_{\frac12})]\|\\
\label{Patch.E21}
&= \|[v,\theta(hf_{\frac12})]\|+\|[v,\theta(f_{\frac12})]\|,
\end{align}
as 
\begin{equation}\label{Patch.Commute}
[\theta^{\oplus 2}(C([0,1])),(\theta\otimes\id_2)(U)]=0.
\end{equation}
For $i=0,1$, the compatibility of $\nu_i$ with $\theta$ gives 
\begin{equation}
\rho_i(hx)=\theta^{\oplus 2}(h)\rho_i(x).
\end{equation} 
Combining this with \eqref{Patch.E21} gives
\begin{align}
\lefteqn{\|\rho(hx)-\theta^{\oplus 2}(h)\rho(x)\|}\nonumber\\
&\leq \sum_{i=0,\frac12,1}\|\rho_i(hx)-\theta^{\oplus 2}(h)\rho_i(x)\|\nonumber\\
&\leq \|[v,\theta(hf_{\frac12})]\|+\|[v,\theta(f_{\frac12})]\|,
\end{align}
establishing (\ref{Patch.C5}).

Finally we show (\ref{Patch.C4}); by rescaling, it again suffices to do this when $x,y \in C_{0}((0,1),A)$ are contractions. The proof amounts to a long, but routine, calculation. We shall estimate
\begin{equation}\label{Patch.C8}
\|\rho_i(x)\rho_j(y)-\theta^{\oplus 2}(f_i)\rho_j(xy)\|,\quad i,j=0,\tfrac12,1.
\end{equation}
There are nine cases to consider.  For $i=j=0$ or $i=j=1$, compatibility of $\nu_i$ with $\theta$ gives
\begin{equation}\label{Patch.E10}
\theta^{\oplus 2}(f_i)\rho_j(xy)=\rho_i(x)\rho_j(y),
\end{equation}
while, for $i,j\in\{0,1\}$ with $i\neq j$, compatibility combines with $f_0f_1=0$, to yield
\begin{equation}\label{Patch.E11}
\theta^{\oplus 2}(f_i)\rho_j(xy)=0=\rho_i(x)\rho_j(y).
\end{equation}
Thus (\ref{Patch.C8}) evaluates to zero when $i,j\in\{0,1\}$.

When $i=j=\tfrac12$, compatibility gives
\begin{equation}
\rho_{\frac12}(x)\rho_{\frac12}(y)=V\begin{pmatrix}\theta(f_{\frac12})&0\\0&\theta(f_{\frac12})\end{pmatrix}\begin{pmatrix}\nu_0(f_{\frac12}xy)&0\\0&0\end{pmatrix}V^*,
\end{equation}
while
\begin{equation}
\theta^{\oplus 2}(f_{\frac12})\rho_{\frac12}(xy)=\begin{pmatrix}\theta(f_{\frac12})&0\\0&\theta(f_{\frac12})\end{pmatrix}V\begin{pmatrix}\nu_0(f_{\frac12}xy)&0\\0&0\end{pmatrix}V^*.
\end{equation}
Using these and (\ref{Patch.Commute}), we have
\begin{align}
\|\rho_{\frac12}(x)\rho_{\frac12}(y)-\theta^{\oplus 2}(f_{\frac12})\rho_{\frac12}(xy)\|
&\leq \|[\theta^{\oplus 2}(f_{\frac12}),V]\| \nonumber \\
& = \|[\theta(f_{\frac12}),v]\|.\label{Patch.E12}
\end{align}

For the cases in which $\{i,j\}=\{0,\frac12\}$, first note that \eqref{Patch.E1} and \eqref{Patch.E5} give
\begin{align}
\begin{pmatrix} \theta(f_0)&0 \\ 0&0 \end{pmatrix}&=
\begin{pmatrix} \theta(f_0)&0 \\ 0&0 \end{pmatrix}(\theta\otimes \id_{M_2} )(U)\nonumber\\
&=(\theta\otimes \id_{M_2})(U)\begin{pmatrix} \theta(f_0)&0 \\ 0&0 \end{pmatrix},
\end{align}
and therefore, using \eqref{Patch.DefV},
\begin{equation}
\begin{pmatrix} \theta(f_0)&0 \\ 0&0 \end{pmatrix} = 
\begin{pmatrix} \theta(f_0)&0 \\ 0&0 \end{pmatrix}V = 
V\begin{pmatrix} \theta(f_0)&0 \\ 0&0 \end{pmatrix}.
\label{Patch.E22}
\end{equation}
Then
\begin{align}
\rho_{\frac12}(x)\rho_0(y)&\stackrel{\hphantom{(\ref{Patch.E22})}}{=} V\begin{pmatrix}\theta(f_{\frac12})\nu_0(x)&0\\0&0\end{pmatrix}V^*\begin{pmatrix}\theta(f_0)\nu_0(y)&0\\0&0\end{pmatrix}\nonumber\\
&\stackrel{(\ref{Patch.E22})}{=}V\begin{pmatrix}\theta(f_0f_{\frac12})\nu_0(xy)&0\\0&0\end{pmatrix}\nonumber\\
&\stackrel{(\ref{Patch.E22})}{=}\begin{pmatrix}\theta(f_0f_{\frac12})\nu_0(xy)&0\\0&0\end{pmatrix}\nonumber\\
&\stackrel{\hphantom{(\ref{Patch.E22})}}{=}\theta^{\oplus 2}(f_{\frac12})\rho_0(xy),\label{Patch.E14}
\end{align}
and this shows that (\ref{Patch.C8}) vanishes for $(i,j)=(\frac12,0)$. Next compute
\newlength{\leftlengthA} \newlength{\leftlengthB} \newlength{\leftlengthC} \newlength{\leftlengthD} \newlength{\leftlengthE} \newlength{\leftlengthF} \newlength{\leftlengthG}
\settowidth{\leftlengthA}{$\displaystyle{\stackrel{\eqref{Patch.E1},\eqref{Patch.E5},\eqref{Patch.DefV}}=}$}
\settowidth{\leftlengthB}{$\displaystyle{\stackrel{\eqref{Patch.Commute}}=}$}
\settowidth{\leftlengthC}{$\displaystyle{_{\hphantom{\|[v,\theta(f_{\frac12})]\|}}\approx_{\|[v,\theta(f_{\frac12})]\|}}$}
\settowidth{\leftlengthD}{$\displaystyle{_{\hphantom{\|[v,\theta(f_0f_{\frac12})]\|}}\approx_{\|[v,\theta(f_0f_{\frac12})]\|}}$}
\settowidth{\leftlengthE}{$\displaystyle{\stackrel{\eqref{Patch.E1},\eqref{Patch.E5}}=}$}
\settowidth{\leftlengthF}{$\displaystyle{_{\hphantom{\|[v,\theta(f_0f_{\frac12})]\|}}\approx_{\|[v,\theta(f_0f_{\frac12})]\|}}$}
\settowidth{\leftlengthG}{$\displaystyle{_{\hphantom{\|[v,\theta(f_{\frac12})]\|}}\approx_{\|[v,\theta(f_{\frac12})]\|}}$}
\newlength{\rightlengthA} \newlength{\rightlengthB} \newlength{\rightlengthC} \newlength{\rightlengthD} \newlength{\rightlengthE} \newlength{\rightlengthF} \newlength{\rightlengthG}
\newlength{\maxrightlength}
\setlength{\leftlengthA}{-0.5\leftlengthA} \setlength{\leftlengthB}{-0.5\leftlengthB} \setlength{\leftlengthC}{-0.5\leftlengthC} \setlength{\leftlengthD}{-0.5\leftlengthD} \setlength{\leftlengthE}{-0.5\leftlengthE} \setlength{\leftlengthF}{-0.5\leftlengthF} \setlength{\leftlengthG}{-0.5\leftlengthG}
\setlength{\maxrightlength}{-\leftlengthD}
\setlength{\rightlengthA}{\maxrightlength} \addtolength{\rightlengthA}{\leftlengthA} \setlength{\rightlengthB}{\maxrightlength} \addtolength{\rightlengthB}{\leftlengthB} \setlength{\rightlengthC}{\maxrightlength} \addtolength{\rightlengthC}{\leftlengthC} \setlength{\rightlengthD}{\maxrightlength} \addtolength{\rightlengthD}{\leftlengthD} \setlength{\rightlengthE}{\maxrightlength} \addtolength{\rightlengthE}{\leftlengthE} \setlength{\rightlengthF}{\maxrightlength} \addtolength{\rightlengthF}{\leftlengthF} \setlength{\rightlengthG}{\maxrightlength} \addtolength{\rightlengthG}{\leftlengthG}
\begin{align}
\notag
&\hspace*{-2em} \begin{pmatrix}\theta(f_0)&0\\0&\theta(f_0)\end{pmatrix}V\begin{pmatrix}\theta(f_{\frac12})&0 \\ 0&\theta(f_{\frac12}) \end{pmatrix} \\
\notag
&\hspace*{\leftlengthA} 
\stackrel{\eqref{Patch.E1},\eqref{Patch.E5},\eqref{Patch.DefV}}=
\hspace*{\rightlengthA} 
\begin{pmatrix}\theta(f_0)&0\\0&\theta(f_0)v\end{pmatrix}( \theta\otimes\id_{M_2})(U)\begin{pmatrix}\theta(f_{\frac12})&0 \\ 0&\theta(f_{\frac12})\end{pmatrix} \\
\notag
&\hspace*{\leftlengthB}
\stackrel{\eqref{Patch.Commute}}=
\hspace*{\rightlengthB} 
\begin{pmatrix}\theta(f_0f_{\frac12})&0\\0&\theta(f_0)v\theta(f_{\frac12})\end{pmatrix}( \theta\otimes\id_{M_2})(U)\\
\notag
&\hspace*{\leftlengthC}
{}_{\hphantom{\|[v,\theta(f_{\frac12})]\|}}\approx_{\|[v,\theta(f_{\frac12})]\|}
\hspace*{\rightlengthC} 
\begin{pmatrix}\theta(f_0f_{\frac12})&0\\0&\theta(f_0f_{\frac12})v\end{pmatrix}( \theta\otimes\id_{M_2})(U)\\
\notag
&\hspace*{\leftlengthD}
{}_{\hphantom{\|[v,\theta(f_0f_{\frac12})]\|}}\approx_{\|[v,\theta(f_0f_{\frac12})]\|}
\hspace*{\rightlengthD} 
\begin{pmatrix}\theta(f_0f_{\frac12})&0\\0&v\theta(f_0f_{\frac12})\end{pmatrix}( \theta\otimes\id_{M_2})(U)\\
\notag
&\hspace*{\leftlengthE}
\stackrel{\eqref{Patch.E1},\eqref{Patch.E5}}=
\hspace*{\rightlengthE} 
\begin{pmatrix}\theta(f_0f_{\frac12})&0\\0&v\theta(f_0f_{\frac12})\end{pmatrix} \\
\notag
&\hspace*{\leftlengthF}
{}_{\hphantom{\|[v,\theta(f_0f_{\frac12})]\|}}\approx_{\|[v,\theta(f_0f_{\frac12})]\|}
\hspace*{\rightlengthF} 
\begin{pmatrix}\theta(f_0f_{\frac12})&0\\0&\theta(f_0f_{\frac12})v\end{pmatrix} \\
&\hspace*{\leftlengthG}
{}_{\hphantom{\|[v,\theta(f_{\frac12})]\|}}\approx_{\|[v,\theta(f_{\frac12})]\|}
\hspace*{\rightlengthG} 
\begin{pmatrix}\theta(f_0f_{\frac12})&0\\0&\theta(f_0)v\theta(f_{\frac12})\end{pmatrix}.
\label{Patch.E6}
\end{align}
Then with $\varepsilon_0:=2(\|[v,\theta(f_{\frac12})]\|+\|[v,\theta(f_0f_{\frac12})]\|)$, we have
\begin{align}
\rho_0(x)\rho_{\frac12}(y)&\stackrel{\hphantom{\eqref{Patch.E22}{}_{\varepsilon_0}}}{=_{\hphantom{\varepsilon_{0}}}} \begin{pmatrix}\nu_0(x)\theta(f_0)&0\\0&0\end{pmatrix}V\begin{pmatrix}\theta(f_{\frac12})\nu_{0}(y)&0\\0&0\end{pmatrix}V^*\nonumber\\
&\stackrel{\eqref{Patch.E22}\hphantom{{}_{\varepsilon_0}}}{=_{\hphantom{\varepsilon_{0}}}}\begin{pmatrix}\theta(f_0f_{\frac12})\nu_0(xy)&0\\0&0\end{pmatrix}V^*\nonumber\\
\notag
&\stackrel{\hphantom{\eqref{Patch.E22}{}_{\varepsilon_0}}}{=_{\hphantom{\varepsilon_{0}}}} \begin{pmatrix}\theta(f_0f_{\frac12})&0 \\ 0&\theta(f_0)v\theta(f_{\frac12})\end{pmatrix}\begin{pmatrix} \nu_0(xy)&0\\0&0 \end{pmatrix}V^* \\
&\stackrel{(\ref{Patch.E6})\hphantom{_{\varepsilon_0}}}{\approx_{\varepsilon_0}}\begin{pmatrix}\theta(f_0)&0\\0&\theta(f_0)\end{pmatrix}V\begin{pmatrix}\nu_0(f_{\frac12}xy)&0\\0&0\end{pmatrix}V^*\nonumber\\
&\stackrel{\hphantom{\eqref{Patch.E22}{}_{\varepsilon_0}}}{=_{\hphantom{\varepsilon_{0}}}}\theta^{\oplus 2}(f_0)\rho_{\frac12}(xy).\label{Patch.E13}
\end{align}

For the cases when $\{i,j\}=\{\frac12,1\}$, similarly to \eqref{Patch.E6}, we compute 
\settowidth{\leftlengthA}{$\displaystyle{\stackrel{\eqref{Patch.E2},\eqref{Patch.E5},\eqref{Patch.DefV}}=}$}
\settowidth{\leftlengthB}{$\displaystyle{\stackrel{\eqref{Patch.Commute}}=}$}
\settowidth{\leftlengthC}{$\displaystyle{_{\|[v,\theta(f_{\frac12})]\|}\approx_{\|[v,\theta(f_{\frac12})]\|}}$}
\settowidth{\leftlengthD}{$\displaystyle{_{\|[v,\theta(f_1f_{\frac12})]\|}\approx_{\|[v,\theta(f_1f_{\frac12})]\|}}$}
\settowidth{\leftlengthE}{$\displaystyle{\stackrel{\eqref{Patch.E2},\eqref{Patch.E5}}=}$}
\settowidth{\leftlengthF}{$\displaystyle{_{\|[v,\theta(f_1f_{\frac12})]\|}\approx_{\|[v,\theta(f_1f_{\frac12})]\|}}$}
\setlength{\leftlengthA}{-0.5\leftlengthA} \setlength{\leftlengthB}{-0.5\leftlengthB} \setlength{\leftlengthC}{-0.5\leftlengthC} \setlength{\leftlengthD}{-0.5\leftlengthD} \setlength{\leftlengthE}{-0.5\leftlengthE} \setlength{\leftlengthF}{-0.5\leftlengthF}
\setlength{\maxrightlength}{-\leftlengthD}
\setlength{\rightlengthA}{\maxrightlength} \addtolength{\rightlengthA}{\leftlengthA} \setlength{\rightlengthB}{\maxrightlength} \addtolength{\rightlengthB}{\leftlengthB} \setlength{\rightlengthC}{\maxrightlength} \addtolength{\rightlengthC}{\leftlengthC} \setlength{\rightlengthD}{\maxrightlength} \addtolength{\rightlengthD}{\leftlengthD} \setlength{\rightlengthE}{\maxrightlength} \addtolength{\rightlengthE}{\leftlengthE} \setlength{\rightlengthF}{\maxrightlength} \addtolength{\rightlengthF}{\leftlengthF}
\begin{align}
\notag
&\hspace*{-2em}\begin{pmatrix}\theta(f_1)&0\\0&\theta(f_1)\end{pmatrix}V\begin{pmatrix}\theta(f_{\frac12})&0 \\ 0&\theta(f_{\frac12}) \end{pmatrix} \\
\notag
&\hspace*{\leftlengthA}
\stackrel{\eqref{Patch.E2},\eqref{Patch.E5},\eqref{Patch.DefV}}=
\hspace*{\rightlengthA}
\begin{pmatrix}0&\theta(f_1)v\\ \theta(f_1)&0\end{pmatrix}\hspace*{-0.3em}(\theta\otimes\id_{M_2})(U)\hspace*{-0.3em}\begin{pmatrix}\theta(f_{\frac12})&0 \\ 0&\theta(f_{\frac12})\end{pmatrix} \\
\notag
&\hspace*{\leftlengthB}
\stackrel{\eqref{Patch.Commute}}=
\hspace*{\rightlengthB}
\begin{pmatrix}0&\theta(f_1)v\theta(f_{\frac12})\\ \theta(f_1f_{\frac12})&0\end{pmatrix}(\theta\otimes\id_{M_2})(U) \\
\notag
&\hspace*{\leftlengthC}
\phantom{_{\|[v,\theta(f_{\frac12})]\|}}\approx_{\|[v,\theta(f_{\frac12})]\|}
\hspace*{\rightlengthC}
\begin{pmatrix}0&\theta(f_1f_{\frac12})v\\ \theta(f_1f_{\frac12})&0\end{pmatrix}(\theta\otimes\id_{M_2})(U) \\
\notag
&\hspace*{\leftlengthD}
\phantom{_{\|[v,\theta(f_1f_{\frac12})]\|}}\approx_{\|[v,\theta(f_1f_{\frac12})]\|}
\hspace*{\rightlengthD}
\begin{pmatrix}0&v\theta(f_1f_{\frac12})\\ \theta(f_1f_{\frac12})&0\end{pmatrix}(\theta\otimes\id_{M_2})(U) \\
\label{Patch.E23a}
&\hspace*{\leftlengthE}
\stackrel{\eqref{Patch.E2},\eqref{Patch.E5}}=
\hspace*{\rightlengthE}
\begin{pmatrix}v\theta(f_1f_{\frac12})&0\\ 0&\theta(f_1f_{\frac12})\end{pmatrix} \\
\label{Patch.E23}
&\hspace*{\leftlengthF}
\phantom{_{\|[v,\theta(f_1f_{\frac12})]\|}}\approx_{\|[v,\theta(f_1f_{\frac12})]\|}
\hspace*{\rightlengthF}
\begin{pmatrix}\theta(f_1f_{\frac12})v&0\\ 0&\theta(f_1f_{\frac12})\end{pmatrix}.
\end{align}
With $\varepsilon_1:=\|[v,\theta(f_{\frac12})]\|+2\|[v,\theta(f_1f_{\frac12})]\|$, this gives
\begin{align}
\rho_1(x)\rho_{\frac12}(y) \hspace*{-0.5em}
&\stackrel{\hphantom{\eqref{Patch.E2},\eqref{Patch.E5},\eqref{Patch.DefV}}}{=_{\hphantom{\varepsilon_{1}}}}\begin{pmatrix}\nu_1(f_1x)&0\\0&0\end{pmatrix}V\begin{pmatrix}\nu_{0}(f_{\frac12}y)&0\\0&0\end{pmatrix}V^*\nonumber\\
\notag
&\stackrel{\hphantom{\eqref{Patch.E2},\eqref{Patch.E5},\eqref{Patch.DefV}}}{\stackrel{\eqref{Patch.E23}\hphantom{_{\varepsilon_1}}}{\approx_{\varepsilon_1}}} \begin{pmatrix}\theta(f_1)\nu_1(f_{\frac12}x)v\nu_0(y)&0\\0&0\end{pmatrix}V^* \\
&\stackrel{\eqref{Patch.E2},\eqref{Patch.E5},\eqref{Patch.DefV}}{=_{\hphantom{\varepsilon_{1}}}} (\theta\otimes\id_{M_2})(U^{*})\begin{pmatrix}0&0\\ \nu_1(f_1f_{\frac12}x)v\nu_0(y)&0\end{pmatrix}V^*,
\end{align}
whereas 
\begin{align}
\rho_{\frac12}(f_1xy) \hspace*{-0.5em}
 &\stackrel{\hphantom{\eqref{Patch.E2},\eqref{Patch.E5},\eqref{Patch.DefV}}}=V \begin{pmatrix} \nu_0(f_1f_{\frac12}xy)&0\\0&0 \end{pmatrix} V^* \nonumber\\
&\stackrel{\eqref{Patch.E2},\eqref{Patch.E5},\eqref{Patch.DefV}}= (\theta \otimes \id_{M_2})(U^*) \begin{pmatrix} 0&0\\v\nu_0(f_1f_{\frac12}x)\nu_0(y)&0 \end{pmatrix} V^*.
\end{align}
Putting these together and using (\ref{Patch.E21}) with $h:=f_1$, we obtain
\begin{align}
\lefteqn{\|\rho_1(x)\rho_{\frac12}(y)-\theta^{\oplus 2}(f_1)\rho_{\frac12}(xy)\|}\nonumber\\
&\leq  \|\rho_1(x)\rho_{\frac12}(y)-\rho_{\frac12}(f_1xy)\| \nonumber \\
& \qquad+\|\rho_{\frac12}(f_1xy)-\theta^{\oplus 2}(f_1)\rho_{\frac12}(xy)\|\nonumber\\
&\leq \varepsilon_1+\|\nu_1(f_{\frac12}f_1x)v-v\nu_0(f_{\frac12}f_1x)\|\nonumber\\
&\qquad +\|[v,\theta(f_1f_{\frac12})]\|+\|[v,\theta(f_{\frac12})]\|\nonumber\\
&= \|\nu_1(f_{\frac12}f_1x)v-v\nu_0(f_{\frac12}f_1x)\| \nonumber \\
& \qquad+3\|[v,\theta(f_1f_{\frac12})]\|+2\|[v,\theta(f_{\frac12})]\|.
\label{Patch.E15}
\end{align}

For the final case, we compute
\begin{align}
V\begin{pmatrix}\theta(f_{\frac12}f_1)&0\\0&0\end{pmatrix} \hspace*{-0.5em} &\stackrel{\eqref{Patch.E2},\eqref{Patch.E5},\eqref{Patch.DefV}}{_{\hphantom{\|[\theta(f_{\frac12}f_1),v]\|}}=_{\hphantom{\|[\theta(f_{\frac12}f_1),v]\|}}}(\theta\otimes \id_{M_2})(U^{*})\begin{pmatrix}0&0\\v\theta(f_{\frac12}f_1)&0\end{pmatrix}\nonumber\\
&\stackrel{\hphantom{\eqref{Patch.E2},\eqref{Patch.E5},\eqref{Patch.DefV}}}{_{\hphantom{\|[\theta(f_{\frac12}f_1),v]\|}}\approx_{\|[\theta(f_{\frac12}f_1),v]\|}}(\theta\otimes \id_{M_2})(U^{*})\begin{pmatrix}0&0\\\theta(f_{\frac12}f_1)v&0\end{pmatrix}\nonumber\\
&\stackrel{\eqref{Patch.E2},\eqref{Patch.E5},\eqref{Patch.DefV}}{_{\hphantom{\|[\theta(f_{\frac12}f_1),v]\|}}=_{\hphantom{\|[\theta(f_{\frac12}f_1),v]\|}}}\begin{pmatrix}\theta(f_{\frac12}f_1)v&0\\0&0\end{pmatrix}.
\label{Patch.E24}
\end{align}
Therefore, with $\varepsilon_1':=\|[v,\theta(f_{\frac12})]\|+\|[v,\theta(f_1f_{\frac12})]\|$, we have
\begin{align}
\rho_{\frac12}(x)\rho_1(y)
&\stackrel{\hphantom{\eqref{Patch.E24}}{}_{\hphantom{\|[\theta(f_{\frac12}f_1),v]\|}}}{=_{\hphantom{\|[\theta(f_{\frac12}f_1),v]\|}}} V\begin{pmatrix}\nu_0(f_{\frac12}x)&0\\0&0\end{pmatrix}V^*\begin{pmatrix}\nu_1(f_1y)&0\\0&0\end{pmatrix}\nonumber\\
&\stackrel{\eqref{Patch.E23a}{}_{\hphantom{\|[\theta(f_{\frac12}f_1),v]\|}}}{\approx_{\varepsilon'_{1}\hphantom{\|[f_{\frac12}f_1),v]\|}}} V\begin{pmatrix}\nu_0(x)&0\\0&0\end{pmatrix}\begin{pmatrix}\theta(f_{\frac12}f_1)v^*\nu_1(y)&0\\0&\theta(f_{\frac12}f_1)\end{pmatrix} \nonumber \\
&\stackrel{\hphantom{\eqref{Patch.E24}}{}_{\hphantom{\|[\theta(f_{\frac12}f_1),v]\|}}}{=_{\hphantom{\|[\theta(f_{\frac12}f_1),v]\|}}} V\begin{pmatrix}\theta(f_{\frac12}f_1)&0\\0&0\end{pmatrix}\begin{pmatrix}\nu_0(x)v^*\nu_1(y)&0\\0&0\end{pmatrix}\nonumber\\
&\stackrel{\eqref{Patch.E24}{}_{\hphantom{\|[\theta(f_{\frac12}f_1),v]\|}}}{\approx_{\|[\theta(f_{\frac12}f_1),v]\|}} \begin{pmatrix}\theta(f_1)\theta(f_{\frac12})v\nu_0(x)v^*\nu_1(y)&0\\0&0\end{pmatrix},\label{Patch.E26}
\end{align}
while
\begin{equation}
\theta^{\oplus 2}(f_{\frac12})\rho_{1}(xy) = \begin{pmatrix}\theta(f_1)\nu_1(f_{\frac12}x)\nu_1(y)&0\\0&0\end{pmatrix}.
\end{equation}
Putting these together yields the estimate
\begin{align}
\lefteqn{\|\rho_{\frac12}(x)\rho_1(y)-\theta^{\oplus 2} (f_{\frac12})\rho_1(xy)\|}\nonumber\\
&\leq \varepsilon'_1+\|[\theta(f_1f_{\frac12}),v]\| \nonumber \\
& \qquad +\|\theta(f_{\frac12})v\nu_0(x)v^*-\nu_1(f_{\frac12}x)\| \nonumber \\
&\leq \varepsilon'_1+\|[\theta(f_1f_{\frac12}),v)\| \nonumber \\
& \qquad +\|[\theta(f_{\frac12}),v]\|+\|v\nu_0(f_{\frac12}x)-\nu_1(f_{\frac12}x)v\|\nonumber\\
&= 2\|[\theta(f_{\frac12}),v]\|+2\|[\theta(f_1f_{\frac12}),v]\| \nonumber \\
& \qquad +\|v\nu_0(f_{\frac12}x)-\nu_1(f_{\frac12}x)v\|.
\label{Patch.E16}
\end{align}

In conclusion, (\ref{Patch.C8}) vanishes for $(i,j)=(0,1)$, $(1,0)$, $(0,0)$, $(1,1)$, $(\frac12,0)$ and in the other four cases $(i,j)=(\frac12,\frac12)$, $(0,\frac12)$, $(1,\frac12)$, $(\frac12,1)$, we have the estimates (\ref{Patch.E12}), (\ref{Patch.E13}), (\ref{Patch.E15}) and (\ref{Patch.E16}).   Thus
\begin{align}
\lefteqn{ \|\rho(x)\rho(y)-\rho(xy)\| }\nonumber\\
&\stackrel{(\ref{Patch.E3}), (\ref{Patch.E20})}{=} \Big\|\sum_{i,j=0,\frac12,1}\rho_i(x)\rho_j(y)-\sum_{i,j=0,\frac12,1}\theta^{\oplus 2}(f_i) \rho_j(xy) \Big\|\nonumber\\
&\stackrel{\phantom{(\ref{Patch.E3}), (\ref{Patch.E20})}}{\leq}  \sum_{i,j=0,\frac12,1}\|\rho_i(x)\rho_j(y)- \theta^{\oplus 2}(f_i) \rho_j(xy)\|\nonumber\\
\notag
&\stackrel{\phantom{(\ref{Patch.E3}), (\ref{Patch.E20})}}{\leq}  7\|[v,\theta(f_{\frac12})]\|+2\|[v,\theta(f_0f_{\frac12})]\|+5\|[v,\theta(f_{\frac12}f_1)]\| \\
\notag
&\phantom{\stackrel{\phantom{(\ref{Patch.E3}), (\ref{Patch.E20})}}{\leq}} \quad+\|\nu_1(f_{\frac12}f_1x)v-v\nu_0(f_{\frac12}f_1x)\|\nonumber\\
&\phantom{\stackrel{\phantom{(\ref{Patch.E3}), (\ref{Patch.E20})}}{\leq}} \qquad+\|\nu_1(f_{\frac12}x)v-v\nu_0(f_{\frac12}x)\|,
\end{align}
establishing (\ref{Patch.C4}) (recalling that $x$ and $y$ are assumed to be contractions).
\end{proof}

\section{Proof of Theorem \ref{thm:MainThm}}\label{Section:Proof}

\noindent
In this section we give the proof of our main result. By Remark \ref{QDNonUnital}  (\ref{QDNonUnital3}) (and since a $\mathrm{C}^{*}$-algebra satisfies the UCT if and only if its smallest unitisation does), it suffices to prove the theorem when $A$ is unital. So let $A$ be a separable, unital and nuclear $\mathrm{C}^*$-algebra in the UCT class and let $\tau_A\in T(A)$ be faithful.  Fix a finite subset $\mathcal F_A\subset A$ (which we may take to consist of self-adjoint contractions and contain $1_A$) and a tolerance $\varepsilon>0$.  We will produce $N\in\mathbb N$ and a completely positive map $\Psi:A\rightarrow  \Q_\omega\otimes M_{2N} \cong \Q_\omega$ such that
\begin{align}
&\|\Psi(ab)-\Psi(a)\Psi(b)\|<\varepsilon,\quad a,b\in \mathcal F_A,\label{Objective.1}\\
&\tfrac12\tau_A(a)=(\tau_{\Q_\omega}\otimes\tau_{M_{2N}})(\Psi(a)),\quad a\in A. \label{Objective.2}
\end{align}
Once this is achieved, quasidiagonality of $\tau_A$ follows from the characterisation in Proposition \ref{Prop.EasyQDT} (\ref{EasyQDT.3}).

Define functions $f,\grave g,\acute g, \grave g_{-1}, \acute g_{+1} \in C([0,1])$ by 
\begin{align}
f|_{[0,\frac19]}&\equiv 0, & f|_{[\frac29, \frac79]} \equiv 1, & \quad \qquad f|_{[\frac89,1]} \equiv 0, \nonumber\\
\notag
\grave g|_{[0,\frac89]}&\equiv 1, & \grave g(1)=0, \\
\notag
\acute g(0)&=0, & \acute g|_{[\frac19, 1]} \equiv 1, \\
\notag
\grave g_{-1}|_{[0,\frac29]}&\equiv 1, & \grave g_{-1}|_{[\frac13,1]}\equiv0, \\
\acute g_{+1}|_{[0,\frac23]}&\equiv 0, & \acute g_{+1}|_{[\frac79, 1]}\equiv 1,\label{Proof.Deffg}
\end{align}
and such that each function is linear on the complements of the intervals used to define it. These functions are shown in Figure \ref{NewFig}.

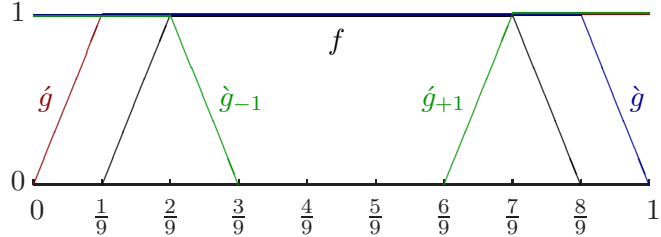
\begin{figure}[h!]
\begin{picture}(240,80)
\setlength{\unitlength}{0.15mm}
\linethickness{0.15mm}

\put(0,30){$0$}  % y label
\put(0,180){$1$}  % y label
\put(20,35){\line(1,0){540}}  % x axis
\put(20,35){\line(0,1){5}}  % x tick
\put(80,35){\line(0,1){5}}  % x tick
\put(140,35){\line(0,1){5}}  % x tick
\put(200,35){\line(0,1){5}}  % x tick
\put(260,35){\line(0,1){5}}  % x tick
\put(320,35){\line(0,1){5}}  % x tick
\put(380,35){\line(0,1){5}}  % x tick
\put(440,35){\line(0,1){5}}  % x tick
\put(500,35){\line(0,1){5}}  % x tick
\put(560,35){\line(0,1){5}}  % x tick

\put(17,5){$0$}  % x label
\put(72,0){$\frac19$}  % x label
\put(132,0){$\frac29$}  % x label
\put(192,0){$\frac39$}  % x label
\put(252,0){$\frac49$}  % x label
\put(312,0){$\frac59$}  % x label
\put(372,0){$\frac69$}  % x label
\put(432,0){$\frac79$}  % x label
\put(492,0){$\frac89$}  % x label
\put(557,5){$1$}  % x label
\color{b}
\put(20,35){\line(2,5){60}}  % \acute g
\put(80,186){\line(1,0){480}}  % \acute g
\put(24,105){$\acute g$}  % \acute g
\color{black}
\put(80,35){\line(2,5){60}}  % f
\put(140,184){\line(1,0){300}}  % f
\put(440,185){\line(2,-5){60}}  % f
\put(278,155){$f$}  % f
\color{c}
\put(20,185){\line(1,0){480}}  % \grave g
\put(500,185){\line(2,-5){60}}  % \grave g
\put(543,105){$\grave g$}  % \grave g
\color{a}
\put(20,184){\line(1,0){120}}  % \grave g_{-1}
\put(140,185){\line(2,-5){60}}  % \grave g_{-1}
\put(183,105){$\grave g_{-1}$}  % \grave g_{-1}
\color{a}
\put(380,35){\line(2,5){60}}  % \acute g_{+1}
\put(440,187){\line(1,0){120}}  % \acute g_{+1}
\put(360,105){$\acute g_{+1}$}  % \acute g_{+1}

\end{picture}
\caption{$f$, $\grave{g}$, $\acute{g}$, $\grave{g}_{-1}$, $\acute{g}_{+1}$}\label{NewFig}
\end{figure}

Define
\begin{align}
\notag
\mathcal F &:= \{f \otimes a \mid a \in \mathcal F_A\} \cup \{f \otimes ab \mid a,b \in \mathcal F_A\} \subset C_0((0,1),A),\label{DefF} \\
\mathcal F' &:= \{f,\grave g,\acute g,\grave g_{-1},\acute g_{+1}\} \subset C([0,1])
\end{align}
and choose
\begin{equation}\label{DefEta} 0 < \eta < \tfrac\varepsilon{24}. \end{equation}
As in Remark \ref{rmk:Patching}, use Lemma \ref{lem:Patching} to find a finite subset $\mathcal G\subset C_0((\frac13,\frac23),A)$ and $\delta>0$ such that Property \ref{Def-G} below is satisfied.
Note that we change from the matrix notation of Lemma \ref{lem:Patching} (\ref{Patch.C1}), (\ref{Patch.C2}) to $\oplus$-notation in (\ref{Def-G.1}), (\ref{Def-G.2}) below; see Notation \ref{notation}.

\begin{property}\label{Def-G}
Let $E$ be a unital $\mathrm{C}^*$-algebra with trace $\tau_E$. Let $\theta:C([0,1])\rightarrow E$ be a unital $^*$-homomorphism and let $\nu_0,\nu_1:C_0((0,1),A)\rightarrow E$ be compatible $^*$-homomorphisms satisfying $\tau_E\circ\nu_i=\tau_{\leb} \otimes \tau_A$ for $i=0,1$, where $\tau_{\leb}$ is the Lebesgue trace defined in (\ref{eq:LebTraceDef1}).
Suppose there exists a unitary $v \in E$ satisfying
\begin{equation}
\label{eq:Def-G-v}
v\nu_0(x)v^*\approx_\delta \nu_1(x),\quad x\in\mathcal G.
\end{equation}
Then there exists a completely positive map $\rho:C_0((0,1),A)\rightarrow M_2(E) \cong E \otimes M_{2}$ such that
\begin{enumerate}[(i)]
\item \label{Def-G.1}
$\rho|_{C_0((0,\frac13),A)} = \nu_0|_{C_0((0,\frac13),A)}\oplus 0_E$, 
\item \label{Def-G.2}
$\rho|_{C_0((\frac23,1),A)} =\nu_1|_{C_0((\frac23,1),A)}\oplus 0_E$, 
\item $(\tau_{E} \otimes \mathrm{Tr}_2)\circ\rho=\tau_{\leb} \otimes \tau_A$, \label{Def-G.3} 
\item $\rho(xy)\approx_\eta\rho(x)\rho(y)$, for $x,y\in \mathcal F$, and \label{Def-G.4}
\item $\rho(hx)\approx_\eta \theta^{\oplus 2}(h)\rho(x)$, for $x\in\mathcal F$, $h\in\mathcal F'$. \label{Def-G.5}
\end{enumerate}
\end{property}

Define 
\begin{equation}
 C:= \{x\in C([0,1],A) \mid \exists \lambda \in \mathbb C \text{ s.t.\ }x|_{[0,\frac 13]}\equiv x|_{[\frac23,1]}\equiv\lambda 1_{A} \}.
\label{eq:Cdef}
\end{equation}
As an abstract algebra, $C$ is isomorphic to $C_0((\frac13,\frac23),A)^\sim$, and so, as noted after Definition \ref{UCT:def}, $C$ satisfies the UCT.  However we regard $C$ as the subalgebra of $C([0,1],A)$ given in (\ref{eq:Cdef}) so that $\tau_{\leb}\otimes \tau_A$ restricts to a trace on $C$.

Since $\tau_A$ and $\tau_{\leb}$ are faithful traces, so too is $\tau_{\leb}\otimes\tau_A$ by Kirchberg's slice lemma (see \cite[Lemma 4.1.9]{R:Book} in R{\o}rdam's book for a published version). Thus we may define a control function $\Delta:C^1_+ \setminus \{0\} \to \mathbb N$ by taking $\Delta(c)$ to be a square number such that 
\begin{equation}
\label{eq:DeltaDef}
(\tau_{\leb} \otimes \tau_A)(c) > \frac{2}{\sqrt{\Delta(c)}}.
\end{equation}
Apply the stable uniqueness theorem (Theorem \ref{thm:StableUniqueness}) to $C$ with respect to the control function $\Delta$, and the finite set $\mathcal G$ and tolerance $\delta$ of Property \ref{Def-G}, to obtain $n\in\mathbb N$ satisfying Property \ref{Def-N} below.
Note that since $C$ is nuclear, so are all $^*$-homomorphisms whose domain is $C$.
Also recall that $\Q_\omega$ is an admissible target algebra of finite type.

\begin{property}\label{Def-N}
Let $D$ be a $\mathrm C^*$-algebra isomorphic to $\Q_\omega$, $\iota:C\rightarrow D$ a $\Delta$-full unital $^*$-homomorphism and $\phi,\psi:C\rightarrow D$ unital $^*$-homomorphisms satisfying $\phi_*=\psi_*:\underline{K}(C)\rightarrow\underline{K}(D)$. Then there exists a unitary $u\in M_{n+1}(D)$ such that
\begin{equation}
u(\phi(x)\oplus\iota^{\oplus n}(x))u^*\approx_{\delta} \psi(x)\oplus\iota^{\oplus n}(x),\quad x\in\mathcal G.
\end{equation}
\end{property}

In terms of this fixed integer $n$, define integers
\begin{equation}
N:= 2n, \quad m:=4n+1,
\end{equation}
and relatively open intervals $I_i\subset [0,1]$ for $i=0,\dots,N+1$ by
\begin{equation}
I_i:=\left(\tfrac{2i-2}m, \tfrac{2i+1}m\right) \cap [0,1], \quad i=0,\dots,N+1.
\end{equation}
In this way $I_0=[0,\tfrac1m)$ and $I_{N+1}=(1-\tfrac1m,1]$. Note that
\begin{equation}\label{e5.12}
I_i\cap I_j=\emptyset,\quad i,j=0,\dots,N+1,\ |i-j|>1.
\end{equation}
The interval $I_i$ has length $\tfrac3m$ when $i=1,\dots,N$.

For $i=1,\dots,N$, define functions $\alpha_i:[0,1] \to [0,1]$ by:
\begin{equation}\label{Proof.DefAlpha} \alpha_i|_{[0,\frac{2i-2}m]} \equiv 0, \quad \alpha_i|_{[\frac{2i+1}m,1]} \equiv 1,\quad \alpha_i|_{[\frac{2i-2}m,\frac{2i+1}m]}\text{ is linear.} \end{equation}
Note that $\alpha_i$ maps $I_i$ to $(0,1)$ by a linear homeomorphism, stretching by $\frac{m}3$.
Hence, let us record for future use,
\begin{equation}
\label{eq:alphaTr}
(\tau_{\leb} \otimes \tau_A)(x\circ\alpha_i) = \frac 3m (\tau_{\leb} \otimes \tau_A)(x), \quad x \in C_0((0,1),A).
\end{equation}
We then define positive contractions in $C([0,1])$ by
\begin{equation}\label{Deffi}
f_i:=f \circ\alpha_i, \quad \grave g_i:=\grave g \circ \alpha_i, \quad \acute g_i:= \acute g \circ \alpha_i,\quad i=1,\dots,N,
\end{equation}
where $f,\grave g, \acute g$ are defined in \eqref{Proof.Deffg}. Set $\grave{g}_0:=\grave{g}_{-1}\circ\alpha_1$ and $\acute{g}_{N+1}:=\acute{g}_{+1} \circ \alpha_N$.  These definitions ensure that
\begin{align}
\grave{g}_{i-1}&=\grave{g}_{-1}\circ\alpha_i,\quad\text{and}\label{Deffia.1}\\
\acute{g}_{i+1}&=\acute{g}_{+1}\circ\alpha_i,\quad i=1,\dots,N\label{Deffia.2}.
\end{align}

Define $f_0$ and $f_{N+1}$ by
\begin{align}
\notag
f_0|_{[0,\frac1{3m}]} &\equiv 1, & f_0|_{[\frac2{3m},1]} \equiv 0,\\
f_{N+1}|_{[0,1-\frac{2}{3m}]} &\equiv 0, & f_{N+1}|_{[1-\frac1{3m},1]} \equiv 1,
\end{align}
and again demanding that $f_0$ and $f_{N+1}$ are linear on the complements of the intervals of definition above. Note that 
\begin{equation}f_i\in C_0(I_i), \quad i=0,\dots,N+1.\label{finCOIi}
\end{equation} These functions and intervals are displayed in Figure \ref{FigI}.

\begin{figure}[h!]
\begin{picture}(350,130)
\setlength{\unitlength}{0.15mm}
\linethickness{0.15mm}

\put(0,100){$0$}  % y label
\put(0,202){$1$}  % y label
\put(20,105){\line(1,0){765}}  % x axis
\put(20,105){\line(0,1){5}}  % x tick
\put(71,105){\line(0,1){5}}  % x tick
\put(224,105){\line(0,1){5}}  % x tick
\put(275,105){\line(0,1){5}}  % x tick
\put(326,105){\line(0,1){5}}  % x tick
\put(377,105){\line(0,1){5}}  % x tick
\put(428,105){\line(0,1){5}}  % x tick
\put(479,105){\line(0,1){5}}  % x tick
\put(530,105){\line(0,1){5}}  % x tick
\put(581,105){\line(0,1){5}}  % x tick
\put(734,105){\line(0,1){5}}  % x tick
\put(785,105){\line(0,1){5}}  % x tick
\put(11,75){$0$}  % x label
\put(59,70){$\tfrac1m$}  % x label
\put(202,70){$\tfrac{2i-4}m$}  % x label
\put(304,70){$\tfrac{2i-2}m$}  % x label
\put(355,70){$\tfrac{2i-1}m$}  % x label
\put(416,70){$\tfrac{2i}m$}  % x label
\put(457,70){$\tfrac{2i+1}m$}  % x label
\put(559,70){$\tfrac{2i+3}m$}  % x label
\put(698,70){$1\hspace*{-1mm}-\hspace*{-1mm}\tfrac1m$}  % x label
\put(779,75){$1$}  % x label
\put(126,75){$\cdots$}  % x label
\put(636,75){$\cdots$}  % x label
\put(20,208){\line(1,0){17}}  % f_0 
\put(37,207){\line(1,-6){17}}  % f_0 
\put(23,222){$f_0$}  % f_0 
\color{a}
\put(37,105){\line(1,6){17}}  % f_1 
\put(54,208){\line(1,0){51}}  % f_1 
\put(105,207){\line(1,-6){17}}  % f_1 
\put(67,222){$f_1 $}  % f_1 
\color{a}
\put(241,105){\line(1,6){17}}  % f_{i-1}
\put(258,208){\line(1,0){85}}  % f_{i-1}
\put(343,207){\line(1,-6){17}}  % f_{i-1}
\put(278,222){$f_{i-1}$}  % f_{i-1}
\color{black}
\put(343,105){\line(1,6){17}}  % f_i
\put(360,208){\line(1,0){85}}  % f_i
\put(445,207){\line(1,-6){17}}  % f_i
\put(388,222){$f_i$}  % f_i
\color{a}
\put(445,105){\line(1,6){17}}  % f_{i+1}
\put(462,207){\line(1,0){85}}  % f_{i+1}
\put(547,207){\line(1,-6){17}}  % f_{i+1}
\put(482,222){$f_{i+1}$}  % f_{i+1}
\color{b}
\put(326,105){\line(1,6){17}}  % \acute g_i
\put(343,206){\line(1,0){442}}  % \acute g_i
\put(303,149){$\acute g_i$}  % \acute g_i
\color{c}
\put(20,207){\line(1,0){442}}  % \grave g_i
\put(462,207){\line(1,-6){17}}  % \grave g_i
\put(486,149){$\grave g_i$}  % \grave g_i
\color{a}
\put(683,105){\line(1,6){17}}  % f_N
\put(700,207){\line(1,0){51}}  % f_N
\put(751,207){\line(1,-6){17}}  % f_N
\put(706,222){$f_N$}  % f_N
\color{black}
\put(751,105){\line(1,6){17}}  % f_{N+1}
\put(768,207){\line(1,0){17}}  % f_{N+1}
\put(768,222){$f_{N+1}$}  % f_{N+1}
\color{black}
\put(227,40){\line(1,0){146}}  % I_{i-1}
\put(273,5){$I_{i-1}$}  % I_{i-1}
\put(224,40){\circle{7}}  % I_{i-1}
\put(377,40){\circle{7}}  % I_{i-1}
\put(330,50){\line(1,0){146}}  % I_i
\put(396,10){$I_i$}  % I_i
\put(326,50){\circle{7}}  % I_i
\put(479,50){\circle{7}}  % I_i
\put(431,40){\line(1,0){146}}  % I_{i+1}
\put(492,5){$I_{i+1}$}  % I_{i+1}
\put(428,40){\circle{7}}  % I_{i+1}
\put(581,40){\circle{7}}  % I_{i+1}
\put(23,50){\line(1,0){44}}  % I_0
\put(18,15){$I_0$}  % I_0
\put(20,50){\circle*{7}}  % I_0
\put(71,50){\circle{7}}  % I_0
\put(737,50){\line(1,0){44}}  % I_{N+1}
\put(732,15){$I_{N+1}$}  % I_{N+1}
\put(734,50){\circle{7}}  % I_{N+1}
\put(785,50){\circle*{7}}  % I_{N+1}

\end{picture}
\caption{The intervals $I_i$ and functions $f_i$, $\grave g_i$, $\acute g_i$.}\label{FigI}
\end{figure}
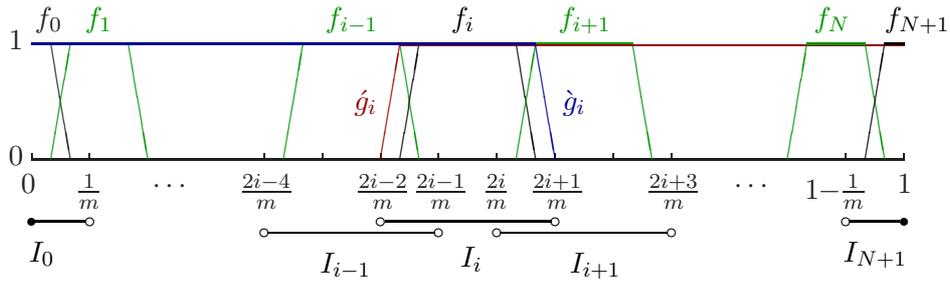

By construction, 
\begin{equation}
\label{eq:fiPOU}
\sum_{i=0}^{N+1} f_i = 1_{C([0,1])}, \end{equation}
and in fact, 
\begin{equation}
\label{eq:fiLocalPOU}
(f_{i-1}+f_i+f_{i+1})|_{I_i} \equiv 1,\quad i=0,\dots,N+1,
\end{equation}
where we define $f_{-1}:=0=:f_{N+2}$ to handle the end point cases.  Also each $\grave g_i$ and $\acute g_i$ acts as a unit on the corresponding $f_i$, i.e.
\begin{align}
f_i&= \grave g_if_i=f_i\grave g_i,\quad i=0,\dots,N,\notag\\
f_i&=\acute g_if_i=f_i\acute g_i,\quad i=1,\dots,N+1.\label{eq:giUnitfi} 
\end{align}
Further, for $i=1,\dots,N$,
\begin{equation}
\label{eq:figiProd}
f_i\grave g_{i-1} \in C_0(\left(\tfrac{{\scriptsize 2i-2}}m,\tfrac{2i-1}m\right))\text{ and }
\acute g_{i+1}f_i \in C_0(\left(\tfrac{2i}m,\tfrac{2i+1}m\right)).
\end{equation}
For $i=1,\dots,N$, write
\begin{align}\label{eq:Falphai2}
\mathcal F_i&:=\{f_i\otimes a \mid a\in\mathcal F_A\}\cup\{f_i\otimes ab \mid a,b\in\mathcal F_A\} \quad \text{and} \\
\mathcal F_i'&:=\{f_i,\grave{g}_i,\acute{g}_i,\grave{g}_{i-1},\acute{g}_{i+1}\};\label{eq:F'alphai2}
\end{align}
slightly abusing notation, by (\ref{DefF}),  (\ref{Deffi}), (\ref{Deffia.1}) and (\ref{Deffia.2}), we have
\begin{equation}\label{eq:Falphai}
 \mathcal F_i = \mathcal F \circ \alpha_i\quad\text{and}\quad \mathcal F_i' = \mathcal F' \circ \alpha_i.
\end{equation}

Use Lemma \ref{MapsExist} to find a unital $^*$-homomorphism $\Theta:C([0,1])\rightarrow\Q_\omega$ and compatible $^*$-homomorphisms $\grave{\Phi}:C_0([0,1),A)\rightarrow \Q_\omega$, $\acute{\Phi}:C_0((0,1],A)\rightarrow\Q_\omega$ satisfying
\begin{align}
\tau_{\Q_\omega}\circ \grave{\Phi}&=\tau_{\leb} \otimes \tau_A \quad \text{and}\label{DefGrave}\\
\tau_{\Q_\omega}\circ \acute{\Phi}&=\tau_{\leb} \otimes \tau_A.\label{DefAcute}
\end{align}
Then define $^*$-homomorphisms $\sigma_0,\dots,\sigma_N$ by
\begin{align}
\sigma_0&:=\grave{\Phi}^{\oplus N}:C_0([0,1),A)\rightarrow \Q_\omega\otimes M_N,\nonumber\\
\sigma_{N}&:=\acute{\Phi}^{\oplus N}:C_0((0,1],A)\rightarrow \Q_\omega \otimes M_N,\text{ and}\nonumber\\
\label{eq:sigmaDef}
\sigma_i&:=\grave\Phi|_{C_0((0,1),A)}^{\oplus(N-i)} \oplus \acute\Phi|_{C_0((0,1),A)}^{\oplus i}: C_0((0,1),A) \to \Q_\omega \otimes M_{N}
\end{align}
for $i=1,\ldots,N-1$. 
Note that while $\sigma_1,\dots,\sigma_{N-1}$ are defined on the suspension $C_0((0,1),A)$, $\sigma_0$ and $\sigma_{N}$ are defined on the cones $C_0([0,1),A)$ and $C_0((0,1],A)$ respectively.

By applying Lemma \ref{lem:CharFunction} to the positive contraction $\Theta(\id_{[0,1]})$, and the intervals $I_0,\dots,I_{N+1}$, there exist projections $q_0,\dots,q_{N+1}\in \Q_\omega$ such that Property \ref{Propertyq} below is satisfied.  For condition (\ref{Propertyq.2}) below, the corresponding property in Lemma \ref{lem:CharFunction} ensures that $q_i$ acts as a unit on $\Theta(C_0(I_i))$, and therefore on the hereditary $\mathrm C^*$-subalgebra generated by $\Theta(C_0(I_i))$, which contains $\grave\Phi(C_0(I_i),A)$ and $\acute\Phi(C_0(I_i),A)$ by compatibility. Condition (\ref{Propertyq.3}) below follows from the corresponding property in Lemma \ref{lem:CharFunction} and (\ref{eq:fiLocalPOU}), while condition (\ref{Propertyq.4}) uses the corresponding property from Lemma \ref{lem:CharFunction} and (\ref{e5.12}).
\begin{property}\label{Propertyq}
\begin{enumerate}[(i)]
\item Each $q_i$ commutes with the image of $\Theta$.\label{Propertyq.1}
\item For each $i=1,\dots,N$, the projection $q_i$ acts as a unit on $\Theta(C_0(I_i))$, $\grave{\Phi}(C_0(I_i,A))$ and $\acute{\Phi}(C_0(I_i,A))$, and hence $q_i^{\oplus N}$ acts as a unit on the image of $\sigma_j|_{C_0(I_i,A)}$ for all $j$.
Likewise, $q_0$ acts as a unit on $\Theta(C_0(I_0))$ and $\grave\Phi(C_0(I_0,A))$, while $q_{N+1}$ acts as a unit on both $\Theta(C_0(I_{N+1}))$ and $\acute\Phi(C_0(I_{N+1},A))$.
\label{Propertyq.2}
\item For each $i=0,\dots,N+1$, $\Theta(f_{i-1}+f_i+f_{i+1})$ acts as a unit on $q_i$.\label{Propertyq.3}
\item $q_iq_j=0$ for $|i-j|>1$.\label{Propertyq.4}
\item
\begin{equation}\label{Proof.Traceq_i}
\tau_{\Q_\omega}(q_i)=|I_i|=\begin{cases}\frac3m,&i=1,\dots,N;\\\frac1m,&i=0\text{ or }i=N+1.\end{cases}
\end{equation}
\end{enumerate}
\end{property}

We will now use the patching lemma of the previous section to find a family of approximately multiplicative, approximately compatible maps $\rho_0,\dots,\rho_{N+1}$ such that $\rho_i$ agrees with $\sigma_{i-1}$ on the left-hand third of $I_i$, i.e., $(\frac{2i-2}{m},\frac{2i-1}{m})$ and agrees with $\sigma_i$ on the right-hand third of $I_i$, namely $(\frac{2i}{m},\frac{2i+1}{m})$. Our precise objective is set out in the following claim.

\begin{claim}
\label{MainClaim}
There exist completely positive maps
\begin{equation}\label{DefRhoi}
\rho_i : C_0(I_i,A) \to q_{i}\Q_\omega q_{i} \otimes M_N \otimes M_2, \quad i=0,\dots,N+1
\end{equation}
such that:
\begin{enumerate}[(i)] 
\item For $i=0,\dots,N+1$, $(\tau_{\Q_\omega} \otimes \tau_{M_{2N}}) \circ \rho_i = \frac12 \tau_{\leb} \otimes \tau_A$. \label{MainClaim.1}
\item As maps into $\Q_\omega \otimes M_N \otimes M_2$, we have
\begin{align}
\label{eq:rho0}
\rho_0&=\sigma_0|_{C_0([0,\frac1m),A)} \oplus 0_N, \\
\label{eq:rhoN+1}
\rho_{N+1}&=\sigma_N|_{C_0((1-\frac1m,1],A)} \oplus 0_N,\\
\label{eq:rhoiRestrict1}
\rho_i|_{C_0((\frac{2i-2}m,\frac{2i-1}m),A)} &= \sigma_{i-1}|_{C_0((\frac{2i-2}m,\frac{2i-1}m),A)} \oplus 0_N, \\
\label{eq:rhoiRestrict2}
\rho_i|_{C_0((\frac{2i}m,\frac{2i+1}m),A)} &= \sigma_{i}|_{C_0((\frac{2i}m,\frac{2i+1}m),A)} \oplus 0_N,\quad i=1,\dots,N.
\end{align}
\item For $i=1,\dots,N$, $x,y \in \mathcal F_i$, and $h \in \mathcal F_i'$,
\begin{align}
\label{eq:rhoMult}
\rho_i(xy) &\approx_\eta \rho_i(x)\rho_i(y), \\
\label{eq:rhoLin}
\rho_i(x)\Theta^{\oplus 2N}(h) &\approx_\eta \rho_i(xh) \approx_\eta \Theta^{\oplus 2N}(h)\rho_i(x),
\end{align}
while $\rho_0$ and $\rho_{N+1}$ are $^*$-homomorphisms compatible with $\Theta^{\oplus 2N}$.
\end{enumerate}
\end{claim}

\begin{proof}[Proof of claim]
The maps $\rho_0$ and $\rho_{N+1}$ are already prescribed by \eqref{eq:rho0} and \eqref{eq:rhoN+1}; since $q_0^{\oplus N}$ and $q_{N+1}^{\oplus N}$ act as units on the images of $\sigma_0|_{C_0(I_0,A)}$ and $\sigma_{N+1}|_{C_0(I_{N+1},A)}$ respectively by Property \ref{Propertyq} (\ref{Propertyq.2}), $\rho_0$ and $\rho_{N+1}$ have the codomains required by  (\ref{DefRhoi}). Further, \eqref{eq:rho0} and \eqref{eq:rhoN+1} ensure that $\rho_0$ and $\rho_{N+1}$ are $^*$-homomorphisms compatible with $\Theta^{\oplus 2N}$ which satisfy (\ref{MainClaim.1}). 

Fix $i\in\{1,\dots,N\}$.
Using the subalgebra $C$ of $C([0,1],A)$ from \eqref{eq:Cdef}, define $\phi,\psi:C \to q_i\Q_\omega q_i$ by
\begin{equation}
\label{eq:phiDef}
 \phi(c) := \lambda q_i + \grave\Phi(x\circ \alpha_i) \quad \text{and}\quad \psi(c):= \lambda q_i + \acute\Phi(x \circ \alpha_i) \end{equation}
for $c=x+\lambda 1_C$ where $x \in C_0((\tfrac13,\tfrac23),A)$. As $x\circ\alpha_i\in C_0(I_i)$, Property \ref{Propertyq} (\ref{Propertyq.2}) shows that $\phi$ and $\psi$ define unital $^*$-homomorphisms.

We will take
\begin{equation}
\iota:= \begin{cases} \phi, \quad & i \leq n; \\ \psi, \quad &i > n \end{cases} \end{equation}
in Property \ref{Def-N} so let us now show that both $\phi$ and $\psi$ are $\Delta$-full maps.

The unique normalised trace $\tau_{q_i\Q_\omega q_i}$ on $q_i\Q_\omega q_i$ is given by 
\begin{equation}
\label{eq:CornerTrace}
\tau_{q_i\Q_\omega q_i}(z)=\frac{1}{\tau_{\Q_\omega}(q_i)}\tau_{\Q_\omega}(z) \stackrel{\eqref{Proof.Traceq_i}}= \frac m3 \tau_{\Q_\omega}(z),\quad z\in q_i\Q_\omega q_i.
\end{equation}
For $c=x+\lambda 1_C \in C_+ \setminus \{0\}$ where $x\in C_0((\tfrac13,\tfrac23),A)$ and $\lambda \in \mathbb C$, we have
\begin{eqnarray}
\tau_{q_i\Q_\omega q_i}(\phi(c)) &\stackrel{\eqref{eq:phiDef}}=& \tau_{q_i\Q_\omega q_i}(\lambda 1_{q_i\Q_\omega q_i} + \grave\Phi(x\circ\alpha_i)) \nonumber\\
&\stackrel{\eqref{eq:CornerTrace}}=& \lambda + \frac m3\tau_{\Q_\omega}(\grave\Phi(x\circ\alpha_i)) \nonumber\\
&\stackrel{\eqref{DefGrave}}{=}& \lambda+\frac{m}{3}(\tau_{\leb} \otimes \tau_A)(x\circ\alpha_i)\nonumber\\
&\stackrel{\eqref{eq:alphaTr}}=& \lambda+(\tau_{\leb} \otimes \tau_A)(x) \nonumber\\
&\stackrel{\eqref{eq:Cdef}}{=}& (\tau_{\leb}\otimes \tau_A)(c) \nonumber\\
&\stackrel{\eqref{eq:DeltaDef}}{>}& \frac{2}{\sqrt{\Delta(c)}}.\label{Proof.TracePhi}
\end{eqnarray}
Using (\ref{DefAcute}) in place of (\ref{DefGrave}), and otherwise computing identically, we also have
\begin{equation}\label{Proof.TracePsi}
\tau_{q_i\Q_\omega q_i}(\psi(c))=(\tau_{\leb}\otimes\tau_A)(c) > \frac2{\sqrt{\Delta(c)}}.
\end{equation}
Entering (\ref{Proof.TracePhi}) and (\ref{Proof.TracePsi}) into Lemma \ref{lem:TrFull} (noting that $\sqrt{\Delta(c)} \in \mathbb N$) shows that $\phi$ and $\psi$ are $\Delta$-full.

Since $\acute\Phi$ is contractible, so is $\acute\Phi|_{C_0(I_i,A)}:C_0(I_i,A) \to \Q_\omega$. By Proposition \ref{prop:Qfacts} (\ref{prop:Qfacts.2}), there exist $k\in\mathbb N$ and a (not necessarily unital) embedding $\Q_\omega\hookrightarrow q_i\Q_\omega q_i\otimes M_k$ which maps $q_izq_i\in \Q_\omega$ to $q_izq_i\otimes e_{11}$.  Therefore, defining $\phi_0,\psi_0:C_0((\frac13,\frac23),A)\rightarrow q_i\Q_\omega q_i\otimes M_k$ by
\begin{align}
\phi_0(x)&:=\grave{\Phi}(x\circ\alpha_i)\otimes e_{11},\quad x\in C_0((\tfrac13,\tfrac23),A),\nonumber\\
\psi_0(x)&:=\acute{\Phi}(x\circ\alpha_i)\otimes e_{11},\quad x\in C_0((\tfrac13,\tfrac23),A),
\end{align}
$\phi_0$ and $\psi_0$ are homotopic.\footnote{The point here is that while $\acute{\Phi}$ and $\grave{\Phi}$ restrict to contractible and hence homotopic maps $C_0(I_i,A)\rightarrow \Q_\omega$, it doesn't simply follow that these restrictions are homotopic within the space of $^*$-homomorphisms $C_0(I_i,A) \to q_i\Q_\omega q_i$.}
Thus the unitisations of $\phi_0$ and $\psi_0$ are homotopic, hence induce the same map on $\underline{K}$, and $\phi_*=\psi_*$.\footnote{The maps $\phi_0$ and $\psi_0$ have unitisations $\phi\oplus\pi^{\oplus(k-1)}$ and $\psi\oplus\pi^{\oplus(k-1)}$ respectively, where $\pi:C\to q_i\Q_\omega q_i$ is the canonical unital $^*$-homomorphism  $C\to q_i\Q_\omega q_i$ with kernel $C_0((\frac13,\frac23),A)$. Thus $\phi_*+(k-1)\pi_*=\psi_*+(k-1)\pi_*$ and subtracting $(k-1)\pi_*$ from both sides gives $\phi_*=\psi_*$.}

By Property \ref{Def-N}, there exists a unitary $u\in M_{n+1}(q_i\Q_\omega q_i)$ such that
\begin{equation}\label{ProofAppDefN}
u(\phi(x) \oplus \iota^{\oplus n}(x))u^* \approx_\delta \psi(x) \oplus \iota^{\oplus n}(x), \quad x\in \mathcal G. \end{equation}
For $i\leq n$, this is
\begin{equation}
u(\phi^{\oplus (n+1)}(x))u^*\approx_\delta\psi(x)\oplus\phi^{\oplus n}(x),\quad x\in\mathcal G.
\end{equation}
Thus working in $q_i\Q_\omega q_i\otimes M_N$, we have
\begin{align}
\lefteqn {(u\phi^{\oplus(n+1)}(x)u^*)\oplus \phi^{\oplus(N-i-n)}(x)\oplus\psi^{\oplus (i-1)}(x)} \nonumber\\
&\approx_\delta  \psi(x)\oplus\phi^{\oplus n}(x)\oplus \phi^{\oplus(N-i-n)}(x)\oplus\psi^{\oplus(i-1)}(x),\quad x\in\mathcal G.\label{Proof.Ver4.1}
\end{align}
Applying appropriate permutation unitaries to both sides of (\ref{Proof.Ver4.1}) we obtain a unitary $v\in q_i\Q_\omega q_i \otimes M_{N}$ with
\begin{equation}\label{Proof.PatchingCondition}
v(\phi^{\oplus(N-i+1)}(x)\oplus\psi^{\oplus(i-1)}(x))v^*\approx_\delta\phi^{\oplus(N-i)}(x)\oplus\psi^{\oplus i}(x),\quad x\in \mathcal G.
\end{equation}
When $i>n$, we have $\iota=\psi$, so that adding $\phi^{\oplus(N-i)}(x)\oplus\psi^{\oplus(i-n-1)}(x)$ to both sides of (\ref{ProofAppDefN}) and applying appropriate permutation unitaries in $q_i\Q_\omega q_i\otimes M_{N}$ likewise gives a unitary $v\in q_i\Q_\omega q_i\otimes M_N$ satisfying (\ref{Proof.PatchingCondition}).

We next wish to use Property \ref{Def-G}, taking $E:=q_{i} \mathcal{Q}_{\omega} q_{i} \otimes M_{N}$. To this end, we define maps 
$\nu_0,\nu_1:C_0((0,1),A) \to q_i\Q_\omega q_i \otimes M_N$
by
\begin{equation}
\label{eq:nuDef}
\nu_j(x):=\sigma_{i+j-1}(x \circ \alpha_i),\quad x\in C_0((0,1),A),\ j=0,1,
\end{equation}
noting that the codomain can indeed be taken to be $q_i\Q_\omega q_i \otimes M_N$ as $\alpha_i$ maps $I_i$ homeomorphically onto $(0,1)$ by \eqref{Proof.DefAlpha}, and $q_i^{\oplus N}$ acts as a unit on the image of $\sigma_{i-1}|_{C_0(I_i,A)}$ and $\sigma_i|_{C_0(I_i,A)}$ by Property \ref{Propertyq} (\ref{Propertyq.2}).
By construction
\begin{align}
\nu_0|_{C_0((\frac13,\frac23),A)}&= (\phi^{\oplus(N-i+1)}\oplus\psi^{\oplus(i-1)})|_{C_0((\frac13,\frac23),A)}, \\
\nu_1|_{C_0((\frac13,\frac23),A)}&= (\phi^{\oplus(N-i)}\oplus\psi^{\oplus i})|_{C_0((\frac13,\frac23),A)}.
\end{align}
Since $\mathcal G \subset C_0((\tfrac13, \tfrac23),A)$, it follows that the unitary $v$ giving \eqref{Proof.PatchingCondition} verifies \eqref{eq:Def-G-v} of Property \ref{Def-G}.

Define $\theta:C([0,1]) \to q_i\Q_\omega q_i \otimes M_N$ by 
\begin{equation}
\label{eq:thetaDef}
 \theta(h) := (q_i\Theta(h \circ \alpha_i)q_i)^{\oplus N}, \quad h \in C([0,1]); \end{equation}
this is a unital $^*$-homomorphism by Property \ref{Propertyq} (\ref{Propertyq.1}).
Since $\acute\Phi,\grave\Phi$ are compatible with $\Theta$, the definitions in \eqref{eq:sigmaDef} and \eqref{eq:nuDef} ensure that $\nu_0$ and $\nu_1$ are compatible with $\theta$.
For $x \in C_0((0,1),A)$ and $j=0,1$, we compute
\begin{align}
\notag
(\tau_{q_i\Q_\omega q_i} \otimes \tau_{M_N})(\nu_j(x))
&\stackrel{\eqref{eq:CornerTrace},\eqref{eq:nuDef}}= \frac m3 (\tau_{\Q_\omega} \otimes\tau_{M_N})(\sigma_{i+j-1}(x\circ\alpha_i)) \\
\notag
&\stackrel{\hphantom{\eqref{eq:CornerTrace},\eqref{eq:nuDef}}}{\stackrel{\eqref{eq:sigmaDef}}=} \frac m{3N} \tau_{\Q_\omega}\big((N-i-j+1)\grave\Phi(x\circ\alpha_i)\\
\notag
&\stackrel{\hphantom{\eqref{eq:CornerTrace},\eqref{eq:nuDef}}}{\stackrel{\hphantom{\eqref{eq:sigmaDef}}}{\hphantom{=}}}\qquad +(i+j-1)\acute\Phi(x\circ\alpha_i)\big) \\
\notag
&\stackrel{\hphantom{\eqref{eq:CornerTrace},\eqref{eq:nuDef}}}{\stackrel{\eqref{DefGrave},\eqref{DefAcute}}=} \frac m{3N}\, N(\tau_{\leb}\otimes \tau_A)(x \circ \alpha_i) \\
&\stackrel{\hphantom{\eqref{eq:CornerTrace},\eqref{eq:nuDef}}}{\stackrel{\eqref{eq:alphaTr}}= }(\tau_{\leb}\otimes \tau_A)(x).
\end{align}
This completes the verification of the hypotheses in Property \ref{Def-G}, and so we obtain a completely positive map 
\begin{equation}\label{e5.54}
\rho:C_0((0,1),A)\rightarrow q_i\Q_\omega q_i \otimes M_{N} \otimes M_{2} =E\otimes M_2
\end{equation} 
satisfying (\ref{Def-G.1})-(\ref{Def-G.5}) of Property \ref{Def-G}.
Define a completely positive map $\rho_i:C_0(I_i,A) \to q_i\Q_\omega q_i \otimes M_N \otimes M_2$ by
\begin{equation}\label{e5.55} \rho_i(x) := \rho(x \circ \alpha_i|_{I_i}^{-1}),\quad x\in C_0(I_i,A),\end{equation}
(which makes sense as $\alpha_i|_{I_i}:I_i \to (0,1)$ is a homeomorphism).

To end the proof of the claim, we link Property \ref{Def-G} to the conditions of Claim \ref{MainClaim}.
Property \ref{Def-G} (\ref{Def-G.3}), \eqref{eq:alphaTr}, and \eqref{eq:CornerTrace} combine to give Claim \ref{MainClaim} (\ref{MainClaim.1}), while \eqref{eq:nuDef}, Property \ref{Def-G} (\ref{Def-G.1}) and (\ref{Def-G.2}) give  \eqref{eq:rhoiRestrict1} and \eqref{eq:rhoiRestrict2} respectively.
Using \eqref{eq:Falphai}, \eqref{eq:nuDef}, and \eqref{eq:thetaDef}, the estimates for $\rho$ in Properties \ref{Def-G} (\ref{Def-G.4}), (\ref{Def-G.5}) in terms of $\mathcal F$ and $\mathcal F'$ are transformed to give the required estimates \eqref{eq:rhoMult}, \eqref{eq:rhoLin} for the sets $\mathcal F_i$ and $\mathcal F_i'$ (for \eqref{eq:rhoLin} also using the fact that $\mathcal F_A$ consists of self-adjoints so that $\mathcal F$ and $\mathcal F_i$ are closed under taking adjoints). This concludes the proof of the claim.
\end{proof}

We now define the completely positive map 
\begin{equation}
\Psi:A \to \Q_\omega \otimes M_N \otimes M_2
\end{equation} 
to be used to witness quasidiagonality of $\tau_A$ (via \eqref{Objective.1} and \eqref{Objective.2}) by
\begin{equation}\label{e5.57}
\Psi(a) := \sum_{i=0}^{N+1} \rho_i(f_i \otimes a),\quad a\in A.
\end{equation}
Note that $\rho_i(f_i \otimes a)$ makes sense since $f_i \in C_0(I_i)$.
Condition (\ref{MainClaim.1}) of Claim \ref{MainClaim} gives $(\tau_{Q_\omega} \otimes \tau_{M_{2N}})(\rho_i(f_i\otimes a))=\tfrac12\tau_{\leb}(f_i)\tau_A(a)$, for each $i \in \{0,\ldots,N+1\}$ and $a\in A$.
Then, as the $f_i$ form a partition of unity for $C([0,1])$ by \eqref{eq:fiPOU}, 
\begin{align}
(\tau_{\Q_\omega} \otimes \tau_{M_{2N}})(\Psi(a)) =& \sum_{i=0}^{N+1} (\tau_{\Q_\omega} \otimes \tau_{M_{2N}}) \rho_i(f_i \otimes a) \notag \\
\notag
=& \frac12 \sum_{i=0}^{N+1} \tau_{\leb}(f_i)\tau_A(a) \\
=& \frac12 \tau_A(a),\quad a\in A,
\end{align}
establishing (\ref{Objective.2}).

We complete the proof by showing (\ref{Objective.1}), so fix $a,b\in \mathcal F_A$. By Property \ref{Propertyq} (\ref{Propertyq.4})  $q_iq_j=0$ when $|i-j|>1$, and thus the images of $\rho_i$ and $\rho_j$ are orthogonal for $|i-j|>1$. For notational simplicity, define both $\rho_{-1},\rho_{N+2}$ to be the zero map and recall the similar convention $f_{-1} = 0 = f_{N+2}$. Then we have
\begin{align}
\Psi(a)\Psi(b) &=  \sum_{i,j=0}^{N+1}\rho_i(f_i \otimes a)\rho_j(f_j \otimes b)\nonumber\\
&=  \sum_{k=-1}^1\sum_{i=0}^{N+1} \rho_i(f_i \otimes a)\rho_{i+k}(f_{i+k} \otimes b).
\end{align}
By Property \ref{Propertyq} (\ref{Propertyq.3}) and (\ref{DefRhoi}), we also have
\begin{equation}
\Psi(ab)=\sum_{k=-1}^1\sum_{i=0}^{N+1}\rho_i(f_i\otimes ab)\Theta^{\oplus 2N}(f_{i+k}).
\end{equation}
In this way, we can write
\begin{align}
\notag
&\hspace*{-2em}\Psi(a)\Psi(b)-\Psi(ab) \\
\notag
=&\sum_{k=-1}^1\sum_{\substack{i=0\\i\text{ odd}}}^{N+1} \rho_i(f_i \otimes a)\rho_{i+k}(f_{i+k} \otimes b) - \rho_i(f_i \otimes ab)\Theta^{\oplus 2N}(f_{i+k})\\
\label{eq:OSum2}&\quad+ \sum_{k=-1}^1\sum_{\substack{i=0\\i\text{ even}}}^{N+1} \rho_i(f_i \otimes a)\rho_{i+k}(f_{i+k} \otimes b) - \rho_i(f_i \otimes ab)\Theta^{\oplus 2N}(f_{i+k}).
\end{align}
The decomposition into odd and even $i$ in the sums in \eqref{eq:OSum2} gives the estimate
\begin{align}
\lefteqn{\|\Psi(a)\Psi(b)-\Psi(ab)\|}\nonumber\\
& \leq\ 2\sum_{k=-1}^1\max_i\|\rho_i(f_i \otimes a)\rho_{i+k}(f_{i+k} \otimes b) - \rho_i(f_i \otimes ab)\Theta^{\oplus 2N}(f_{i+k})\|.\label{eq:OGDecomp}
\end{align}
This follows as, for $k$ fixed, writing 
\begin{equation}
T_{i,k}:=\rho_i(f_i \otimes a)\rho_{i+k}(f_{i+k} \otimes b) - \rho_i(f_i \otimes ab)\Theta^{\oplus 2N}(f_{i+k}),
\end{equation}
we have $T_{i,k}=q_i^{\oplus 2N}T_{i,k}q_{i+k}^{\oplus 2N}$ (recall that $q_{i+k}$ acts as a unit on $\Theta(f_{i+k})$).  As $(q_i)_{i\text{ even}}$ are pairwise orthogonal projections and $(q_{i+k})_{i\text{ even}}$ are also pairwise orthogonal projections\footnote{Though not necessarily pairwise orthogonal to $(q_i)_{i\text{ even}}$.} one has $\|\sum_{i\text{ even}}T_{i,k}\|=\max\{\|T_{i,k}\|\mid i\text{ even}\}$. Likewise $\|\sum_{i\text{ odd}}T_{i,k}\|=\max\{\|T_{i,k}\|\mid i\text{ odd}\}$, giving \eqref{eq:OGDecomp}.

We now estimate the terms in \eqref{eq:OGDecomp}, starting with the case $k=0$.  For $i=1,\dots,N$ we have  $f_i\otimes a,f_i\otimes b,f_i\otimes ab\in \mathcal F_i$ by (\ref{eq:Falphai2}) and $f_i\in\mathcal F_i'$ by (\ref{eq:F'alphai2}), so that
\begin{align}
\rho_i(f_i\otimes a)\rho_i(f_i\otimes b) 
&\stackrel{\eqref{eq:rhoMult}\hphantom{_\eta}}{\approx_\eta} \rho_i(f_i^2 \otimes ab) \\
\notag
&\stackrel{\eqref{eq:rhoLin}\hphantom{_\eta}}{\approx_\eta} \rho_i(f_i \otimes ab)\Theta^{\oplus 2N}(f_i).
\end{align}
For $i=0$ or $i=N+1$, one has $\rho_i(f_i\otimes a)\rho_i(f_i\otimes b)=\rho_i(f_i\otimes ab)\Theta^{\oplus 2N}(f_i)$, 
as $\rho_0$ and $\rho_{N+1}$ are homomorphisms compatible with $\Theta^{\oplus 2N}$.
Thus the $k=0$ contribution to \eqref{eq:OGDecomp} is at most $4\eta$.

Now fix $k=\pm 1$, and $i=0,\dots,N+1$ such that $i+k\in\{0,\dots,N+1\}$.  Set 
\begin{equation}
\tilde{g}_i:=\begin{cases}\acute{g}_i,&k=-1;\\\grave{g}_i,&k=+1,\end{cases}\quad\text{and}\quad \tilde{g}_{i+k}:=\begin{cases}\grave{g}_{i+k},&k=-1;\\\acute{g}_{i+k},&k=+1.\end{cases}
\end{equation}
In this way, \eqref{eq:giUnitfi} (see Figure \ref{FigI}) gives 
\begin{equation}\label{e5.67}
\tilde{g}_if_i=f_i=f_i\tilde{g}_i\quad \text{and} \quad \tilde{g}_{i+k}f_{i+k}=f_{i+k}=f_{i+k}\tilde{g}_{i+k},
\end{equation}
so that as $\tilde{g}_i,\tilde{g}_{i+k}\in\mathcal F_i'$ and $\tilde{g}_i,\tilde{g}_{i+k}\in\mathcal F_{i+k}'$ (see \eqref{eq:F'alphai2}), we have\footnote{When one of $i$ or $i+k$ is $0$ or $N+1$, exact compatibility of $\rho_0$ and $\rho_{N+1}$ can be used in place of \eqref{eq:rhoLin}, leading to better estimates.}
\begin{align}
\lefteqn{ \rho_i(f_i \otimes a)\rho_{i+k}(f_{i+k} \otimes b)}\nonumber \\
&\stackrel{\eqref{e5.67}\;\;\;}{=_{\phantom{2\eta}}} \rho_i(f_i\tilde{g}_{i}\otimes a)\rho_{i+k}(\tilde{g}_{i+k}f_{i+k}\otimes b)\nonumber\\
&\stackrel{\eqref{eq:rhoLin}\;\;\;}{\approx_{2\eta}}
\rho_i(f_i \otimes a)\Theta^{\oplus 2N}(\tilde g_i \tilde g_{i+k})\rho_{i+k}(f_{i+k} \otimes b) \nonumber\\
&\stackrel{\eqref{eq:rhoLin}\;\;\;}{\approx_{2\eta}}
\rho_i(f_i\tilde g_{i+k} \otimes a)\rho_{i+k}(\tilde g_i f_{i+k} \otimes b).\label{eNew5.1}
\end{align}
Set
\begin{equation}
\tilde{\sigma}_i:=\begin{cases}\sigma_{i-1},&k=-1;\\\sigma_i,&k=+1.\end{cases}
\end{equation}
Then, using \eqref{eq:figiProd}, \eqref{eq:rhoiRestrict1} and \eqref{eq:rhoiRestrict2} for the first and last equalities below, we have
\begin{align}
\lefteqn{\rho_i(f_i\tilde g_{i+k} \otimes a)\rho_{i+k}(\tilde g_i f_{i+k} \otimes b)}\nonumber\\
&\stackrel{\hphantom{\text{Compatibility},\ \eqref{e5.67}}}{=}\big(\tilde{\sigma}_i(f_i\tilde{g}_{i+k}\otimes a)\oplus 0_N\big)\big(\tilde{\sigma}_i(\tilde{g}_if_{i+k}\otimes b)\oplus 0_N\big)\nonumber\\
&\stackrel{\text{Compatibility},\ \eqref{e5.67}}{=}\big(\tilde{\sigma}_{i}(f_{i}\tilde{g}_{i+k}\otimes ab)\oplus 0_N\big)\Theta^{\oplus 2N}(f_{i+k})\nonumber\\
&\stackrel{\hphantom{\text{Compatibility},\ \eqref{e5.67}}}{=}\rho_i(f_i\tilde{g}_{i+k}\otimes ab)\Theta^{\oplus 2N}(f_{i+k}).\label{eNew5.2}
\end{align}
Finally we use that $\tilde{g}_{i+k}\in\mathcal F_i'$ again\footnote{When $i=0$ or $i=N+1$, one has equality in the next estimate.} to see that
\begin{align}
\lefteqn{\rho_i(f_i\tilde{g}_{i+k}\otimes ab)\Theta^{\oplus 2N}(f_{i+k})}\nonumber\\
&\stackrel{\eqref{eq:rhoLin}\;}{\approx_{\eta}}\rho_i(f_i\otimes ab)\Theta^{\oplus 2N}(f_{i+k}\tilde{g}_{i+k})\nonumber\\
&\stackrel{\eqref{e5.67}\;}{=_{\hphantom{\eta}}}\rho_i(f_i\otimes ab)\Theta^{\oplus 2N}(f_{i+k}).\label{eNew5.3}
\end{align}

Combining \eqref{eNew5.1}, \eqref{eNew5.2} and \eqref{eNew5.3} gives
\begin{equation}
\rho_i(f_i\otimes a)\rho_{i+k}(f_{i+k}\otimes b)\approx_{5\eta}\rho_i(f_i\otimes ab)\Theta^{\oplus 2N}(f_{i+k}),
\end{equation}
so that the $k=\pm 1$ terms in \eqref{eq:OGDecomp} contribute at most $10\eta$ each. Putting these estimates together with the $k=0$ case, (\ref{eq:OGDecomp}) gives
\begin{equation} \|\Psi(ab)-\Psi(a)\Psi(b)\| \leq 24\eta \stackrel{(\ref{DefEta})}{<} \varepsilon,\quad a,b\in\mathcal F_A, \end{equation}
establishing (\ref{Objective.2}) and hence completing the proof of Theorem \ref{thm:MainThm}. \qed

\begin{remark}
It is worth noting that in the argument above we can replace $\Q_\omega$ by the ultrapower of the CAR algebra\footnote{Or any other ultraproduct  of UHF algebras, since these are admissible target algebras of finite type.} $M_{2^\infty}$ to produce  approximately multiplicative maps $\Psi:A\rightarrow (M_{2^\infty})_\omega\otimes M_{2N}$ which realise $\frac{1}{2}\tau_A$.   In the proof, we can also arrange for $N$ to be of the form $2^k$, so that $(M_{2^\infty})_\omega\otimes M_{2N}\cong (M_{2^\infty})_\omega$, whence reindexing will yield a $^*$-homomorphism $\Phi':A\rightarrow (M_{2^\infty})_\omega$ realising $\frac{1}{2}\tau_A$.  This will induce (as $\tau_A$ is faithful) a unital embedding of $A$ into $\Theta(1_A)(M_{2^\infty})_\omega \Theta(1_A)$, which is isomorphic to an ultraproduct of UHF algebras. Note, however,  that this does not necessarily give a unital embedding of $A$ into $(M_{2^\infty})_\omega$. Indeed, the unit of $A$ might be exactly $3$-divisible in $K_0$, whence no such embedding exists. 
\end{remark}

\begin{remark}
We need to assume that $\tau_A$ is faithful since -- unlike nuclearity -- the property of satisfying the UCT does not in general pass to quotients.\footnote{Every $\mathrm{C}^*$-algebra is the continuous image of its cone, which is homotopic to $0$, so satisfies the UCT, and there exist non-nuclear $\mathrm{C}^*$-algebras which do not satisfy the UCT \cite{S:KThy}.} However, we will see in Corollary \ref{CorBV2} below that in the situation of the theorem (when $A$ has a faithful trace to begin with), all traces, faithful or not, are quasidiagonal.
\end{remark}

In order to facilitate a generalisation by Gabe of Theorem \ref{thm:MainThm} to amenable traces on exact $\mathrm{C}^{*}$-algebras, we note that the above proof establishes the following statement.\footnote{Gabe's article was under preparation when this paper was submitted; it is now available as \cite{G:arXiv}.}
\begin{lemma}\label{lem:gen}
Let $A$ be a separable unital exact $\mathrm{C}^*$-algebra in the UCT class and let $\tau_A$ be a faithful trace on $A$. Suppose there are  $^*$-homomorphisms $\Theta:C([0,1])\rightarrow\Q_\omega$, $\grave{\Phi}:C_0([0,1),A)\rightarrow \Q_\omega$ and $\acute{\Phi}:C_0((0,1],A)\rightarrow\Q_\omega$ such that $\Theta$ is unital, $\grave{\Phi}$ and $\acute{\Phi}$ are nuclear and compatible with $\Theta$, 
\begin{equation}
\tau_{\Q_\omega}\circ \grave{\Phi}=\tau_{\leb} \otimes \tau_A, \quad \text{and} \quad \tau_{\Q_\omega}\circ \acute{\Phi}=\tau_{\leb} \otimes \tau_A.
\end{equation}
Then for any finite subset $\mathcal F_A$ of $A$ and $\varepsilon>0$, there exist $N\in\mathbb N$ and a nuclear c.p.\ map $\Psi:A\rightarrow \Q_\omega\otimes M_{2N}$ such that \eqref{Objective.1} and \eqref{Objective.2} hold.
\end{lemma}
To see this, one works through the proof of Theorem \ref{thm:MainThm}, checking that the hypothesis of Lemma \ref{lem:gen} enables all maps  in the proof be taken to be nuclear.  Property \ref{Def-G} is obtained from the patching lemma (Lemma \ref{lem:Patching}). By construction the $\rho$ given by Lemma \ref{lem:Patching} is defined explicitly in terms of $\nu_0$, $\nu_1$ and a unitary $u$ in \eqref{Patch.DefRho1}, \eqref{Patch.DefRho2} and \eqref{Patch.E20}, and so is nuclear when $\nu_0$ and $\nu_1$ are.   The controlled stable uniqueness theorem (Theorem \ref{thm:StableUniqueness}) is valid for separable unital exact domains in the UCT class and nuclear $^*$-homomorphisms, so Property \ref{Def-N} holds with $\phi$ and $\psi$ additionally assumed nuclear.  With these adjustments to Properties \ref{Def-G} and \ref{Def-N}, the maps $\rho_0,\dots,\rho_{N+1}$ produced in Claim \ref{MainClaim} are all nuclear. The definition of $\Psi$ in terms of these $\rho_i$ from \eqref{e5.57} ensures that it too is nuclear.

\section{Consequences: Structure and classification}\label{Structure}

\noindent
In this section we shall explain how Corollaries \ref{cor:QD}-\ref{monotracial-classification} follow from Theorem~\ref{thm:MainThm}, and put into context the main result and its consequences for the structure and classification of simple nuclear $\mathrm C^*$-algebras.

\bigskip

We begin with Corollary \ref{cor:QD}. The version below includes an additional statement pointed out to us by Nate Brown; it implies that if there is a faithful (hence quasidiagonal) trace to begin with, then in fact all  traces are quasidiagonal. 

\begin{cor}\label{CorBV2}
Every separable, nuclear $\mathrm{C}^{*}$-algebra in the UCT class with a faithful trace is quasidiagonal. In particular, each simple, separable, stably finite, nuclear $\mathrm{C}^*$-algebra satisfying the UCT is quasidiagonal. 

Further if $A$ is a separable, unital, nuclear and quasidiagonal $\mathrm{C}^*$-algebra satisfying the UCT, then all traces on $A$ are quasidiagonal.
\end{cor}
\begin{proof}
The first sentence is an immediate consequence of Theorem \ref{thm:MainThm} and the fact that a $\mathrm{C}^*$-algebra with a faithful quasidiagonal trace is quasidiagonal (see Proposition \ref{Prop.EasyQDT} and Remark \ref{QDNonUnital} (\ref{QDNonUnital3})). 

For the second sentence, let $A$ be a separable, simple, nuclear, stably finite $\mathrm C^*$-algebra that satisfies the UCT.
It suffices to show that there is a $\mathrm C^*$-algebra stably isomorphic to $A$ which has a trace (as such a trace is automatically faithful by simplicity). If $A\otimes \mathcal K$ contains a nonzero projection $p$, then the hereditary $\mathrm C^*$-subalgebra $p(A\otimes\mathcal K)p$ is stably isomorphic to $A$ by L.~Brown's Theorem (\cite{B:PJM}), and has a trace by \cite{BlaHan:dimtraces} and \cite{Haa:quasitraces}. On the other hand, if $A$ is stably projectionless then by \cite[Corollary 2.2]{T:MA}, $A$ is stably isomorphic to an algebraically simple and stably finite $\mathrm C^*$-algebra $B$.
By \cite{K:FI} and \cite[Corollary 2.5]{T:MA}, $B$ has a faithful trace, as required.

For the last claim, let $\mathcal{N}$ be the class of separable nuclear $\mathrm{C}^*$-algebras satisfying the UCT.  Then $\mathcal{N}$ contains $\mathbb{C}$ and is closed under: countable inductive limits with injective connecting maps, tensoring by finite dimensional matrix algebras, and extensions\footnote{If $0\rightarrow I\rightarrow E\rightarrow B\rightarrow 0$ is a short exact sequence with $I,B\in \mathcal{N}$, then $E\in \mathcal{N}$.} (for nuclearity this is a standard exercise, while for the UCT, this is a consequence of the characterisation in terms of the bootstrap class; see \cite[Definition 22.3.4]{B:KThy}).  If $A\in \mathcal{N}$ is simple, then Theorem \ref{thm:MainThm} shows that all traces on $A$ are quasidiagonal. Thus \cite[Lemma 6.1.20 (3) $\Rightarrow$ (1)]{B:MAMS} shows\footnote{The definition of a Popa algebra from \cite[Definition 1.2]{B:MAMS} is not required here, all that matters is that such algebras are  simple.} that every quasidiagonal $\mathrm{C}^*$-algebra in $\mathcal{N}$ has the property that all of its amenable traces are quasidiagonal.  \end{proof}

\bigskip

When $G$ is countable, Corollary~\ref{cor:rosenberg} follows from Theorem~\ref{thm:MainThm}, since $\mathrm{C}^*_{\mathrm{r}}(G)$ satisfies the UCT by a result of Tu \cite[Lemma 3.5 and Proposition 10.7]{Tu:KT}, and the canonical trace is well-known to be faithful. The general case holds since any discrete amenable group can be exhausted by an increasing net of countable amenable subgroups, and the same goes for the group $\mathrm{C}^{*}$-algebras; moreover, by Arveson's extension theorem, Voiculescu's approximation form of quasidiagonality is a local property and therefore follows from the separable case (see also the appendix of \cite{Bro:survey}).

\bigskip

The theorem below summarises a number of known facts (many of which are quite deep) and combines them with our main result.\footnote{In the case of at most one trace, Theorem~\ref{main-summary} can be entirely derived from Theorem~\ref{thm:MainThm} in conjunction with published results, whereas some of the statements in the general form also employ arXiv preprints still under review.} Corollary~\ref{cor:Classification} is obviously contained in Theorem~\ref{main-summary} \eqref{Summary.4}. Although Corollary~\ref{monotracial-classification} can be read off from Theorem~\ref{main-summary} \eqref{Summary.4} in conjunction with \cite{SWW:Invent}, it is worth mentioning explicitly how it follows from published results together with Theorem~\ref{thm:MainThm}: In the traceless case, Corollary~\ref{monotracial-classification} is Kirchberg--Phillips classification (cf.\ \cite[Theorem~8.4.1]{R:Book}) and doesn't use Theorem~\ref{thm:MainThm}. The monotracial situation follows from \cite[Theorem~5.4]{LN:Adv} (which generalises \cite{Win:localizing}); since Theorem \ref{thm:MainThm} makes quasidiagonality automatic, the tracial rank hypothesis required in \cite{LN:Adv} follows from \cite[Theorem~6.1]{MS:DMJ}. In both cases the Elliott invariant reduces to ordered $K$-theory.

\begin{thm}
\label{main-summary}
Let $A$ be a separable, simple, unital $\mathrm{C}^{*}$-algebra with finite nuclear dimension. Suppose in addition that $A$ satisfies the UCT.\footnote{The UCT is not needed for some parts of the theorem. We make it a global assumption for the sake of brevity, but will flag up in the proof where it is actually used.} 
Then: 
\begin{enumerate}[(i)]
\item\label{Summary.1} $A$ has nuclear dimension at most $2$; if $A$ has no traces then it has nuclear dimension $1$, and if $A$ has traces it has decomposition rank at most $2$. If $T(A)$ is nonempty and has compact extreme boundary, then $A$ in fact has decomposition rank $1$ or $0$, and the latter happens if and only if $A$ is AF. 
\item $A$ is purely infinite if and only if $A$ has no traces if and only if $A$ has infinite decomposition rank.  \label{Summary.2}

\noindent In this case, $A$ is an inductive limit of direct sums of $\mathrm C^*$-algebras of the form $\mathcal{O}_{n} \otimes M_{k} \otimes C(S^{1})$, where $\mathcal{O}_{n}$ denotes a Cuntz algebra (including $n=\infty$).
\item\label{Summary.3} $A$ is stably finite if and only if $A$ has a trace if and only if the decomposition rank of $A$ is finite. 

\noindent In this case, $A$ is an inductive limit of subhomogeneous $\mathrm{C}^{*}$-algebras of topological dimension at most $2$.
\item\label{Summary.4}
If $B$ is another such $\mathrm{C}^{*}$-algebra (and both $A$ and $B$ are infinite dimensional), then any isomorphism between the Elliott invariants of $A$ and $B$ lifts to a $^{*}$-isomorphism between the $\mathrm C^*$-algebras. 
\end{enumerate}
\end{thm}

\begin{proof}
Let us begin with (\ref{Summary.3}). It is a well-known consequence of \cite{BlaHan:dimtraces} and \cite{Haa:quasitraces}  that if $A$ is stably finite then it has a trace;\footnote{As $A$ has finite nuclear dimension, the shorter argument of \cite{BW:CRMASSRC} can be used in place of \cite{Haa:quasitraces}.} the converse is trivial (and no use of the UCT is required for either direction).  

If $A$ has traces, these are all quasidiagonal by Theorem~\ref{thm:MainThm} and $A$ is classifiable by \cite[Theorem~4.3]{EGLN:arXiv} (see also \cite[Corollary~4.4 and Theorem~4.5]{EGLN:arXiv}),  via the main result of \cite{GLN:arXiv}.
Now $A$ is isomorphic to one of the models constructed in \cite{E:CMS} (since those exhaust the invariant), which in turn have topological dimension at most $2$, as pointed out in \cite[1.11]{W:PLMS}; the UCT is of course necessary for this argument. 

By \cite{W:PLMS} these models have decomposition rank at most $2$, and therefore so does the limit $A$ by \cite[(3.2)]{KirWin:dr}. Conversely, finite decomposition rank implies  quasidiagonality by \cite[Theorem~4.4]{KirWin:dr} and hence the existence of a trace by \cite[2.4]{V:IEOT} (neither of these implications require the UCT). This completes the proof of (\ref{Summary.3}).

We next turn to (\ref{Summary.2}). If $A$ is purely infinite then it is clearly traceless; the converse follows from the dichotomy of \cite[Theorem 5.4]{WinZac:dimnuc} which shows that $A$ is either stably finite or purely infinite, the former being impossible by (\ref{Summary.3}) (neither direction requires the UCT). We already know from (\ref{Summary.3}) that $A$ is traceless if and only if it has infinite decomposition rank; the reverse direction uses the UCT. The statement about inductive limits is \cite[Corollary~8.4.11]{R:Book} and again requires the UCT.

For (\ref{Summary.1}), the statement about AF algebras is \cite[Remark~2.2 (iii)]{WinZac:dimnuc}; the UCT is not involved. In the traceless case $A$ is a Kirchberg algebra by (\ref{Summary.2})
and we have nuclear dimension $1$ by \cite[Theorem~G]{BBSTWW:arXiv}, still without assuming UCT. If $A$ has a trace, we have already seen in (\ref{Summary.3}) that it has decomposition rank at most $2$; the UCT is heavily involved. Under the extra trace space condition, the decomposition rank is at most $1$ by \cite[Theorem~F]{BBSTWW:arXiv}, and the UCT still enters through Theorem~\ref{thm:MainThm}.  

We have already mentioned that the classification statement of (\ref{Summary.4}) in the absence of traces is Kirchberg--Phillips classification (cf.\ \cite[Theorem~8.4.1]{R:Book}). When there are traces, these are all quasidiagonal by Theorem~\ref{thm:MainThm} and the result follows as in (\ref{Summary.1}) from \cite[Theorem~4.3]{EGLN:arXiv}. All of these require the UCT. 
\end{proof}

Theorem~\ref{main-summary} mostly focuses on the classification of simple nuclear $\mathrm{C}^{*}$-algebras. For completeness we briefly mention the Toms--Winter conjecture which  emphasises their structure and predicts that finite noncommutative topological dimension, $\mathcal{Z}$-stability, and strict comparison occur simultaneously. The conjecture makes sense for both nuclear dimension and decomposition rank acting as topological dimension; in the latter case it subsumes the Blackadar--Kirchberg problem (in the simple, unital case), whereas the nuclear dimension version encapsulates both the stably finite and the purely infinite situation. Combining the two settings leads to the full Toms--Winter conjecture for simple unital $\mathrm{C}^*$-algebras (see \cite[Remark 3.5]{TW:JFA} and \cite[Conjecture 9.3]{WinZac:dimnuc}).

\begin{conjecture}
\label{fullTW}
Let $A$ be a separable, simple, unital, nuclear and infinite dimensional $\mathrm{C}^{*}$-algebra. Then, the following are equivalent.
\begin{enumerate}[(i)]
\item $A$ has finite nuclear dimension.
\item $A$ is $\mathcal{Z}$-stable.
\item $A$ has strict comparison.
\end{enumerate}
If $A$ is stably finite, (i) may be replaced by
\begin{enumerate}[(i)]
\item[(i')] $A$ has finite decomposition rank.
\end{enumerate}
\end{conjecture}

The implications (i') $\Rightarrow$ (i) $\Rightarrow$ (ii) $\Rightarrow$ (iii) hold in full generality (the first is trivial, the second is \cite[Corollary~6.3]{W:Invent2}, and the third is \cite[Theorem~4.5]{R:IJM}). When the extreme boundary of $T(A)$ is compact and finite dimensional,\footnote{Certain non-compact but still finite dimensional boundaries are handled in \cite{Z:JFA}.} we have (iii) $\Rightarrow$ (ii)  by \cite{KR:Crelle, S:Preprint2, TWW:IMRN} (building on the unique trace case of \cite{MS:Acta}) and (ii)  $\Rightarrow$ (i) is \cite[Theorem~B]{BBSTWW:arXiv} when $A$ has compact tracial extreme boundary; both conditions are of course satisfied in the monotracial case. All of these implications are independent of classification (or the UCT, for that matter). Theorem~\ref{main-summary}  yields (i) $\Rightarrow$ (i') for stably finite $\mathrm{C}^{*}$-algebras when they in addition satisfy the UCT. This is the first abstract result of this kind (although quasidiagonality was identified as the differentiating feature between nuclear dimension and decomposition rank in \cite{MS:DMJ}, particularly in light of \cite{SWW:Invent}). The UCT enters through our main theorem and also through the full-blown classification result from \cite{EGLN:arXiv}. In the case of compact tracial extreme boundary, (i) $\Rightarrow$ (i') follows from Theorem~\ref{thm:MainThm} together with \cite[Theorem~F]{BBSTWW:arXiv}; when $A$ is monotracial one can combine Theorem~\ref{thm:MainThm} with \cite[Corollary~1.2]{MS:DMJ} (still using the UCT and stable uniqueness, but not the full-fledged classification of simple $\mathrm C^*$-algebras). Upon taking the intersection of these conditions, we see that Conjecture~\ref{fullTW} holds when $A$ has compact (possibly empty) tracial extreme boundary with finite covering dimension. As a structural counterpart of Corollary~\ref{monotracial-classification} we highlight the special case of at most one trace (which again only involves our main theorem in conjunction with published results):

\begin{cor}
The full Toms--Winter conjecture (Conjecture \ref{fullTW}) holds for $\mathrm{C}^{*}$-algebras with at most one trace and which satisfy the UCT.
\end{cor}

We close with some more concrete applications of classification methods. Recall that a $\mathrm{C}^{*}$-algebra is AF-embeddable if it is isomorphic to a $\mathrm{C}^{*}$-subalgebra of an AF algebra. One can refine this notion by asking for embeddings into simple AF algebras, or by prescribing a trace and asking this to be picked up by the embedding composed with a (perhaps even unique) trace on the AF algebra. AF-embeddability clearly implies quasidiagonality, and conversely many quasidiagonal algebras are known to be AF-embeddable, a standout result being Ozawa's homotopy invariance of AF-embeddability \cite{O:GAFA} (see \cite[Chapter 8]{BrOz} for an overview and further results). As with Theorem \ref{main-summary}, Theorem \ref{thm:MainThm} can be used to remove quasidiagonality assumptions from applications of classification to AF-embeddability; we highlight the monotracial case.

\begin{cor}\label{AF-embeddable} 
Let $A$ be a separable, simple, unital, monotracial and nuclear $\mathrm{C}^{*}$-algebra which satisfies the UCT.  Then $A$ embeds unitally into a simple, monotracial AF algebra.
\end{cor}
\begin{proof}
Upon tensoring by the universal UHF algebra $\Q$, we may assume that $A$ is $\mathcal{Q}$-stable. Then Theorem \ref{thm:MainThm} provides the quasidiagonality required to use \cite[Theorem 6.1]{MS:DMJ} to see that $A$ is TAF. The ordered $K_{0}$-group of $A$ has Riesz interpolation and is simple and weakly unperforated (see \cite[p.\ 59]{R:Book}); since $K_{0}(A)$ is torsion-free ($A$ absorbs $\mathcal{Q}$), it is in fact unperforated (cf.\ \cite[Definitions~1.4.3 and 3.3.2]{R:Book}) and hence a dimension group by \cite{EHS:AJM}. Now by \cite[Proposition~1.4.2 and Corollary~1.5.4]{R:Book} it is the ordered $K_{0}$-group of a unital, simple AF algebra $B$, which will automatically  be $\mathcal{Q}$-stable (by Elliott's AF classification \cite{E:JA}) since $K_{0}(B) \cong K_0(B \otimes \mathcal Q)$. Upon sending $K_{1}(A)$ to $\{0\} = K_{1}(B)$, we obtain a morphism between the ordered $K$-groups of the TAF algebras $A$ and $B$, which lifts to a unital $^{*}$-homomorphism from $A$ to $B$ by \cite{Dad:IJM} (note again that by $\mathcal{Q}$-stability there is no torsion and so total $K$-theory reduces to just $K$-theory). Since $A$ is simple, this is automatically an embedding. Finally, since $A$ is monotracial there is only one state on $K_{0}(A)$ (see \cite[Theorem~1.1.11]{R:Book}), hence also on $K_{0}(B)$, which in turn implies (by \cite[Proposition~1.1.12]{R:Book}) that $B$ is monotracial as well. Since the embedding of $A$ into $B$ is unital, it necessarily preserves the trace.
\end{proof}

In particular one can apply the previous corollary to enhance Corollary~\ref{cor:rosenberg}, extending the corresponding result for elementary amenable, countable, discrete groups from \cite{ORS:GAFA}.

\begin{cor}
\label{AF-embeddablegps} 
If $G$ is a countable, discrete, amenable group, then $\mathrm{C}^{*}_{\mathrm{r}}(G)$ embeds unitally into a unital, simple, monotracial AF algebra in such a way that the canonical trace on $\mathrm{C}^{*}_{\mathrm{r}}(G)$ is realised. 
\end{cor}

\begin{proof}
By \cite[Proposition~2.1]{ORS:GAFA}, $\mathrm{C}^{*}_{\mathrm{r}}(G)$ embeds into a separable, unital, simple, nuclear, monotracial, UCT $\mathrm{C}^{*}$-algebra $B(G)$ in a fashion which picks up the canonical trace on $\mathrm{C}^{*}_{\mathrm{r}}(G)$.\footnote{$B(G)$ is the Bernoulli shift crossed product $\bigotimes_G(M_{2^\infty})\rtimes G$, from which the statement about the preservation of the canonical trace follows.}  The result then follows from Corollary \ref{AF-embeddable}.
\end{proof}

Strongly self-absorbing $\mathrm{C}^{*}$-algebras, whose study from an abstract perspective was initiated in \cite{TW:TAMS}, are omnipresent in the structure and classification theory of nuclear $\mathrm{C}^{*}$-algebras. These are simple, nuclear, with at most one trace (results that go back to \cite{EffRos:PJM}) and are $\mathcal Z$-stable by \cite{W:JNCG}. The only known examples are the Jiang--Su algebra $\mathcal Z$, UHF algebras of infinite type, the Cuntz algebras $\mathcal{O}_{\infty}$ and $\mathcal{O}_{2}$, and tensor products of $\mathcal{O}_{\infty}$ with UHF algebras of infinite type.
We are very interested in whether there are any other examples (cf.\ \cite[Question~5.11]{TW:TAMS}), mostly because strongly self-absorbing $\mathrm{C}^{*}$-algebras may be thought of as a microcosm of simple nuclear $\mathrm{C}^{*}$-algebras, and they witness important structural questions such as the UCT problem and the Blackadar--Kirchberg problem. When assuming the UCT we now have an answer, since in this case the possible $K$-groups were computed in \cite{TW:TAMS} and we may apply Corollary~\ref{monotracial-classification}:

\begin{cor}
The strongly self-absorbing $\mathrm{C}^{*}$-algebras in the UCT class are precisely the known ones. 
\end{cor}

\end{document}